\documentclass[12pt,a4paper,reqno]{amsart}

\usepackage{xcolor}
\usepackage[all,2cell,arrow,color]{xy}
\usepackage{amsmath}
\usepackage[english]{babel}
\usepackage{amssymb}
\usepackage{amsfonts}
\usepackage{mathtools}
\usepackage{stmaryrd}
\usepackage{tensor}
\numberwithin{equation}{section}
\usepackage{mathptmx}
\usepackage{graphicx} 
\usepackage{rotating}
\usepackage{pdflscape}
\usepackage[pagebackref,colorlinks,urlcolor=darkgreen,citecolor=darkgreen,linkcolor=darkgreen]{hyperref}
\usepackage{enumerate,xspace}
\usepackage{tikz}
\usetikzlibrary{intersections}
\tikzstyle{braid}=[thick]
\tikzstyle{twocell}=[->,double, double equal sign distance,thick]
\tikzset{arr/.style={circle,draw,inner sep=0}}
\tikzset{empty/.style={inner sep=0pt, minimum size=0pt}}

\UseAllTwocells
\CompileMatrices

\definecolor{darkgreen}{rgb}{0,0.45,0} 
\def\red{\color{red}}
 \def\blue{\color{blue}}
 \def\black{\color{black}}
 \definecolor{lightgrey}{rgb}{0.666666,0.666666,0.666666}
 \def\grey{\color{lightgrey}}
\def\green{\color{green}}

\newcommand{\sfc}{\ensuremath{\mathsf C}\xspace}
\newcommand{\rev}{\ensuremath{{}^{\mathsf{rev}}}\xspace}
\newcommand{\Hilb}{\ensuremath{\mathsf{Hilb}}\xspace}
\newcommand{\Vect}{\ensuremath{\mathsf{Vect}}\xspace}

\newcommand{\mM}{\ensuremath{\mathbb M}}
\newcommand{\nto}{\nrightarrow}

  \newtheorem{proposition}{Proposition}[section]
  \newtheorem{lemma}[proposition]{Lemma}
  \newtheorem{corollary}[proposition]{Corollary}
  \newtheorem{theorem}[proposition]{Theorem}

  \theoremstyle{definition}
  \newtheorem{definition}[proposition]{Definition}
  \newtheorem{example}[proposition]{Example}

\theoremstyle{remark}
  \newtheorem{remark}[proposition]{Remark}
 \newcommand{\Section}{\setcounter{definition}{0}\section}
  \newcounter{c}
  \renewcommand{\[}{\setcounter{c}{1}$$}
  \newcommand{\etyk}[1]{\vspace{-7.4mm}$$\begin{equation}\Label{#1}
  \addtocounter{c}{1}}
  \renewcommand{\]}{\ifnum \value{c}=1 $$\else \end{equation}\fi}
\renewcommand{\theequation}{\thesection.\arabic{equation}}

\makeatletter
\newcommand*{\inlineequation}[2][]{%
  \begingroup
    \refstepcounter{equation}%
    \ifx\\#1\\%
    \else
      \label{#1}%
    \fi
    \relpenalty=10000 %
    \binoppenalty=10000 %
    \ensuremath{%
      #2%
    }%
    ~\@eqnnum
  \endgroup
}
\makeatother

\begin{document}

\title[Weak multiplier bimonoids]{Weak multiplier bimonoids}

\author{Gabriella B\"ohm} 
\address{Wigner Research Centre for Physics, H-1525 Budapest 114,
P.O.B.\ 49, Hungary}
\email{bohm.gabriella@wigner.mta.hu}
\author{Jos\'e G\'omez-Torrecillas}
\address{Departamento de \'Algebra and CITIC, Universidad de Granada, 
E-18071 Granada, Spain} 
\email{gomezj@ugr.es}
\author{Stephen Lack}
\address{Department of Mathematics, Macquarie University NSW 2109, Australia}
\email{steve.lack@mq.edu.au}

\date{March 2016}

\begin{abstract}
Based on the novel notion of `weakly counital fusion morphism', regular weak
multiplier bimonoids in braided monoidal categories are introduced.    
They generalize weak multiplier bialgebras over fields
\cite{BohmGomezTorrecillasLopezCentella:wmba} and multiplier bimonoids in 
braided monoidal categories \cite{BohmLack:braided_mba}. Under some
assumptions the so-called base object of a regular weak multiplier bimonoid is
shown to carry a coseparable comonoid structure; hence to possess a monoidal
category of bicomodules. In this case, appropriately defined modules over a
regular weak multiplier bimonoid are proven to constitute a monoidal category
with a strict monoidal forgetful type functor to the category of bicomodules
over the base object.  

Braided monoidal categories considered include various categories of modules
or graded modules, the category of complete bornological spaces, and the
category of complex Hilbert spaces and continuous linear transformations. 
\end{abstract}
  
\maketitle

\section{Introduction} \label{sec:intro}

{\em Hopf algebras} can be used to describe symmetries in various
situations. Classically, they are vector spaces equipped with the additional
structures of a compatible algebra and a coalgebra, and they
have categories of representations (modules or comodules) with a
monoidal structure, strictly preserved by the forgetful functor to
vector spaces. There are Tannaka-style results which
allow a Hopf algebra to be reconstructed from its monoidal category of
representations together with the forgetful functor. 

Various applications lead one to consider more general monoidal
categories, which none\-the\-less share many features with those of the
previous paragraph. One can then ask whether they might 
be the categories of representations of some object more general than 
a Hopf algebra in the classical sense. It turns out that this is 
often the case. This is perhaps reminiscent of non-commutative 
geometry, where non-commutative algebras can sometimes be seen as 
algebras of functions on some (hypothetical) ``non-commutative 
spaces''. Or for another example, the category of sheaves on a 
topological space is a topos; but any (Grothendieck) topos can be 
seen as a generalized space (or as the category of sheaves on a 
generalized space). 

Returning to Hopf algebras,
one direction of generalization is where the representations involve
an underlying object more general, or simply different, than a vector
space. Examples might include modules over commutative rings,  graded 
vector spaces,  Hilbert spaces, or  bornological vector spaces 
\cite{Hogbe-Nlend,Meyer}. A unified treatment of all these situations 
is possible using the notion of Hopf monoid in a braided monoidal category. 

There is another direction of generalization, in which the base
objects remain vector spaces, but one generalizes to structures which
are  not Hopf algebras.  
For example, functions on a finite group, with 
values in a field, constitute a Hopf algebra. But if the group is no 
longer finite, then the functions of finite support form 
neither an algebra (there is no unit for the pointwise 
multiplication) nor a coalgebra (the comultiplication which
is dual to the group multiplication does not land in the tensor square of the
vector space of finitely supported functions). To axiomatize this situation,
the notion of {\em multiplier Hopf algebra} over a field was proposed by Van
Daele in \cite{VanDaele:multiplier_Hopf}. 

Yet another generalization of Hopf algebras involves aspects of
both of these types of generalization, and proved important in the
study of fusion categories \cite{Etingof/Nikshych/Ostrik:2005}. The
basic examples of fusion categories are categories of
representations of (semi-simple) Hopf algebras. In a fusion category, every
irreducible object has an associated `dimension', and in the Hopf
algebra case this dimension is always an integer, but in general this
need not be the case. It turns out that these non-integral cases can
be seen as categories of representations of {\em weak Hopf
algebras} \cite{WHAI,Nill}. A weak Hopf algebra is both an algebra and a
coalgebra, but the compatibility conditions between these structures are
weaker than in an ordinary Hopf algebra, reflecting the fact that the
forgetful  functor from the category of
representations to vector spaces is no longer strict
monoidal. But there is a strict monoidal forgetful functor from
the category of representations of the weak Hopf algebra to the
category of bimodules over a certain separable Frobenius algebra, 
determined by the weak Hopf algebra and called the {\em base object}.
Note, however, that the monoidal category of bimodules over
the base object is not braided, so that this does not reduce to the
earlier generalization of Hopf monoids in a braided monoidal
category. 

A common generalization of multiplier Hopf algebras and weak Hopf 
algebras was proposed by Van Daele and Wang in 
\cite{VDaWa,VDaWa:Banach} under the name {\em weak multiplier Hopf 
algebra}. 

The axioms of a Hopf algebra, and all of the generalizations listed 
above, include the existence of a so-called antipode. Omitting this 
requirement one obtains the more general notion of bialgebra and its
various generalizations. The category of representations of a
bialgebra is still monoidal, with a strict monoidal forgetful functor
to the base category; what is lost in the absence of an antipode is
the ability to lift the closed structure of the base monoidal 
category to the category of representations.

{\em Weak multiplier bialgebras} and {\em multiplier bialgebras} over 
vector spaces were defined and analyzed in 
\cite{BohmGomezTorrecillasLopezCentella:wmba}. Their representations 
are certain non-degenerate modules, and once again the category of 
representations has a monoidal structure, not preserved by the 
forgetful functor to vector spaces. But as in the case of weak Hopf 
algebras, there is still a strict monoidal forgetful functor to an 
intermediate monoidal category. This time, rather than the monoidal 
category of bimodules over a separable Frobenius algebra, it is the 
monoidal category of bicomodules over a coseparable coalgebra 
constructed from the weak multiplier bialgebra, once again called the 
base object.   

We have recently begun the large program of studying all of these 
generalizations of Hopf algebras, along with their categories of 
representations, not just over vector spaces but in more general 
braided monoidal categories.  
In \cite{BohmLack:braided_mba} we defined multiplier bimonoids in
any braided monoidal category. Under further assumptions, we
constructed a monoidal category of representations of a multiplier
bimonoid. We further developed the theory of multiplier bimonoids in
two subsequent papers \cite{BohmLack:cat_of_mbm,BohmLack:simplicial}. 
Then in \cite{BohmLack:multiplier_Hopf_monoid} we turned to 
multiplier Hopf monoids; that is, multiplier bimonoids with a 
suitable antipode map.  
The present work can be seen as the next step of this program where we
generalize the regular weak multiplier bialgebras of
\cite{BohmGomezTorrecillasLopezCentella:wmba} to more general braided 
monoidal categories. 

The outline of the paper is as follows. In Section \ref{sec:axioms} 
we define the central notion of the paper, that of  {\em regular
weak multiplier bimonoid}; this in turn depends on  the notion of {\em
weakly counital fusion morphism}, also defined in Section \ref{sec:axioms}. We
also discuss various duality principles for these structures, arising 
from symmetry properties of the axioms. In  Section \ref{sec:base} we 
study the base objects, and under appropriate assumptions we show 
that they admit coseparable comonoid structure. While in the category
\Vect of vector spaces this was done
\cite{BohmGomezTorrecillasLopezCentella:wmba} under the assumption that 
the comultiplication is full, here we use a different assumption, 
which follows from fullness of the comultiplication in the case of 
\Vect. 
In Section \ref{sec:modules} we define and study the category of 
modules over a regular weak multiplier bimonoid. Under favorable 
conditions we prove that it is monoidal, via a monoidal structure 
lifted from the category of bicomodules over the base object. Again, 
in contrast to \cite{BohmGomezTorrecillasLopezCentella:wmba}, this is 
done assuming not fullness but a substitute which follows from it in 
the category of vector spaces.   
We do not address here the analogous question about comodules over a
regular weak multiplier bimonoid. We also do not discuss the
notion of antipode on a weak multiplier bimonoid (and its bearing on the
structure of the category of modules); that is, we do not study
weak multiplier Hopf monoids. Section \ref{sec:closed} is devoted to the
study of the particular case when the braided monoidal base category 
is also {\em closed}, as is the case in most of our examples of
interest, but not in the case of Hilbert spaces. We discuss 
consequences of closedness on the assumptions and constructions of 
the previous sections. 

{\bf Acknowledgement.} We thank Ralf Meyer and Christian Voigt for highly
enlightening discussions about bornological vector spaces. We gratefully
acknowledge the financial support of the Hungarian Scientific Research Fund
OTKA (grant K108384), of `Ministerio de Econom\' {\i}a y 
Competitividad' and `Fondo Europeo de Desarrollo Regional FEDER' 
(grant MTM2013-41992-P), as well as the Australian
Research Council Discovery Grant (DP130101969) and an ARC Future Fellowship
(FT110100385). GB expresses thanks for the kind invitations and the warm
hospitality that she experienced visiting the University of Granada in Nov
2014, Feb 2015 and Oct 2015. SL is grateful for the warm hospitality of his
hosts during visits to the Wigner Research Centre in Sept-Oct 2014 and
Aug-Sept 2015. 
  

\section{The axioms}\label{sec:axioms}

The subject of this section is the introduction of the central notion of the
paper: {\em regular weak multiplier bimonoid}.

Throughout, we work in a braided monoidal category  $\mathsf C$. We 
do not assume that its monoidal structure is strict but --- relying on
coherence --- we omit explicit mention of the associativity and unit 
isomorphisms. The composite of morphisms $f\colon A\to B$ and 
$g\colon B\to C$ will be denoted by $g.f\colon A\to C$. Any 
identity morphism will be denoted by $1$. The monoidal 
product will be denoted by juxtaposition, the monoidal unit by $I$, 
and the braiding by $c$. For $n$ copies of the same object $A$, we 
also use the power notation $AA\dots A=A^n$. The same category 
$\mathsf C$ with the reversed monoidal product but the same braiding 
will be denoted by $\mathsf C^{\mathsf{rev}}$. The same monoidal 
category with the inverse braiding will be denoted by 
$\overline\sfc$. Performing both these ``dualities'' (in either 
order) gives a braided monoidal category $\overline\sfc\rev$.
 
By a {\em semigroup} in a braided monoidal category ${\mathsf
C}$ we mean a pair $(A,m)$ consisting of an object $A$ of ${\mathsf C}$ and
a morphism $m\colon A^2\to A$ --- called the {\em multiplication} --- which
obeys the associativity condition $m.m1=m.1m$. The existence of a unit is not
required.

Recall from \cite{BohmLack:cat_of_mbm} the notion of {\em \mM-morphism} 
$X\nto A$ for a semigroup $A$ and an arbitrary object $X$. This means a pair 
$\xymatrix@C=20pt{XA \ar[r]|(.5){\,f_1\,}& A & \ar[l]|(.5){\,f_2\,} AX}$ of
morphisms in $\mathsf C$ making
commutative \begin{equation} \label{eq:components_compatibility} 
\xymatrix{
AXA \ar[r]^-{1f_{1}} \ar[d]_-{f_{2}1} &
A^{2} \ar[d]^-m \\
A^2 \ar[r]_-m &
A.}
\end{equation}
As will be explained in Section \ref{sec:closed}, \mM-morphisms $X \nto A$
are related to morphisms from $X$ to the multiplier monoid of $A$ whenever the
latter is available. 

We say that a morphism $v\colon ZV \to W$ in $\mathsf C$ is {\em
non-degenerate on the left with respect to some class $\mathcal Y$ of objects
in $\mathsf C$} if the map 
$$
{\mathsf C}(X,VY) \to {\mathsf C}(ZX,WY),\qquad
g \mapsto 
\xymatrix{ 
ZX \ar[r]^-{1g} &
ZVY \ar[r]^-{v1} &
WY}
$$
is injective for any object $X$, and any object $Y$ in $\mathcal
Y$. Symmetrically, $v$ is said to be {\em non-degenerate on the right with
respect to the class $\mathcal Y$} if $v.c$ is non-degenerate on the left with
respect to $\mathcal Y$. The multiplication $\xymatrix@C=18pt{A^2
\ar[r]|(.5){\,m\,}& A}$ of a semigroup $A$ is termed {\em non-degenerate with
respect to $\mathcal Y$} if it is non-degenerate with respect to $\mathcal Y$
both on the left and the right. If a morphism is (left or right)
non-degenerate with respect to the one-element class $\{I\}$ then we simply
call it {\em non-degenerate}.     

If $\xymatrix@C=20pt{XA \ar[r]|(.5){\,f_1\,}& A & \ar[l]|(.5){\,f_2\,} AX}$ is
an $\mM$-morphism, and the multiplication of $A$ is non-degenerate with
respect to some class $\mathcal Y$, then $f_1$ is non-degenerate on the right
with respect to $\mathcal Y$ if and only if $f_2$ is non-degenerate on the
left with respect to $\mathcal Y$.  
Moreover, the following diagrams commute (see
\cite{BohmLack:cat_of_mbm}).     
\begin{equation}\label{eq:components_are_A-linear}
\xymatrix{
XA^2 \ar[r]^-{1m} \ar[d]_-{f_1 1} & 
XA \ar[d]^-{f_1}
&&
A^2X\ar[r]^-{m1} \ar[d]_-{1f_2} &
AX \ar[d]^-{f_2} \\
A^2 \ar[r]_-m &
A
&&
A^2 \ar[r]_-m &
A}
\end{equation}


Throughout the paper, string diagrams will be used to denote morphisms in 
braided monoidal categories. In order to give a gentle introduction
to their use, in the following definition we use both commutative diagrams and
string diagrams in parallel to present the axioms --- this is meant to serve
as the {\em Rosetta-stone}.  

\begin{definition} \label{def:weakly_counital_fusion_morphism}
A {\em weakly counital fusion morphism} in a braided monoidal category
$\mathsf C$ is an object $A$ equipped with three morphisms $t\colon A^2\to A^2$
(called the {\em fusion morphism}), $e\colon A^2\to A^2$, and 
$j\colon A\to I$ (called the {\em counit}). We introduce the morphism
\begin{center}
\begin{tabular}{rl}
$\xymatrix@C=15pt{m:=A^2 \ar[r]^-t & A^2 \ar[r]^-{j1} & A}\qquad\qquad\qquad$
&\raisebox{-15pt}{$
\begin{tikzpicture}
\draw[braid] (0,1) to[out=270,in=180] (.5,.5) to [out=0,in=270] (1,1);
\draw [braid] (.5,.5) to [out=270,in=90] (.5,0);
\draw (1.5,.5) node {$:=$};
\path (2.5,.5) node[arr,name=t] {$\ t\ $}
(2.2,.2) node[arr,name=j] {$\ $};
\draw [braid] (2.2,1) to[out=270,in=135] (t);
\draw [braid] (2.8,1) to[out=270,in=45] (t);
\draw [braid] (t) to[out=225,in=45] (j);
\draw [braid] (t) to[out=315,in=90] (2.8,0);
\end{tikzpicture}
$}
\end{tabular}
\end{center}
and impose the following axioms. 

\hypertarget{AxI}{{\bf Axiom I.}}
 The morphism $t$ obeys the fusion equation:
\medskip

\begin{center}
\begin{tabular}{rl}
$\xymatrix@C=15pt{
A^3 \ar[r]^-{1t} \ar[d]_-{t1} &
A^3 \ar[r]^-{c1} &
A^3 \ar[r]^-{1t} &
A^3 \ar[r]^-{c^{-1}1} &
A^3\ar[d]^-{t1} \\
A^3 \ar[rrrr]_-{1t} &&&&
A^3}\qquad\qquad$ &
\raisebox{-48pt}{$
\begin{tikzpicture}[scale=.7]
\path (2,3) node[arr,name=t1u] {$\ t\ $} 
(2,2) node[arr,name=t1d]  {$\ t\ $} 
(1.3,1) node[arr,name=t1l] {$\ t\ $};
\path[braid,name path=s1] (1.3,3.5) to[out=270,in=135] (t1d);
\draw[braid] (1.7,3.5) to[out=270,in=135] (t1u);
\draw[braid] (2.3,3.5) to[out=270,in=45] (t1u);
\draw[braid,name path=s4]  (t1u) to[out=225,in=90] (1.1,2) to[out=270,in=45]
(t1l);
\fill[white, name intersections={of=s4 and s1}] (intersection-1) circle(0.1);
\draw[braid] (1.3,3.5) to[out=270,in=135] (t1d);
\path[braid,name path=s5] (t1d) to[out=225,in=135] (t1l);
\fill[white, name intersections={of=s4 and s5}] (intersection-1) circle(0.1);
\draw[braid] (t1d) to[out=225,in=135] (t1l);
\draw[braid] (t1u) to[out=315,in=45] (t1d);
\draw[braid] (t1d) to[out=315,in=90] (2.5,.5);
\draw[braid] (t1l) to[out=315,in=90] (1.6,.5);
\draw[braid] (t1l) to[out=225,in=90] (1,.5);
\draw (3,2) node {$=$};
\path (4,2.5) node[arr,name=t1ul] {$\ t\ $} 
(4.7,1.5) node[arr,name=t1dr] {$\ t\ $};
\draw[braid] (3.7,3.5) to[out=270,in=135] (t1ul);
\draw[braid] (4.3,3.5) to[out=270,in=45] (t1ul);
\draw[braid] (t1ul) to[out=225,in=90] (3.7,.5);
\draw[braid] (t1ul) to[out=315,in=135] (t1dr);
\draw[braid] (t1dr) to[out=225,in=90] (4.3,.5);
\draw[braid] (t1dr) to[out=315,in=90] (5,.5);
\draw[braid] (5,3.5) to[out=270,in=45] (t1dr);
\end{tikzpicture}
$}
\end{tabular}
\end{center}
\bigskip

\hypertarget{AxII}{{\bf Axiom II.}} 
 The morphism $e$ is idempotent:
\medskip

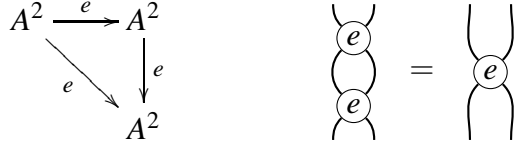
\begin{center}
\begin{tabular}{rl}
$\xymatrix{
A^2\ar[r]^-e\ar[rd]_-e &
A^2 \ar[d]^-e\\
& A^2}\qquad\qquad$ &
\raisebox{-42pt}{$\begin{tikzpicture}[scale=.9]
\path 
(1,2) node[arr,name=eu]  {$\ e\ $} 
(1,1) node[arr,name=ed] {$\ e\ $};
\draw[braid] (.7,2.5) to[out=270,in=135] (eu);
\draw[braid] (1.3,2.5) to[out=270,in=45] (eu);
\draw[braid] (eu) to[out=225,in=135] (ed);
\draw[braid] (eu) to[out=315,in=45] (ed);
\draw[braid] (ed) to[out=225,in=90] (.7,.5);
\draw[braid] (ed) to[out=315,in=90] (1.3,.5);
\draw (2,1.5) node {$=$};
\draw (3,1.5) node [arr,name=e] {$\ e\ $}; 
\draw[braid] (2.7,2.5) to[out=270,in=135] (e);
\draw[braid] (3.3,2.5) to[out=270,in=45] (e);
\draw[braid] (e) to[out=225,in=90] (2.7,.5);
\draw[braid] (e) to[out=315,in=90] (3.3,.5);
\end{tikzpicture}$}
\end{tabular}
\end{center}
\bigskip

\hypertarget{AxIII}{{\bf Axiom III.}}
The morphism $t$ is invariant under post-composition by $e$:
\bigskip

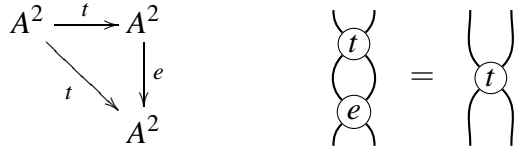
\begin{center}
\begin{tabular}{rl}
$\xymatrix{
A^2\ar[r]^-t\ar[rd]_-t &
A^2 \ar[d]^-e\\
& A^2}\qquad\qquad$ &
\raisebox{-43pt}{$
\begin{tikzpicture}[scale=.9]
\path 
(1,2) node[arr,name=eu]  {$\ t\ $} 
(1,1) node[arr,name=ed] {$\ e\ $};
\draw[braid] (.7,2.5) to[out=270,in=135] (eu);
\draw[braid] (1.3,2.5) to[out=270,in=45] (eu);
\draw[braid] (eu) to[out=225,in=135] (ed);
\draw[braid] (eu) to[out=315,in=45] (ed);
\draw[braid] (ed) to[out=225,in=90] (.7,.5);
\draw[braid] (ed) to[out=315,in=90] (1.3,.5);
\draw (2,1.5) node {$=$};
\draw (3,1.5) node [arr,name=e] {$\ t\ $}; 
\draw[braid] (2.7,2.5) to[out=270,in=135] (e);
\draw[braid] (3.3,2.5) to[out=270,in=45] (e);
\draw[braid] (e) to[out=225,in=90] (2.7,.5);
\draw[braid] (e) to[out=315,in=90] (3.3,.5);
\end{tikzpicture}$}
\end{tabular}
\end{center}
\bigskip

\hypertarget{AxIV}{{\bf Axiom IV.}}
The following composite morphism is invariant under pre-composition by
$e$: 
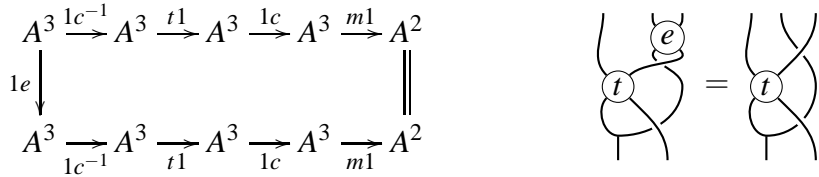
\begin{center}
\begin{tabular}{rl}
$\xymatrix@C=15pt{
A^3 \ar[d]_-{1e}\ar[r]^-{1c^{-1}} &
A^3 \ar[r]^-{t1} &
A^3 \ar[r]^-{1c} &
A^3 \ar[r]^-{m1} &
A^2 \ar@{=}[d]\\
A^3 \ar[r]_-{1c^{-1}} &
A^3 \ar[r]_-{t1} &
A^3 \ar[r]_-{1c} &
A^3 \ar[r]_-{m1} &
A^2}\qquad\qquad $
&\raisebox{-45pt}{$\begin{tikzpicture}[scale=.65]
\path 
(2,3) node[arr,name=e]  {$\ e\ $} 
(1,2) node[arr,name=tl] {$\ t\ $};
\path (2.3,2) node[empty,name=lt] {}
(1,1) node[empty,name=ml] {};
\draw[braid] (.7,3.5) to[out=270,in=135] (tl);
\draw[braid] (1.7,3.5) to[out=270,in=135] (e);
\draw[braid] (2.3,3.5) to[out=270,in=45] (e);
\path[braid,name path=s1] (e) to[out=315,in=45] (tl);
\path[braid,name path=s3] (tl) to[out=315,in=90] (2,.5);
\draw[braid,name path=s2] (e) to[out=225,in=90] (lt) to[out=270,in=0] (ml)
to[out=180, in=225] (tl);
\fill[white, name intersections={of=s2 and s1}] (intersection-1) circle(0.1);
\fill[white, name intersections={of=s2 and s3}] (intersection-1) circle(0.1);
\draw[braid] (e) to[out=315,in=45] (tl);
\draw[braid] (tl) to[out=315,in=90] (2,.5);
\draw[braid] (ml) to[out=270,in=90] (1,.5);
\draw (3,2) node {$=$};
\draw (4,2) node[arr,name=tr] {$\ t\ $};
\path (5,2) node[empty,name=rt] {}
(4,1) node[empty,name=mr] {};
\draw[braid] (3.7,3.5) to[out=270,in=135] (tr);
\path[braid,name path=s4] (5,3.5) to[out=270,in=45] (tr);
\path[braid,name path=s6] (tr) to[out=315,in=90] (5,.5);
\draw[braid,name path=s5] (4.3,3.5) to[out=270,in=90] (rt) to[out=270,in=0] (mr)
to[out=180, in=225] (tr);
\fill[white, name intersections={of=s4 and s5}] (intersection-1) circle(0.1);
\fill[white, name intersections={of=s6 and s5}] (intersection-1) circle(0.1);
\draw[braid] (5,3.5) to[out=270,in=45] (tr);
\draw[braid] (tr) to[out=315,in=90] (5,.5);
\draw[braid] (mr) to[out=270,in=90] (4,.5);
\end{tikzpicture}$}
\end{tabular}
\end{center}
\bigskip

\hypertarget{AxV}{{\bf Axiom V.}} 
The following commutativity relation holds between $t$ and $e$:
 
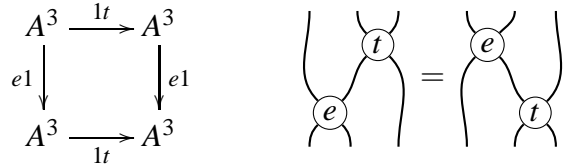
\begin{center}
\begin{tabular}{rl}
\xymatrix{
A^3 \ar[r]^-{1t}\ar[d]_-{e1} &
A^3 \ar[d]^-{e1} \\
A^3 \ar[r]_-{1t} &
A^3}\qquad\qquad&
\raisebox{-45pt}{$\begin{tikzpicture}[scale=.9]
\path (1,1) node[arr,name=el]  {$\ e\ $} 
(1.7,2) node[arr,name=tl] {$\ t\ $};
\draw[braid] (.7,2.5) to[out=270,in=135] (el);
\draw[braid] (1.4,2.5) to[out=270,in=135] (tl);
\draw[braid] (2,2.5) to[out=270,in=45] (tl);
\draw[braid] (tl) to[out=225,in=45] (el);
\draw[braid] (el) to[out=225,in=90] (.7,.5);
\draw[braid] (el) to[out=315,in=90] (1.3,.5);
\draw[braid] (tl) to[out=315,in=90] (2,.5);
\draw (2.5,1.5) node {$=$};
\path 
(3.3,2) node[arr,name=er]  {$\ e\ $} 
(4,1) node[arr,name=tr] {$\ t\ $};
\draw[braid] (3,2.5) to[out=270,in=135] (er);
\draw[braid] (3.7,2.5) to[out=270,in=45] (er);
\draw[braid] (4.3,2.5) to[out=270,in=45] (tr);
\draw[braid] (er) to[out=315,in=135] (tr);
\draw[braid] (er) to[out=225,in=90] (3,.5);
\draw[braid] (tr) to[out=225,in=90] (3.7,.5);
\draw[braid] (tr) to[out=315,in=90] (4.3,.5);
\end{tikzpicture}\qquad\qquad$}
\end{tabular}
\end{center}
\bigskip

\hypertarget{AxVI}{{\bf Axiom VI.}}
And also the following commutativity relation holds between
$t$ and $e$: 
 
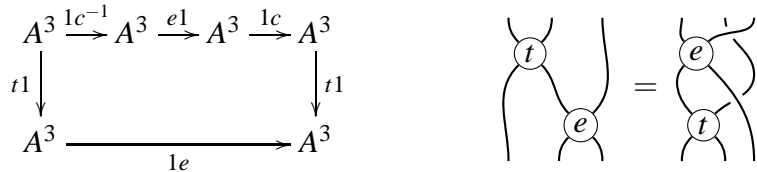
\begin{center}
\begin{tabular}{rl}
$\xymatrix@C=15pt{
A^3\ar[r]^-{1c^{-1}}\ar[d]_-{t1} &
A^3 \ar[r]^-{e1} &
A^3\ar[r]^-{1c} &
A^3 \ar[d]^-{t1} \\
A^3 \ar[rrr]_-{1e} &&&
A^3}\qquad\qquad$
\raisebox{-45pt}{$\begin{tikzpicture}[scale=.95]
\path (1,2) node[arr,name=tl]  {$\ t\ $} 
(1.7,1) node[arr,name=el] {$\ e\ $};
\draw[braid] (.7,2.5) to[out=270,in=135] (tl);
\draw[braid] (1.3,2.5) to[out=270,in=45] (tl);
\draw[braid] (2,2.5) to[out=270,in=45] (el);
\draw[braid] (tl) to[out=315,in=135] (el);
\draw[braid] (tl) to[out=225,in=90] (.7,.5);
\draw[braid] (el) to[out=225,in=90] (1.4,.5);
\draw[braid] (el) to[out=315,in=90] (2,.5);
\draw (2.6,1.5) node {$=$};
\path (3.3,2) node[arr,name=er] {$\ e\ $}
(3.4,1) node[arr,name=tr] {$\ t\ $};
\draw (4.1,1.7) node[empty,name=tp] {};
\draw[braid] (3.1,2.5) to[out=270,in=135] (er);
\path[braid,name path=s1] (4.1,2.5) to[out=270,in=45] (er);  
\path[braid,name path=s3] (er) to[out=315,in=90] (4.1,.5);  
\draw[braid,name path=s2] (3.7,2.5) to[out=270,in=90] (tp) to[out=270,in=45]
(tr);   
\fill[white, name intersections={of=s2 and s1}] (intersection-1) circle(0.1);
\fill[white, name intersections={of=s2 and s3}] (intersection-1) circle(0.1);
\draw[braid] (4.1,2.5) to[out=270,in=45] (er);  
\draw[braid] (er) to[out=315,in=90] (4.1,.5);  
\draw[braid] (er) to[out=225,in=135] (tr);  
\draw[braid] (tr) to[out=225,in=90] (3.1,.5);
\draw[braid] (tr) to[out=315,in=90] (3.7,.5);
\end{tikzpicture}$}
\end{tabular}
\end{center}
\bigskip

\hypertarget{AxVII}{{\bf Axiom VII.}} 
 The following compatibility relation holds between $t$ and
$e$: 

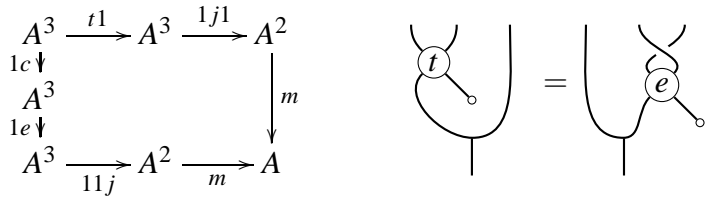
\begin{center}
\begin{tabular}{rl}
\xymatrix@R=7pt{
A^3 \ar[r]^-{t1} \ar[d]_-{1c} &
A^3 \ar[r]^-{1j1} &
A^2 \ar[dd]^-m\\
A^3\ar[d]_-{1e} \\
A^3 \ar[r]_-{11j} &
A^2\ar[r]_-m &
A} \qquad \qquad &
\raisebox{-53pt}{$\begin{tikzpicture}
\path (1,2) node[arr,name=t] {$\ t\ $}
(1.5,1.5) node[arr,name=ul] {$\ $}
(1.5,1) node[empty,name=ml] {};
\draw[braid] (.7,2.5) to[out=270,in=135] (t);
\draw[braid] (1.3,2.5) to[out=270,in=45] (t);
\draw[braid] (2,2.5) to[out=270,in=0] (ml) to[out=180,in=225] (t);
\draw[braid] (t) to[out=315,in=135] (ul);
\draw[braid] (ml) to[out=270,in=90] (1.5,.5);
\draw (2.6,1.7) node {$=$};
\path (4,1.7) node[arr,name=e] {$\ e\ $}
(4.5,1.2) node[arr,name=ur] {$\ $}
(3.5,1) node[empty,name=mr] {};
\draw[braid] (3,2.5) to[out=270,in=180] (mr) to[out=0,in=225] (e);
\draw[braid] (mr) to[out=270,in=90] (3.5,.5);
\draw[braid] (e) to[out=315,in=135] (ur);
\path[braid,name path=s1] (3.7,2.5) to[out=270,in=45] (e);
\draw[braid,name path=s2] (4.3,2.5) to[out=270,in=135] (e);
\fill[white, name intersections={of=s1 and s2}] (intersection-1) circle(0.1);
\draw[braid] (3.7,2.5) to[out=270,in=45] (e);
\end{tikzpicture}$}
\end{tabular}
\end{center}
\bigskip

\hypertarget{AxVIII}{{\bf Axiom VIII.}}
And also the following compatibility relation holds between
$t$ and 
$e$:

\begin{center}
\begin{tabular}{rl}
\xymatrix@R=7pt{
A^3 \ar[r]^-{t1} \ar[d]_-{1c} &
A^3 \ar[r]^-{1c} &
A^3 \ar[d]^-{m1} \\
A^3 \ar[d]_-{1e} && A^2 \ar[d]^-{j1} \\
A^3 \ar[r]_-{1j1} &
A^2 \ar[r]_-m &
A}\qquad\qquad&
\raisebox{-52pt}{$
\begin{tikzpicture}
\draw (1,2) node[arr,name=tr] {$\ t\ $};
\path (2,1.5) node[empty,name=rt] {}
(1,1) node[empty,name=ml] {}
(1,.5) node[arr,name=lu] {$\ $};
\draw[braid] (.7,2.5) to[out=270,in=135] (tr);
\path[braid,name path=s4] (2,2.5) to[out=270,in=45] (tr);
\path[braid,name path=s6] (tr) to[out=315,in=90] (2,.5);
\draw[braid,name path=s5] (2,2.5) to[out=270,in=90] (rt) to[out=270,in=0] (ml)
to[out=180, in=225] (tr);
\fill[white, name intersections={of=s4 and s5}] (intersection-1) circle(0.1);
\fill[white, name intersections={of=s6 and s5}] (intersection-1) circle(0.1);
\draw[braid] (1.3,2.5) to[out=270,in=45] (tr);
\draw[braid] (tr) to[out=315,in=90] (2,.5);
\draw[braid] (ml) to[out=270,in=90] (lu);
\draw (2.6,1.7) node {$=$};

\path (4,2) node[arr,name=e] {$\ e\ $}
(3.5,1.5) node[arr,name=ru] {$\ $}
(3.5,1) node[empty,name=mr] {};
\path[braid,name path=s1] (3.7,2.5) to[out=270,in=45] (e);
\draw[braid,name path=s2] (4.3,2.5) to[out=270,in=135] (e);
\fill[white, name intersections={of=s1 and s2}] (intersection-1) circle(0.1);
\draw[braid] (3.7,2.5) to[out=270,in=45] (e);
\draw[braid] (e) to[out=225,in=45] (ru);
\draw[braid] (3,2.5) to[out=270,in=180] (mr) to[out=0,in=315] (e);
\draw[braid] (mr) to[out=270,in=90] (3.5,.5);
\end{tikzpicture}$}
\end{tabular}
\end{center}
\end{definition}

It follows by Axioms \hyperlink{AxI}{\textnormal{I}},
\hyperlink{AxVIII}{\textnormal{VIII}} and \hyperlink{AxIII}{\textnormal{III}}
that for a weakly counital fusion morphism $(A,t,e,j)$, 
$m:=j1.t\colon A^2 \to A$ is an associative multiplication:  
$$
\begin{tikzpicture}
\path (1,2) node[arr,name=t1ul] {$\ t\ $} 
(1.7,1.5) node[arr,name=t1dr] {$\ t\ $}
(.7,1.7) node[arr,name=jl] {$\ $} 
(1.4,1.2) node[arr,name=jr] {$\ $};
\draw[braid] (.7,3) to[out=270,in=135] (t1ul);
\draw[braid] (1.3,3) to[out=270,in=45] (t1ul);
\draw[braid] (t1ul) to[out=225,in=45] (jl);
\draw[braid] (t1ul) to[out=315,in=135] (t1dr);
\draw[braid] (t1dr) to[out=225,in=45] (jr);
\draw[braid] (t1dr) to[out=315,in=90] (2,.5);
\draw[braid] (2,3) to[out=270,in=45] (t1dr);
\draw (3,2) node {$=$};

\path (5,2.5) node[arr,name=t1u] {$\ t\ $} 
(5,1.9) node[arr,name=t1d]  {$\ t\ $} 
(4.3,1.2) node[arr,name=t1l] {$\ t\ $}
(4,.9) node[arr,name=jl] {$\ $} 
(4.6,.9) node[arr,name=jr] {$\ $};
\path[braid,name path=s1] (4.3,3.5) to[out=270,in=135] (t1d);
\draw[braid] (4.7,3) to[out=270,in=135] (t1u);
\draw[braid] (5.3,3) to[out=270,in=45] (t1u);
\draw[braid,name path=s4]  (t1u) to[out=225,in=90] (4.2,2) to[out=270,in=45]
(t1l);
\fill[white, name intersections={of=s4 and s1}] (intersection-1) circle(0.1);
\draw[braid] (4.3,3) to[out=270,in=135] (t1d);
\path[braid,name path=s5] (t1d) to[out=225,in=135] (t1l);
\fill[white, name intersections={of=s4 and s5}] (intersection-1) circle(0.1);
\draw[braid] (t1d) to[out=225,in=135] (t1l);
\draw[braid] (t1u) to[out=315,in=45] (t1d);
\draw[braid] (t1d) to[out=315,in=90] (5.3,.5);
\draw[braid] (t1l) to[out=315,in=135] (jr);
\draw[braid] (t1l) to[out=225,in=45] (jl);
\draw (6,2) node {$=$};

\path (7.9,2.5) node[arr,name=t1u] {$\ t\ $} 
(7.9,1.9) node[arr,name=e]  {$\ e\ $} 
(7.6,1.2) node[arr,name=t1d] {$\ t\ $}
(7.3,.9) node[arr,name=jl] {$\ $} 
(7.6,1.6) node[arr,name=jr] {$\ $};
\draw[braid] (7,3) to[out=270,in=135] (t1d);
\draw[braid] (7.6,3) to[out=270,in=135] (t1u);
\draw[braid] (8.2,3) to[out=270,in=45] (t1u);
\draw[braid] (t1u) to[out=225,in=135] (e);
\draw[braid] (t1u) to[out=315,in=45] (e);
\draw[braid] (e) to[out=315,in=45] (t1d);
\draw[braid] (e) to[out=225,in=45] (jr);
\draw[braid] (t1d) to[out=225,in=45] (jl);
\draw[braid] (t1d) to[out=315,in=90] (7.9,.5);
\draw (9,2) node {$=$};

\path (10.9,2.5) node[arr,name=t1u] {$\ t\ $} 
(10.6,1.2) node[arr,name=t1d] {$\ t\ $}
(10.3,.9) node[arr,name=jl] {$\ $} 
(10.6,2.2) node[arr,name=jr] {$\ $};
\draw[braid] (10,3) to[out=270,in=135] (t1d);
\draw[braid] (10.6,3) to[out=270,in=135] (t1u);
\draw[braid] (11.2,3) to[out=270,in=45] (t1u);
\draw[braid] (t1u) to[out=315,in=45] (t1d);
\draw[braid] (t1u) to[out=225,in=45] (jr);
\draw[braid] (t1d) to[out=225,in=45] (jl);
\draw[braid] (t1d) to[out=315,in=90] (10.9,.5);
\end{tikzpicture}\ .
$$
Composing the fusion axiom (Axiom \hyperlink{AxI}{\textnormal{I}})
with $j11$, we obtain the {\em short fusion equation}: 
\begin{equation}\label{eq:short}
\begin{tikzpicture}[scale=.7]
\path (2,3) node[arr,name=t1u] {$\ t\ $} 
(2,2) node[arr,name=t1d]  {$\ t\ $} ;
\path[braid,name path=s1] (1.3,3.5) to[out=270,in=135] (t1d);
\draw[braid] (1.7,3.5) to[out=270,in=135] (t1u);
\draw[braid] (2.3,3.5) to[out=270,in=45] (t1u);
\draw[braid,name path=s4]  (t1u) to[out=225,in=90] (1.1,2) 
to[out=270,in=90] (1.6,1.3) to[out=270,in=0] (1.3,1);
\fill[white, name intersections={of=s4 and s1}] (intersection-1) circle(0.1);
\draw[braid] (1.3,3.5) to[out=270,in=135] (t1d);
\path[braid,name path=s5] (t1d) to[out=225,in=90] (1,1.3) 
to[out=270,in=180] (1.3,1);
\fill[white, name intersections={of=s4 and s5}] (intersection-1) circle(0.1);
\draw[braid] (t1d) to[out=225,in=90] (1,1.3) to[out=270,in=180] (1.3,1);
\draw[braid] (t1u) to[out=315,in=45] (t1d);
\draw[braid] (t1d) to[out=315,in=90] (2.5,.5);
\draw[braid] (1.3,1) to[out=270,in=90] (1.3,.5);
\draw (3,2) node {$=$};
\path 
(4.7,1.5) node[arr,name=t1dr] {$\ t\ $};
\draw[braid] (3.7,3.5) to[out=270,in=180] (4,2.7) to[out=0,in=270] (4.3,3.5);
\draw[braid] (4,2.7) to[out=270,in=135] (t1dr);
\draw[braid] (t1dr) to[out=225,in=90] (4.3,.5);
\draw[braid] (t1dr) to[out=315,in=90] (5.1,.5);
\draw[braid] (5,3.5) to[out=270,in=45] (t1dr);
\end{tikzpicture}
\end{equation}
Some further consequences of the axioms are analyzed in
Appendix~\ref{app:strings}.  

Recall from \cite{BohmLack:braided_mba} that a {\em counital fusion morphism}
is a pair consisting of a morphism $t\colon A^2\to A^2$ obeying Axiom
\hyperlink{AxI}{\textnormal{I}}, and a morphism $j\colon A\to I$ such
that $1j.t=1j$. It is not difficult to see that for any counital fusion
morphism $(t,j)$, there is a weakly counital fusion morphism
$(t,1,j)$. Conversely, if $(t,1,j)$ is a weakly counital fusion morphism and
the induced multiplication $j1.t$ is non-degenerate, then $(t,j)$ is a
counital fusion morphism.   

It is easy to see that braided strong monoidal functors send weakly 
counital fusion morphisms to weakly counital fusion morphisms. Since 
the braiding on \sfc makes the identity functor into a braided strong 
monoidal functor $\sfc\cong \sfc\rev$, we deduce the following 

\begin{lemma}
For morphisms $t\colon A^2 \to A^2$, $e\colon A^2\to A^2$ and $j\colon A\to I$
in a braided monoidal category $\mathsf C$, the following assertions are
equivalent.  
\begin{itemize}
\item[{(i)}] The datum $(t,e,j)$ is a weakly counital fusion morphism in
$\mathsf C$. 
\item[{(ii)}] The datum $(c^{-1}.t.c,c^{-1}.e.c,j)$ is a weakly 
counital fusion morphism in $\mathsf C^{\mathsf{rev}}$. 
\end{itemize}
\end{lemma}

A regular weak multiplier bimonoid in $\mathsf C$ will be defined as a
quadruple of weakly counital fusion morphisms which obey some compatibility
conditions to be discussed next. In developing the theory of regular 
weak multiplier bimonoids, we shall repeatedly use duality principles 
in which we move between the braided monoidal category \sfc, and the 
related braided monoidal categories $\sfc\rev$, $\overline{\sfc}$, 
and $\overline\sfc\rev$. We shall discuss this more fully below, but 
for the purpose of the following lemma, we observe that under these 
duality principles the 4-tuple $(t_1,t_2,t_3,t_4)$ in \sfc will 
correspond, respectively, to 4-tuples $(t_2,t_1,t_4,t_3)$, 
$(t_3,t_4,t_1,t_2)$, and $(t_4,t_3,t_2,t_1)$ in the other three 
categories; furthermore, in the last two cases, the we work with the 
reverse multiplication $m.c^{-1}$.

\begin{lemma}\label{lem:compatibility}
Let $A$ be a semigroup in $\mathsf C$ whose multiplication $m$ is 
non-degenerate with respect to  some class containing $A$. Between some 
morphisms $t_1,t_2,t_3,t_4\colon A^2\to A^2$, consider the following relations.  
\begin{center}
\begin{tabular}{rrr}
\inlineequation[eq:c12]{
\raisebox{-22pt}{$\begin{tikzpicture}[scale=.95]
\path (1,1.5) node[arr,name=t1] {$t_1$};
\draw[braid] (.3,2) to[out=270,in=180] (.5,1) to[out=0,in=225] (t1);
\draw[braid] (.5,1) to[out=270,in=90] (.5,.5);
\draw[braid] (.7,2) to[out=270,in=135] (t1);
\draw[braid] (1.3,2) to[out=270,in=45] (t1);
\draw[braid] (t1) to[out=315,in=90] (1.3,.5);
\path (1.8,1.3) node {$=$}; 
\path (2.5,1.5) node[arr,name=t2] {$t_2$};
\draw[braid] (t2) to[out=315,in=180] (3,1) to[out=0,in=270] (3.2,2);
\draw[braid] (3,1) to[out=270,in=90] (3,.5);
\draw[braid] (2.2,2) to[out=270,in=135] (t2);
\draw[braid] (2.8,2) to[out=270,in=45] (t2);
\draw[braid] (t2) to[out=225,in=90] (2.3,.5);
\end{tikzpicture}$}} &
\inlineequation[eq:c13]{
\raisebox{-22pt}{$\quad \begin{tikzpicture}[scale=.95]
\path (.5,1.5) node[arr,name=t1] {$t_1$};
\path[braid,name path=s1] (1.2,2) to[out=270,in=180] (.8,.8) ;
\draw[braid,name path=s2] (t1) to[out=315,in=0] (.8,.8);
\fill[white, name intersections={of=s1 and s2}] (intersection-1) circle(0.1);
\draw[braid] (1.2,2) to[out=270,in=180] (.8,.8); 
\draw[braid] (.8,.8) to[out=270,in=90] (.8,.5);
\draw[braid] (.2,2) to[out=270,in=135] (t1);
\draw[braid] (.8,2) to[out=270,in=45] (t1);
\draw[braid] (t1) to[out=225,in=90] (.2,.5);
\path (1.8,1.3) node {$=$}; 
\path (2.5,1.5) node[arr,name=t3] {$t_3$};
\draw[braid,name path=s1] (t3) to[out=315,in=180] (2.9,1) to[out=0,in=270]
(2.8,2); 
\draw[braid] (2.9,1) to[out=270,in=90] (2.9,.5);
\draw[braid] (2.3,2) to[out=270,in=135] (t3);
\path[braid,name path=s2] (3.2,2) to[out=270,in=45] (t3);
\fill[white, name intersections={of=s1 and s2}] (intersection-1) circle(0.1);
\draw[braid] (3.2,2) to[out=270,in=45] (t3);
\draw[braid] (t3) to[out=225,in=90] (2.3,.5);
\end{tikzpicture}$}} &
\inlineequation[eq:c14]{
\raisebox{-22pt}{$\begin{tikzpicture}[scale=.95]
\path (1,1.5) node[arr,name=t1] {$t_1$};
\draw[braid,name path=s1] (.3,2) to[out=270,in=0] (.6,.8);
\path[braid,name path=s2] (t1) to[out=225,in=180] (.6,.8);
\fill[white, name intersections={of=s1 and s2}] (intersection-1) circle(0.1);
\draw[braid] (t1) to[out=225,in=180] (.6,.8);
\draw[braid] (.6,.8) to[out=270,in=90] (.6,.5);
\draw[braid] (.7,2) to[out=270,in=135] (t1);
\draw[braid] (1.3,2) to[out=270,in=45] (t1);
\draw[braid] (t1) to[out=315,in=90] (1.3,.5);
\path (1.8,1.3) node {$=$}; 
\path (2.5,1.5) node[arr,name=t4] {$t_4$};
\draw[braid] (t4) to[out=315,in=180] (3,1) to[out=0,in=270] (3.2,2);
\draw[braid] (3,1) to[out=270,in=90] (3,.5);
\draw[braid] (2.2,2) to[out=270,in=135] (t4);
\draw[braid] (2.8,2) to[out=270,in=45] (t4);
\draw[braid] (t4) to[out=225,in=90] (2.3,.5);
\end{tikzpicture}\ \ $}}\\
\\
\inlineequation[eq:c23]{
\raisebox{-22pt}{$\begin{tikzpicture}[scale=.95]
\path (.5,1.5) node[arr,name=t2] {$t_2$};
\path[braid,name path=s1] (1.2,2) to[out=270,in=180] (.8,.8) ;
\draw[braid,name path=s2] (t2) to[out=315,in=0] (.8,.8);
\fill[white, name intersections={of=s1 and s2}] (intersection-1) circle(0.1);
\draw[braid] (1.2,2) to[out=270,in=180] (.8,.8); 
\draw[braid] (.8,.8) to[out=270,in=90] (.8,.5);
\draw[braid] (.2,2) to[out=270,in=135] (t2);
\draw[braid] (.8,2) to[out=270,in=45] (t2);
\draw[braid] (t2) to[out=225,in=90] (.2,.5);
\path (1.8,1.3) node {$=$}; 
\path (3,1.5) node[arr,name=t3] {$t_3$};
\draw[braid] (2.3,2) to[out=270,in=180] (2.5,1) to[out=0,in=225] (t3);
\draw[braid] (2.5,1) to[out=270,in=90] (2.5,.5);
\draw[braid] (2.7,2) to[out=270,in=135] (t3);
\draw[braid] (3.3,2) to[out=270,in=45] (t3);
\draw[braid] (t3) to[out=315,in=90] (3.2,.5);
\end{tikzpicture}\ \ $}} &
\inlineequation[eq:c24]{
\raisebox{-22pt}{$\begin{tikzpicture}[scale=.95]
\path (1,1.5) node[arr,name=t2] {$t_2$};
\draw[braid,name path=s1] (.3,2) to[out=270,in=0] (.6,.8);
\path[braid,name path=s2] (t2) to[out=225,in=180] (.6,.8);
\fill[white, name intersections={of=s1 and s2}] (intersection-1) circle(0.1);
\draw[braid] (t2) to[out=225,in=180] (.6,.8);
\draw[braid] (.6,.8) to[out=270,in=90] (.6,.5);
\draw[braid] (.7,2) to[out=270,in=135] (t2);
\draw[braid] (1.3,2) to[out=270,in=45] (t2);
\draw[braid] (t2) to[out=315,in=90] (1.3,.5);
\path (1.8,1.3) node {$=$}; 
\path (3,1.5) node[arr,name=t4] {$t_4$};
\path[braid,name path=s1] (2.7,2) to[out=270,in=180] (2.5,1) to[out=0,in=225]
(t4); 
\draw[braid] (2.5,1) to[out=270,in=90] (2.5,.5);
\draw[braid,name path=s2] (2.3,2) to[out=270,in=135] (t4);
\fill[white, name intersections={of=s1 and s2}] (intersection-1) circle(0.1);
\draw[braid] (2.7,2) to[out=270,in=180] (2.5,1) to[out=0,in=225] (t4); 
\draw[braid] (3.2,2) to[out=270,in=45] (t4);
\draw[braid] (t4) to[out=315,in=90] (3.2,.5);
\end{tikzpicture}\ \ $}}&
\inlineequation[eq:c34]{
\raisebox{-22pt}{$\begin{tikzpicture}[scale=.95]
\path (1,1.5) node[arr,name=t3] {$t_3$};
\draw[braid,name path=s1] (.3,2) to[out=270,in=0] (.6,.8);
\path[braid,name path=s2] (t1) to[out=225,in=180] (.6,.8);
\fill[white, name intersections={of=s1 and s2}] (intersection-1) circle(0.1);
\draw[braid] (t3) to[out=225,in=180] (.6,.8);
\draw[braid] (.6,.8) to[out=270,in=90] (.6,.5);
\draw[braid] (.7,2) to[out=270,in=135] (t3);
\draw[braid] (1.3,2) to[out=270,in=45] (t3);
\draw[braid] (t3) to[out=315,in=90] (1.3,.5);
\path (1.8,1.3) node {$=$}; 
\path (2.5,1.5) node[arr,name=t4] {$t_4$};
\path[braid,name path=s1] (3.2,2) to[out=270,in=180] (2.8,.8) ;
\draw[braid,name path=s2] (t4) to[out=315,in=0] (2.8,.8);
\fill[white, name intersections={of=s1 and s2}] (intersection-1) circle(0.1);
\draw[braid] (3.2,2) to[out=270,in=180] (2.8,.8); 
\draw[braid] (2.8,.8) to[out=270,in=90] (2.8,.5);
\draw[braid] (2.2,2) to[out=270,in=135] (t4);
\draw[braid] (2.8,2) to[out=270,in=45] (t4);
\draw[braid] (t4) to[out=225,in=90] (2.2,.5);
\end{tikzpicture}$}}
\end{tabular}
\end{center}
\bigskip

\noindent 
If the identity involving $t_i$ and $t_j$ holds, and the identity involving
$t_i$ and $t_k$ holds, for $i,j,k$ different elements of the set
$\{1,2,3,4\}$, then also the identity involving $t_j$ and $t_k$ holds.
\end{lemma}

\begin{proof}
We only show three of the implications, all other implications follow
symmetrically.  

By associativity of the multiplication, \eqref{eq:c12} and \eqref{eq:c13},
$$
\begin{tikzpicture}
\path (.5,2) node[arr,name=t2] {$t_2$};
\path[braid,name path=s1] (1.2,2.5) to[out=270,in=180] (.8,1.3) ;
\draw[braid,name path=s2] (t2) to[out=315,in=0] (.8,1.3);
\fill[white, name intersections={of=s1 and s2}] (intersection-1) circle(0.1);
\draw[braid] (1.2,2.5) to[out=270,in=180] (.8,1.3); 
\draw[braid] (.8,1.3) to[out=270,in=180] (1,1.1) to[out=0,in=270] (1.6,2.5);
\draw[braid] (1,1.1) to[out=270,in=90] (1,.8);
\draw[braid] (.2,2.5) to[out=270,in=135] (t2);
\draw[braid] (.8,2.5) to[out=270,in=45] (t2);
\draw[braid] (t2) to[out=225,in=90] (.2,.8);
\path (2,1.7) node {$=$}; 

\path (2.7,2) node[arr,name=t2] {$t_2$};
\draw[braid,name path=s1] (t2) to[out=315,in=180] (3.1,1.6) to[out=0,in=270]
(3.5,2.5); 
\draw[braid,name path=s2] (3.1,1.6) to[out=270,in=0] (3.2,1);
\path[braid,name path=s3] (3.2,1) to[out=180,in=270] (3.5,1.8)
to[out=90,in=270] (3.2,2.5); 
\fill[white, name intersections={of=s1 and s3}] (intersection-1) circle(0.1);
\fill[white, name intersections={of=s3 and s2}] (intersection-1) circle(0.1);
\draw[braid] (3.2,1) to[out=180,in=270] (3.5,1.8) to[out=90,in=270] (3.2,2.5);
\draw[braid] (3.2,1) to[out=270,in=90] (3.2,.8);
\draw[braid] (2.5,2.5) to[out=270,in=135] (t2);
\draw[braid] (2.9,2.5) to[out=270,in=45] (t2);
\draw[braid] (t2) to[out=225,in=90] (2.5,.8);
\path (3.95,1.7) node {$=$};

\path (5,1.8) node[arr,name=t1] {$t_1$};
\path[braid,name path=s1] (5.1,2.5) to[out=270,in=90] (5.5,1.8)
to[out=270,in=180] (5.3,1.1) ; 
\draw[braid,name path=s2] (t1) to[out=315,in=0] (5.3,1.1);
\draw[braid,name path=s3] (5.5,2.5) to[out=270,in=45] (t1);
\fill[white, name intersections={of=s1 and s2}] (intersection-1) circle(0.1);
\fill[white, name intersections={of=s1 and s3}] (intersection-1) circle(0.1);
\draw[braid] (5.1,2.5) to[out=270,in=90] (5.5,1.8) to[out=270,in=180]
(5.3,1.1) ;  
\draw[braid] (5.3,1.1) to[out=270,in=90] (5.3,.8);
\draw[braid] (4.7,2.5) to[out=270,in=135] (t1);
\draw[braid] (4.4,2.5) to[out=270,in=180] (4.6,1.1) to [out=0,in=225] (t1);
\draw[braid] (4.6,1.1) to[out=270,in=90] (4.6,.8);
\path (5.95,1.7) node {$=$};

\path (7,1.8) node[arr,name=t3] {$t_3$};
\draw[braid] (6.4,2.5) to[out=270,in=180] (6.6,1.3) to[out=0,in=225] (t3);
\draw[braid] (6.6,1.3) to[out=270,in=90] (6.6,.8);
\draw[braid] (6.7,2.5) to[out=270,in=135] (t3);
\draw[braid] (7.3,2.5) to[out=270,in=45] (t3);
\draw[braid] (t3) to[out=315,in=180] (7.4,1.3) to[out=0,in=270] (7.6,2.5);
\draw[braid] (7.4,1.3) to[out=270,in=90] (7.4,.8);
\end{tikzpicture}
$$
and now by non-degeneracy of the multiplication this proves 
\eqref{eq:c23}. Similarly, by \eqref{eq:c12}, associativity of the 
multiplication, and \eqref{eq:c14}, 
$$
\begin{tikzpicture}
\path (1.1,1.7) node[arr,name=t2] {$t_2$};
\draw[braid,name path=s1] (.3,2.2) to[out=270,in=0] (.6,.8);
\path[braid,name path=s2] (t2) to[out=225,in=180] (.6,.8);
\fill[white, name intersections={of=s1 and s2}] (intersection-1) circle(0.1);
\draw[braid] (t2) to[out=225,in=180] (.6,.8);
\draw[braid] (.6,.8) to[out=270,in=90] (.6,.5);
\draw[braid] (.8,2.2) to[out=270,in=135] (t2);
\draw[braid] (1.4,2.2) to[out=270,in=45] (t2);
\draw[braid] (t2) to[out=315,in=180] (1.6,.8) to[out=0,in=270] (1.7,2.2);
\draw[braid] (1.6,.8) to[out=270,in=90] (1.6,.5);
\path (2.5,1.4) node {$=$}; 
\path (4,1.8) node[arr,name=t1] {$t_1$};
\draw[braid] (3.4,2.2) to[out=270,in=180] (3.6,1.5) to[out=0,in=225] (t1);
\path[braid,name path=s3] (3.6,1.5) to[out=270,in=180] (3.4,.8);
\draw[braid,name path=s4] (2.9,2.2) to[out=270,in=0] (3.4,.8);
\fill[white, name intersections={of=s3 and s4}] (intersection-1) circle(0.1);
\draw[braid] (3.6,1.5) to[out=270,in=180] (3.4,.8);
\draw[braid] (3.4,.8) to[out=270,in=90] (3.4,.5);
\draw[braid] (3.7,2.2) to[out=270,in=135] (t1);
\draw[braid] (4.3,2.2) to[out=270,in=45] (t1);
\draw[braid] (t1) to[out=315,in=90] (4.3,.5);
\path (5,1.4) node {$=$}; 

\path (6.5,1.8) node[arr,name=t1] {$t_1$};
\draw[braid,name path=s5] (5.5,2.2) to[out=270,in=0] (6.1,1.2);
\path[braid,name path=s6] (t1) to[out=225,in=180] (6.1,1.2);
\path[braid,name path=s7] (5.9,2.2) to[out=225,in=180] (5.7,.8)
to[out=0,in=270] (6.1,1.2);
\fill[white, name intersections={of=s5 and s7}] (intersection-1) circle(0.1);
\draw[braid] (5.9,2.2) to[out=225,in=180] (5.7,.8) to[out=0,in=270] (6.1,1.2);
\fill[white, name intersections={of=s5 and s6}] (intersection-1) circle(0.1);
\draw[braid] (t1) to[out=225,in=180] (6.1,1.2);
\draw[braid] (5.7,.8) to[out=270,in=90] (5.7,.5);
\draw[braid] (6.2,2.2) to[out=270,in=135] (t1);
\draw[braid] (6.8,2.2) to[out=270,in=45] (t1);
\draw[braid] (t1) to[out=315,in=90] (6.8,.5);
\path (7.3,1.4) node {$=$}; 

\path (8.3,1.5) node[arr,name=t4] {$t_4$};
\path[braid,name path=s1] (8,2.2) to[out=270,in=180] (7.8,1) to[out=0,in=225]
(t4); 
\draw[braid] (7.8,1) to[out=270,in=90] (7.8,.5);
\draw[braid,name path=s2] (7.6,2.2) to[out=270,in=135] (t4);
\fill[white, name intersections={of=s1 and s2}] (intersection-1) circle(0.1);
\draw[braid] (8,2.2) to[out=270,in=180] (7.8,1) to[out=0,in=225] (t4); 
\draw[braid] (8.5,2.2) to[out=270,in=45] (t4);
\draw[braid] (t4) to[out=315,in=180] (8.7,1) to[out=0,in=270] (8.9,2.2);
\draw[braid] (8.7,1) to[out=270,in=90] (8.7,.5);
\end{tikzpicture}
$$ 
and now by non-degeneracy of the multiplication this implies 
\eqref{eq:c24}. Finally, by \eqref{eq:c13}, \eqref{eq:c14}, and associativity
of the multiplication, 
$$
\begin{tikzpicture}
\path (1.1,2) node[arr,name=t3] {$t_3$};
\draw[braid,name path=s1] (.3,2.5) to[out=270,in=0] (.6,1.3);
\path[braid,name path=s2] (t3) to[out=225,in=180] (.6,1.3);
\fill[white, name intersections={of=s1 and s2}] (intersection-1) circle(0.1);
\draw[braid] (t3) to[out=225,in=180] (.6,1.3);
\draw[braid] (.6,1.3) to[out=270,in=90] (.6,.8);
\draw[braid] (.8,2.5) to[out=270,in=135] (t3);
\draw[braid] (1.4,2.5) to[out=270,in=45] (t3);
\draw[braid] (t3) to[out=315,in=180] (1.6,1.3) to[out=0,in=270] (1.7,2.5);
\draw[braid] (1.6,1.3) to[out=270,in=90] (1.6,.8);
\path (2.4,1.7) node {$=$}; 

\path (3.5,1.8) node[arr,name=t1] {$t_1$};
\path[braid,name path=s5] (3.6,2.5) to[out=270,in=90] (4,1.8)
to[out=270,in=180] (3.8,1.1) ; 
\draw[braid,name path=s4] (t1) to[out=315,in=0] (3.8,1.1);
\draw[braid,name path=s3] (4,2.5) to[out=270,in=45] (t1);
\fill[white, name intersections={of=s5 and s4}] (intersection-1) circle(0.1);
\fill[white, name intersections={of=s5 and s3}] (intersection-1) circle(0.1);
\draw[braid] (3.6,2.5) to[out=270,in=90] (4,1.8) to[out=270,in=180]
(3.8,1.1) ;  
\draw[braid] (3.8,1.1) to[out=270,in=90] (3.8,.8);
\draw[braid] (3.2,2.5) to[out=270,in=135] (t1);
\draw[braid,name path=s7] (2.9,2.5) to[out=270,in=0] (3.1,1.1);
\path[braid,name path=s6] (t1) to [out=225,in=180] (3.1,1.1);
\draw[braid] (3.1,1.1) to[out=270,in=90] (3.1,.8);
\fill[white, name intersections={of=s6 and s7}] (intersection-1) circle(0.1);
\draw[braid] (t1) to [out=225,in=180] (3.1,1.1);
\path (4.5,1.7) node {$=$};

\path (5.4,2) node[arr,name=t4] {$t_4$};
\draw[braid,name path=s1] (t4) to[out=315,in=180] (5.8,1.6) to[out=0,in=270]
(6.2,2.5); 
\draw[braid,name path=s2] (5.8,1.6) to[out=270,in=0] (5.9,1);
\path[braid,name path=s3] (5.9,1) to[out=180,in=270] (6.2,1.8)
to[out=90,in=270] (5.9,2.5); 
\fill[white, name intersections={of=s1 and s3}] (intersection-1) circle(0.1);
\fill[white, name intersections={of=s3 and s2}] (intersection-1) circle(0.1);
\draw[braid] (5.9,1) to[out=180,in=270] (6.2,1.8) to[out=90,in=270] (5.9,2.5);
\draw[braid] (5.9,1) to[out=270,in=90] (5.9,.8);
\draw[braid] (5.2,2.5) to[out=270,in=135] (t4);
\draw[braid] (5.6,2.5) to[out=270,in=45] (t4);
\draw[braid] (t4) to[out=225,in=90] (5.2,.8);
\path (6.7,1.7) node {$=$};

\path (7.5,2) node[arr,name=t4] {$t_4$};
\path[braid,name path=s1] (8.2,2.5) to[out=270,in=180] (7.8,1.3) ;
\draw[braid,name path=s2] (t4) to[out=315,in=0] (7.8,1.3);
\fill[white, name intersections={of=s1 and s2}] (intersection-1) circle(0.1);
\draw[braid] (8.2,2.5) to[out=270,in=180] (7.8,1.3); 
\draw[braid] (7.8,1.3) to[out=270,in=180] (8,1.1) to[out=0,in=270] (8.6,2.5);
\draw[braid] (8,1.1) to[out=270,in=90] (8,.8);
\draw[braid] (7.2,2.5) to[out=270,in=135] (t4);
\draw[braid] (7.8,2.5) to[out=270,in=45] (t4);
\draw[braid] (t4) to[out=225,in=90] (7.2,.8);
\end{tikzpicture}
$$ 
and by non-degeneracy of the multiplication this implies \eqref{eq:c34}.  
\end{proof}

\begin{corollary}\label{cor:compatibility}
In the setting of Lemma \ref{lem:compatibility}, the following assertions are
equivalent: 
\begin{itemize}
\item[{(i)}] Conditions \eqref{eq:c12}, \eqref{eq:c13} and \eqref{eq:c14}
  hold; 
\item[{(ii)}] Conditions \eqref{eq:c12}, \eqref{eq:c23} and \eqref{eq:c24}
  hold;
\item[{(iii)}] Conditions \eqref{eq:c13}, \eqref{eq:c23} and \eqref{eq:c34}
  hold;  
\item[{(iv)}] Conditions \eqref{eq:c14}, \eqref{eq:c24} and \eqref{eq:c34}
  hold. 
\end{itemize}
\end{corollary}

\begin{corollary}\label{cor:unique_t}
If there are morphisms $t_1,t_2,t_3,t_4$ satisfying the conditions of Corollary
\ref{cor:compatibility}, then any one of them uniquely determines all of the
others. 
\end{corollary}

\begin{lemma}\label{lem:interchange}
Let $A$ be a semigroup in $\mathsf C$ whose multiplication is non-degenerate
with respect to  some class containing $A$ and $A^2$. Let 
$t_1,t_2,t_3,t_4\colon A^2\to A^2$ be morphisms satisfying the 
conditions of Lemma \ref{lem:compatibility}. The following conditions 
are equivalent to each other.  
\begin{center}
\begin{tabular}{rr}
\inlineequation[eq:interchange_12]{
\raisebox{-22pt}{$\begin{tikzpicture}[scale=.95]
\path (1,1) node[arr,name=t2l]  {$t_2$} 
(1.5,1.5) node[arr,name=t1l] {$t_1$};
\draw[braid] (.7,2) to[out=270,in=135] (t2l);
\draw[braid] (1.2,2) to[out=270,in=135] (t1l);
\draw[braid] (1.8,2) to[out=270,in=45] (t1l);
\draw[braid] (t1l) to[out=225,in=45] (t2l);
\draw[braid] (t2l) to[out=225,in=90] (.7,.5);
\draw[braid] (t2l) to[out=315,in=90] (1.3,.5);
\draw[braid] (t1l) to[out=315,in=90] (1.8,.5);
\draw (2.3,1.5) node {$=$};
\path (3.1,1.5) node[arr,name=t2r]  {$t_2$} 
(3.6,1) node[arr,name=t1r] {$t_1$};
\draw[braid] (2.8,2) to[out=270,in=135] (t2r);
\draw[braid] (3.4,2) to[out=270,in=45] (t2r);
\draw[braid] (3.9,2) to[out=270,in=45] (t1r);
\draw[braid] (t2r) to[out=315,in=135] (t1r);
\draw[braid] (t2r) to[out=225,in=90] (2.8,.5);
\draw[braid] (t1r) to[out=225,in=90] (3.3,.5);
\draw[braid] (t1r) to[out=315,in=90] (3.9,.5);
\end{tikzpicture}$}} &
\inlineequation[eq:interchange_14]{
\raisebox{-22pt}{$\quad \begin{tikzpicture}[scale=.95]
\path (1,1) node[arr,name=t2l]  {$t_4$} 
(1.5,1.5) node[arr,name=t1l] {$t_1$};
\draw[braid] (.7,2) to[out=270,in=135] (t2l);
\draw[braid] (1.2,2) to[out=270,in=135] (t1l);
\draw[braid] (1.8,2) to[out=270,in=45] (t1l);
\draw[braid] (t1l) to[out=225,in=45] (t2l);
\draw[braid] (t2l) to[out=225,in=90] (.7,.5);
\draw[braid] (t2l) to[out=315,in=90] (1.3,.5);
\draw[braid] (t1l) to[out=315,in=90] (1.8,.5);
\draw (2.3,1.5) node {$=$};
\path (3.1,1.5) node[arr,name=t2r]  {$t_4$} 
(3.6,1) node[arr,name=t1r] {$t_1$};
\draw[braid] (2.8,2) to[out=270,in=135] (t2r);
\draw[braid] (3.4,2) to[out=270,in=45] (t2r);
\draw[braid] (3.9,2) to[out=270,in=45] (t1r);
\draw[braid] (t2r) to[out=315,in=135] (t1r);
\draw[braid] (t2r) to[out=225,in=90] (2.8,.5);
\draw[braid] (t1r) to[out=225,in=90] (3.3,.5);
\draw[braid] (t1r) to[out=315,in=90] (3.9,.5);
\end{tikzpicture}$}}\\
\\
\inlineequation[eq:interchange_23]{
\raisebox{-22pt}{$\begin{tikzpicture}[scale=.95]
\path (1,1) node[arr,name=t2l]  {$t_2$} 
(1.5,1.5) node[arr,name=t1l] {$t_3$};
\draw[braid] (.7,2) to[out=270,in=135] (t2l);
\draw[braid] (1.2,2) to[out=270,in=135] (t1l);
\draw[braid] (1.8,2) to[out=270,in=45] (t1l);
\draw[braid] (t1l) to[out=225,in=45] (t2l);
\draw[braid] (t2l) to[out=225,in=90] (.7,.5);
\draw[braid] (t2l) to[out=315,in=90] (1.3,.5);
\draw[braid] (t1l) to[out=315,in=90] (1.8,.5);
\draw (2.3,1.5) node {$=$};
\path (3.1,1.5) node[arr,name=t2r]  {$t_2$} 
(3.6,1) node[arr,name=t1r] {$t_3$};
\draw[braid] (2.8,2) to[out=270,in=135] (t2r);
\draw[braid] (3.4,2) to[out=270,in=45] (t2r);
\draw[braid] (3.9,2) to[out=270,in=45] (t1r);
\draw[braid] (t2r) to[out=315,in=135] (t1r);
\draw[braid] (t2r) to[out=225,in=90] (2.8,.5);
\draw[braid] (t1r) to[out=225,in=90] (3.3,.5);
\draw[braid] (t1r) to[out=315,in=90] (3.9,.5);
\end{tikzpicture}$}} &
\inlineequation[eq:interchange_34]{
\raisebox{-22pt}{$\begin{tikzpicture}[scale=.95]
\path (1,1) node[arr,name=t2l]  {$t_4$} 
(1.5,1.5) node[arr,name=t1l] {$t_3$};
\draw[braid] (.7,2) to[out=270,in=135] (t2l);
\draw[braid] (1.2,2) to[out=270,in=135] (t1l);
\draw[braid] (1.8,2) to[out=270,in=45] (t1l);
\draw[braid] (t1l) to[out=225,in=45] (t2l);
\draw[braid] (t2l) to[out=225,in=90] (.7,.5);
\draw[braid] (t2l) to[out=315,in=90] (1.3,.5);
\draw[braid] (t1l) to[out=315,in=90] (1.8,.5);
\draw (2.3,1.5) node {$=$};
\path (3.1,1.5) node[arr,name=t2r]  {$t_4$} 
(3.6,1) node[arr,name=t1r] {$t_3$};
\draw[braid] (2.8,2) to[out=270,in=135] (t2r);
\draw[braid] (3.4,2) to[out=270,in=45] (t2r);
\draw[braid] (3.9,2) to[out=270,in=45] (t1r);
\draw[braid] (t2r) to[out=315,in=135] (t1r);
\draw[braid] (t2r) to[out=225,in=90] (2.8,.5);
\draw[braid] (t1r) to[out=225,in=90] (3.3,.5);
\draw[braid] (t1r) to[out=315,in=90] (3.9,.5);
\end{tikzpicture}$}} 
\end{tabular}
\end{center}
\end{lemma}

\begin{proof}
Again, we only show that \eqref{eq:interchange_12} implies
\eqref{eq:interchange_14}, and \eqref{eq:interchange_14} implies 
\eqref{eq:interchange_34}; all other implications follow symmetrically.
 
By \eqref{eq:c24}, \eqref{eq:interchange_12} and \eqref{eq:c24} again,
$$
\begin{tikzpicture}
\path (1,1) node[arr,name=t4l]  {$t_4$} 
(1.5,1.5) node[arr,name=t1l] {$t_1$};
\draw[braid,name path=s1] (.5,2) to[out=270,in=135] (t4l);
\path[braid,name path=s2] (.9,2) to[out=270,in=180] (.6,.6) to[out=0,in=225]
(t4l);
\fill[white, name intersections={of=s1 and s2}] (intersection-1) circle(0.1);
\draw[braid] (.9,2) to[out=270,in=180] (.6,.6) to[out=0,in=225] (t4l);
\draw[braid] (1.2,2) to[out=270,in=135] (t1l);
\draw[braid] (1.8,2) to[out=270,in=45] (t1l);
\draw[braid] (t1l) to[out=225,in=45] (t4l);
\draw[braid] (.6,.6) to[out=270,in=90] (.6,.3);
\draw[braid] (t4l) to[out=315,in=90] (1.3,.3);
\draw[braid] (t1l) to[out=315,in=90] (1.8,.3);
\draw (2.2,1.2) node {$=$};

\path (3.1,1) node[arr,name=t2]  {$t_2$} 
(3.5,1.5) node[arr,name=t1] {$t_1$};
\draw[braid] (2.8,2) to[out=270,in=135] (t2);
\draw[braid,name path=s3] (2.4,2) to[out=270,in=0] (2.6,.5);
\path[braid,name path=s4] (t2) to[out=225,in=180] (2.6,.5);
\fill[white, name intersections={of=s3 and s4}] (intersection-1) circle(0.1);
\draw[braid] (t2) to[out=225,in=180] (2.6,.5);
\draw[braid] (3.2,2) to[out=270,in=135] (t1);
\draw[braid] (3.8,2) to[out=270,in=45] (t1);
\draw[braid] (t1) to[out=225,in=45] (t2);
\draw[braid] (2.6,.5) to[out=270,in=90] (2.6,.3);
\draw[braid] (t2) to[out=315,in=90] (3.3,.3);
\draw[braid] (t1) to[out=315,in=90] (3.8,.3);
\draw (4.3,1.2) node {$=$};

\path (5.3,1.5) node[arr,name=t2]  {$t_2$} 
(5.7,1) node[arr,name=t1] {$t_1$};
\draw[braid] (5,2) to[out=270,in=135] (t2);
\draw[braid,name path=s3] (4.6,2) to[out=270,in=0] (4.8,.5);
\path[braid,name path=s4] (t2) to[out=225,in=180] (4.8,.5);
\fill[white, name intersections={of=s3 and s4}] (intersection-1) circle(0.1);
\draw[braid] (t2) to[out=225,in=180] (4.8,.5);
\draw[braid] (5.6,2) to[out=270,in=45] (t2);
\draw[braid] (6,2) to[out=270,in=45] (t1);
\draw[braid] (t2) to[out=315,in=135] (t1);
\draw[braid] (4.8,.5) to[out=270,in=90] (4.8,.3);
\draw[braid] (t1) to[out=225,in=90] (5.4,.3);
\draw[braid] (t1) to[out=315,in=90] (6,.3);
\draw (6.5,1.2) node {$=$};

\path (7.3,1.5) node[arr,name=t4l]  {$t_4$} 
(7.8,1) node[arr,name=t1l] {$t_1$};
\draw[braid,name path=s1] (6.8,2) to[out=270,in=135] (t4l);
\path[braid,name path=s2] (7,2) to[out=270,in=180] (7.1,.8) to[out=0,in=225]
(t4l);
\fill[white, name intersections={of=s1 and s2}] (intersection-1) circle(0.1);
\draw[braid] (7,2) to[out=270,in=180] (7.1,.8) to[out=0,in=225] (t4l);
\draw[braid] (7.5,2) to[out=270,in=45] (t4l);
\draw[braid] (8.1,2) to[out=270,in=45] (t1l);
\draw[braid] (t4l) to[out=315,in=135] (t1l);
\draw[braid] (7.1,.8) to[out=270,in=90] (7.1,.3);
\draw[braid] (t1l) to[out=225,in=90] (7.5,.3);
\draw[braid] (t1l) to[out=315,in=90] (8.1,.3);
\end{tikzpicture}
$$
By the non-degeneracy condition on the multiplication this implies
\eqref{eq:interchange_14}. 

Similarly by \eqref{eq:c13}, \eqref{eq:interchange_14} and 
\eqref{eq:c13} again,  
$$
\begin{tikzpicture}
\path (1,1) node[arr,name=t2l]  {$t_4$} 
(1.5,1.5) node[arr,name=t3l] {$t_3$};
\draw[braid] (.7,2) to[out=270,in=135] (t2l);
\draw[braid] (1.2,2) to[out=270,in=135] (t3l);
\path[braid,name path=s1] (2.1,2) to[out=270,in=45] (t3l);
\draw[braid,name path=s2] (1.7,2) to[out=270,in=0] (1.9,.8) to[out=180,in=315]
(t3l); 
\fill[white, name intersections={of=s1 and s2}] (intersection-1) circle(0.1);
\draw[braid] (2.1,2) to[out=270,in=45] (t3l);
\draw[braid] (t3l) to[out=225,in=45] (t2l);
\draw[braid] (t2l) to[out=225,in=90] (.7,.3);
\draw[braid] (t2l) to[out=315,in=90] (1.3,.3);
\draw[braid] (1.9,.8) to[out=270,in=90] (1.9,.3);
\draw (2.5,1.2) node {$=$};

\path (3.2,1) node[arr,name=t2]  {$t_4$} 
(3.7,1.5) node[arr,name=t1] {$t_1$};
\draw[braid] (2.9,2) to[out=270,in=135] (t2);
\draw[braid] (3.4,2) to[out=270,in=135] (t1);
\draw[braid] (4,2) to[out=270,in=45] (t1);
\path[braid,name path=s3] (4.4,2) to[out=270,in=180] (4.2,.8);
\draw[braid,name path=s4] (t1) to[out=315,in=0] (4.2,.8);
\fill[white, name intersections={of=s3 and s4}] (intersection-1) circle(0.1);
\draw[braid] (4.4,2) to[out=270,in=180] (4.2,.8);
\draw[braid] (t1) to[out=225,in=45] (t2);
\draw[braid] (t2) to[out=225,in=90] (2.9,.3);
\draw[braid] (t2) to[out=315,in=90] (3.5,.3);
\draw[braid] (4.2,.8) to[out=270,in=90] (4.2,.3);
\draw (4.7,1.2) node {$=$};

\path (5.5,1.5) node[arr,name=t2]  {$t_4$} 
(6,1) node[arr,name=t1] {$t_1$};
\draw[braid] (5.2,2) to[out=270,in=135] (t2);
\draw[braid] (5.7,2) to[out=270,in=45] (t2);
\draw[braid] (6.3,2) to[out=270,in=45] (t1);
\path[braid,name path=s5] (6.7,2) to[out=270,in=180] (6.5,.5);
\draw[braid,name path=s6] (t1) to[out=315,in=0] (6.5,.5);
\fill[white, name intersections={of=s5 and s6}] (intersection-1) circle(0.1);
\draw[braid] (6.7,2) to[out=270,in=180] (6.5,.5);
\draw[braid] (t2) to[out=315,in=135] (t1);
\draw[braid] (t2) to[out=225,in=90] (5.2,.3);
\draw[braid] (t1) to[out=225,in=90] (5.8,.3);
\draw[braid] (6.5,.5) to[out=270,in=90] (6.5,.3);
\draw (7,1.2) node {$=$};

\path (7.8,1.5) node[arr,name=t2l]  {$t_4$} 
(8.3,1) node[arr,name=t3l] {$t_3$};
\draw[braid] (7.5,2) to[out=270,in=135] (t2l);
\draw[braid] (8,2) to[out=270,in=45] (t2l);
\path[braid,name path=s1] (8.9,2) to[out=270,in=45] (t3l);
\draw[braid,name path=s2] (8.5,2) to[out=270,in=0] (8.7,.8) to[out=180,in=315]
(t3l); 
\fill[white, name intersections={of=s1 and s2}] (intersection-1) circle(0.1);
\draw[braid] (8.9,2) to[out=270,in=45] (t3l);
\draw[braid] (t2l) to[out=315,in=135] (t3l);
\draw[braid] (t2l) to[out=225,in=90] (7.5,.3);
\draw[braid] (t3l) to[out=225,in=90] (8.1,.3);
\draw[braid] (8.7,.8) to[out=270,in=90] (8.7,.3);
\end{tikzpicture}
$$
By the non-degeneracy condition on the multiplication this implies 
\eqref{eq:interchange_34}.
\end{proof} 

\begin{lemma}\label{lem:induced_multiplication}
Let $A$ be a semigroup in $\mathsf C$ whose multiplication $m$ is
non-degenerate with respect to some class containing $I$ and $A$. Let
$t_1,t_2,t_3,t_4\colon A^2\to A^2$ be morphisms satisfying the conditions of
Lemma \ref{lem:compatibility}. For any morphism $j\colon A\to I$, consider the
following conditions.  
\begin{center} 
\begin{tabular}{rr}
\inlineequation[eq:multiplication_1]{
\raisebox{-22pt}{$\begin{tikzpicture}[scale=.95]
\draw[braid] (1,2) to[out=270,in=180] (1.3,1) to[out=0,in=270] (1.6,2);
\draw[braid] (1.3,1) to[out=270,in=90] (1.3,.5);
\draw (2.1,1.5) node {$=$};
\path (3,1.5) node[arr,name=t] {$t_1$}
(2.5,1) node[arr,name=n] {$\ $};
\draw[braid] (2.7,2) to[out=270,in=135] (t);
\draw[braid] (3.3,2) to[out=270,in=45] (t);
\draw[braid] (t) to[out=225,in=45] (n);
\draw[braid] (t) to[out=315,in=90] (3.3,.5);
\end{tikzpicture}$}} &
\inlineequation[eq:multiplication_2]{
\raisebox{-22pt}{$\begin{tikzpicture}[scale=.95]
\draw[braid] (1,2) to[out=270,in=180] (1.3,1) to[out=0,in=270] (1.6,2);
\draw[braid] (1.3,1) to[out=270,in=90] (1.3,.5);
\draw (2.1,1.5) node {$=$};
\path (3,1.5) node[arr,name=t2] {$t_2$}
(3.5,1) node[arr,name=u2] {$\ $};
\draw[braid] (2.7,2) to[out=270,in=135] (t2);
\draw[braid] (3.3,2) to[out=270,in=45] (t2);
\draw[braid] (t2) to[out=315,in=135] (u2);
\draw[braid] (t2) to[out=225,in=90] (2.7,.5);
\end{tikzpicture}$}}  \\
\\
\inlineequation[eq:multiplication_3]{
\raisebox{-22pt}{$\begin{tikzpicture}[scale=.95]
\draw[braid] (1,2) to[out=270,in=180] (1.3,1) to[out=0,in=270] (1.6,2);
\draw[braid] (1.3,1) to[out=270,in=90] (1.3,.5);
\draw (2.1,1.5) node {$=$};
\path (3,1.5) node[arr,name=t3] {$t_3$}
(2.5,1) node[arr,name=u3] {$\ $};
\path[braid,name path=s1] (2.7,2) to[out=270,in=45] (t3);
\draw[braid,name path=s2] (3.3,2) to[out=270,in=135] (t3);
\fill[white, name intersections={of=s1 and s2}] (intersection-1) circle(0.1);
\draw[braid] (2.7,2) to[out=270,in=45] (t3);
\draw[braid] (t3) to[out=225,in=45] (u3);
\draw[braid] (t3) to[out=315,in=90] (3.3,.5);
\end{tikzpicture}$}} &
\inlineequation[eq:multiplication_4]{
\raisebox{-22pt}{$\quad \begin{tikzpicture}[scale=.95]
\draw[braid] (1,2) to[out=270,in=180] (1.3,1) to[out=0,in=270] (1.6,2);
\draw[braid] (1.3,1) to[out=270,in=90] (1.3,.5);
\draw (2.1,1.5) node {$=$};
\path (3,1.5) node[arr,name=t4] {$t_4$}
(3.5,1) node[arr,name=u4] {$\ $};
\path[braid,name path=s3] (2.7,2) to[out=270,in=45] (t4);
\draw[braid,name path=s4] (3.3,2) to[out=270,in=135] (t4);
\fill[white, name intersections={of=s3 and s4}] (intersection-1) circle(0.1);
\draw[braid] (2.7,2) to[out=270,in=45] (t4);
\draw[braid] (t4) to[out=315,in=135] (u4);
\draw[braid] (t4) to[out=225,in=90] (2.7,.5);
\end{tikzpicture}$}} 
\end{tabular}
\end{center}

Then the following hold.
\begin{itemize}
\item[{(1)}] Conditions \eqref{eq:multiplication_1} and
  \eqref{eq:multiplication_3} are equivalent to each other.
\item[{(2)}] Conditions \eqref{eq:multiplication_2} and
  \eqref{eq:multiplication_4} are equivalent to each other.
\end{itemize}
\end{lemma}

\begin{proof}
Once again, we only prove that \eqref{eq:multiplication_1} implies
\eqref{eq:multiplication_3}; all other implications follow symmetrically.  

By \eqref{eq:c13}, \eqref{eq:multiplication_1}, and associativity of the
multiplication, 
$$
\begin{tikzpicture}
\path (1,1.5) node[arr,name=t3] {$t_3$}
(.5,1) node[arr,name=j] {$\ $};
\path[braid,name path=s1] (.7,2) to[out=270,in=45] (t3);
\draw[braid,name path=s2] (1.3,2) to[out=270,in=135] (t3);
\fill[white, name intersections={of=s1 and s2}] (intersection-1) circle(0.1);
\draw[braid] (.7,2) to[out=270,in=45] (t3);
\draw[braid] (t3) to[out=225,in=45] (j);
\draw[braid] (t3) to[out=315,in=180] (1.5,.8) to[out=0,in=270] (1.7,2);
\draw[braid] (1.5,.8) to[out=270,in=90] (1.5,.5);
\path (2.2,1.4) node {$=$};

\path (3,1.5) node[arr,name=t1] {$t_1$}
(2.5,1) node[arr,name=j] {$\ $};
\draw[braid] (t1) to[out=225,in=45] (j);
\draw[braid,name path=s3] (3.2,2) to[out=270,in=135] (t1);
\draw[braid,name path=s4] (3.6,2) to[out=270,in=45] (t1);
\draw[braid,name path=s5] (t1) to[out=315,in=0] (3.5,.8);
\path[braid,name path=s6] (2.8,2) to[out=270,in=90] (3.6,1.5)
to[out=270,in=180] (3.5,.8);
\fill[white, name intersections={of=s3 and s6}] (intersection-1) circle(0.1);
\fill[white, name intersections={of=s4 and s6}] (intersection-1) circle(0.1);
\fill[white, name intersections={of=s5 and s6}] (intersection-1) circle(0.1);
\draw[braid] (2.8,2) to[out=270,in=90] (3.6,1.5)
to[out=270,in=180] (3.5,.8);
\draw[braid] (3.5,.8) to[out=270,in=90] (3.5,.5);
\path (4.2,1.4) node {$=$};

\draw[braid] (5.4,2) to[out=270,in=180] (5.6,1.5) to[out=0,in=270] (5.8,2);
\draw[braid] (4.9,2) to[out=270,in=180] (5.2,1) to[out=0,in=270] (5.6,1.5);
\draw[braid] (5.2,1) to[out=270,in=270] (5.2,.5);
\path (6.4,1.4) node {$=$};

\draw[braid] (7,2) to[out=270,in=180] (7.2,1.5) to[out=0,in=270] (7.4,2);
\draw[braid] (7.2,1.5) to[out=270,in=180] (7.6,1) to[out=0,in=270] (7.9,2);
\draw[braid] (7.6,1) to[out=270,in=270] (7.6,.5);
\end{tikzpicture}
$$
By the non-degeneracy condition on the multiplication this implies
\eqref{eq:multiplication_3}. 
\end{proof}

\begin{corollary} \label{cor:induced_multiplication}
In the setting of Lemma \ref{lem:induced_multiplication}, the following
assertions are equivalent:
\begin{itemize}
\item[{(i)}] Conditions \eqref{eq:multiplication_1} and
  \eqref{eq:multiplication_2} hold;
\item[{(ii)}] Conditions \eqref{eq:multiplication_1} and
  \eqref{eq:multiplication_4} hold;
\item[{(iii)}] Conditions \eqref{eq:multiplication_2} and
  \eqref{eq:multiplication_3} hold;
\item[{(iv)}] Conditions \eqref{eq:multiplication_3} and
  \eqref{eq:multiplication_4} hold. 
\end{itemize}
\end{corollary}

\begin{lemma} \label{lem:fusion_morphisms}
Let $A$ be a semigroup in $\mathsf C$ whose multiplication $m$ is
non-degenerate with respect to  some class containing $I$, $A$ and $A^2$. Let
$t_1,t_2,t_3,t_4\colon A^2\to A^2$ be morphisms satisfying the conditions of
Lemma \ref{lem:compatibility}. For a morphism $j:A\to I$ satisfying the
conditions in Lemma \ref{lem:induced_multiplication}, and morphisms
$e_1,e_2\colon A^2\to A^2$ such that 
\begin{equation}\label{eq:e_components}
\begin{tikzpicture}
\path (1.5,1.5) node[arr,name=e1] {$e_1$};
\draw[braid,name path=s1] (.4,2) to[out=270,in=180] (.8,1) to[out=0,in=225]
(e1);
\path[braid,name path=s2] (.8,2) to[out=270,in=180] (1.3,1) to[out=0,in=315]
(e1);
\fill[white, name intersections={of=s1 and s2}] (intersection-1) circle(0.1);
\draw[braid] (.8,2) to[out=270,in=180] (1.3,1) to[out=0,in=315]
(e1);
\draw[braid] (.8,1) to[out=270,in=90] (.8,.5);
\draw[braid] (1.3,1) to[out=270,in=90] (1.3,.5);
\draw[braid] (1.2,2) to[out=270,in=135] (e1);
\draw[braid] (1.8,2) to[out=270,in=45] (e1);
\draw (2.5,1.3) node {$=$};
\path (3.5,1.5) node[arr,name=e2] {$e_2$};
\draw[braid,name path=s3] (e2) to[out=225,in=180] (3.7,1) to[out=0,in=270]
(4.2,2);
\path[braid,name path=s4] (e2) to[out=315,in=180] (4.2,1) to[out=0,in=270]
(4.6,2);
\fill[white, name intersections={of=s3 and s4}] (intersection-1) circle(0.1);
\draw[braid] (e2) to[out=315,in=180] (4.2,1) to[out=0,in=270]
(4.6,2);
\draw[braid] (3.7,1) to[out=270,in=90] (3.7,.5);
\draw[braid] (4.2,1) to[out=270,in=90] (4.2,.5);
\draw[braid] (3.2,2) to[out=270,in=135] (e2);
\draw[braid] (3.8,2) to[out=270,in=45] (e2);
\end{tikzpicture},
\end{equation}
the following assertions are equivalent: 
\begin{itemize}
\item[{(i)}] The morphisms $t_1$, $j$ and $e_1$ constitute a weakly counital
  fusion morphism in $\mathsf C$; 
\item[{(ii)}] The morphisms $t_2$, $j$ and $e_2$ constitute a weakly counital
  fusion morphism in $\mathsf C^{\mathsf{rev}}$;  
\item[{(iii)}] The morphisms $t_3$, $j$ and $e_2$ constitute a weakly
  counital fusion morphism in $\overline\sfc$; 
\item[{(iv)}] The morphisms $t_4$, $j$ and $e_1$ constitute a weakly
  counital fusion morphism in $\overline\sfc\rev$. 
\end{itemize}
\end{lemma}

\begin{proof}
We only prove the implications (i)$\Rightarrow$(ii) and (i)$\Rightarrow$(iv);
the rest of the claims follows symmetrically.

Let us begin with (i)$\Rightarrow$(iv). The multiplication in part (iv) is the
opposite 
$\xymatrix@C=15pt{
A^2 \ar[r]^-{c^{ -1}} &
A^2 \ar[r]^-m &
A}$ 
of the multiplication $m$ in part (i); cf. \eqref{eq:multiplication_4}. 
Axiom~\hyperlink{AxII}{\textnormal{II}} in part (iv) has the same form as in
part (i). Using \eqref{eq:c14}, both sides of Axioms 
\hyperlink{AxIV}{\textnormal{IV}}, \hyperlink{AxVII}{\textnormal{VII}}, and
\hyperlink{AxVIII}{\textnormal{VIII}} of part (iv) differ from the respective 
sides of Axioms  \hyperlink{AxIV}{\textnormal{IV}},
\hyperlink{AxVIII}{\textnormal{VIII}}, and \hyperlink{AxVII}{\textnormal{VII}}
of part (i) by braid isomorphisms.  
In order to verify the other axioms of part (iv) we have to use the
non-degeneracy conditions on the multiplication. Then
Axiom~\hyperlink{AxI}{\textnormal{I}} follows by 
\begin{eqnarray*}
&&

\end{eqnarray*}
where we have written above the equality sign the justification for 
the given step; the label $(a)$ stands for associativity of the 
multiplication. 

Axiom \hyperlink{AxIII}{\textnormal{III}} follows by 
$$
%
$$
where the label III refers to the axiom of that number.

Axiom \hyperlink{AxV}{\textnormal{V}} follows by 
$$
%
\ .
$$

The case of assertion (ii) is more involved; non-degeneracy of the
multiplication is used in the verification of each axiom. Then Axiom
\hyperlink{AxII}{\textnormal{II}} is immediate by \eqref{eq:e_components} and
Axiom \hyperlink{AxI}{\textnormal{II}} of part (i). Axiom
\hyperlink{AxI}{\textnormal{I}} follows by    
\begin{eqnarray*}
&&%
 \ .
$$
Hence Axiom \hyperlink{AxV}{\textnormal{V}} is equivalent to
\ref{par:A6.5}; thus it holds true. Similarly, making use of
\eqref{eq:e_components} and \eqref{eq:c24} Axiom 
\hyperlink{AxVI}{\textnormal{VI}} can be re-written in the equivalent form
$$
%
$$
which is identity \ref{par:A6.5} for the weakly counital fusion morphism in
part (iv); hence it holds true. Axiom \hyperlink{AxVII}{\textnormal{VII}}
follows by   
$$
%
$$
where the unlabelled equality is obtained by applying \ref{par:A5} to 
the weakly counital fusion morphism in part (iv) and using 
\eqref{eq:multiplication_4}.   
\end{proof}

\begin{definition} \label{def:reg_wmba}
A {\em regular weak multiplier bimonoid} in a braided monoidal category
$\mathsf C$ consists of morphisms $t_1,t_2,t_3,t_4\colon A^2 \to A^2$,
$e_1,e_2\colon A^2\to A^2$ and $j\colon A\to I$ satisfying the conditions in
Corollary \ref{cor:compatibility}, Lemma \ref{lem:interchange}, Corollary
\ref{cor:induced_multiplication} and Lemma \ref{lem:fusion_morphisms} (in
particular, non-degeneracy of the multiplication with respect to some class
containing $I$, $A$ and $A^2$ is required).  
\end{definition}

The notion of regular weak multiplier bimonoid is invariant under two kinds of
symmetry operations:

\begin{corollary}\label{cor:Z}
For morphisms $t_1,t_2,t_3,t_4,e_1,e_2\colon A^2 \to A^2$ and $j\colon A \to
I$ in a braided monoidal category $\mathsf C$, the following assertions are
equivalent. 
\begin{itemize}
\item $(t_1,t_2,t_3,t_4,e_1,e_2,j)$ is a regular weak multiplier bimonoid in
  $\mathsf C$;
\item
  $(c^{-1}.t_1.c,c^{-1}.t_2.c,c^{-1}.t_3.c,c^{-1}.t_4.c,c^{-1}.e_1.c,.c^{-1}.e_2.c,j)$
  is a regular weak multiplier bimonoid in $\mathsf C^{\mathsf{rev}}$.
\end{itemize}
\end{corollary}

\begin{corollary}\label{cor:Z2xZ2}
For morphisms $t_1,t_2,t_3,t_4,e_1,e_2\colon A^2 \to A^2$ and $j\colon A \to
I$ in a braided monoidal category $\mathsf C$, the following assertions are
equivalent. 
\begin{itemize}
\item $(t_1,t_2,t_3,t_4,e_1,e_2,j)$ is a regular weak multiplier bimonoid in
  $\mathsf C$;
\item $(t_2,t_1,t_4,t_3,e_2,e_1,j)$ is a regular weak multiplier bimonoid in
  $\mathsf C^{\mathsf{rev}}$;
\item $(c^{-1}.t_3.c,c^{-1}.t_4.c,c^{-1}.t_1.c,c^{-1}.t_2.c,c^{-1}.e_2.c,
  c^{-1}.e_1.c,j)$ is a regular weak multiplier bimonoid in
  $\overline{\mathsf C}^{\mathsf{rev}}$;
\item $(c^{-1}.t_4.c,c^{-1}.t_3.c,c^{-1}.t_2.c,c^{-1}.t_1.c,c^{-1}.e_1.c,
  c^{-1}.e_2.c,j)$ is a regular weak multiplier bimonoid in
  $\overline{\mathsf C}$.
\end{itemize}
We refer to the latter three as the {\em opposite-coopposite}, the {\em
opposite}, and the {\em coopposite} of $(t_1,t_2,t_3,t_4,e_1,e_2,j)$. 
\end{corollary}

For a regular weak multiplier bimonoid all identities in Appendix
\ref{app:strings} hold, as well as their opposite, coopposite, and 
opposite-coopposite versions. 

\begin{example}
If $(t_1,t_2,t_3,t_4,j)$ is a regular multiplier bimonoid in a braided
monoidal category $\mathsf C$ in the sense of \cite{BohmLack:braided_mba},
whose multiplication is non-degenerate with respect to  some class containing
$I$, $A$ and $A^2$, then with the identity morphism $A^2\to A^2$ as $e_1$ and
$e_2$ it is a regular weak multiplier bimonoid.
\end{example}

\begin{example}
A regular weak multiplier bialgebra over a field, in the sense of
\cite{BohmGomezTorrecillasLopezCentella:wmba}, is a regular weak multiplier
bimonoid in the symmetric monoidal category $\mathsf{Vect}$ of
vector spaces, see \cite[Theorem 1.2]{Bohm:wmba_comod}. However, not every
regular weak multiplier bimonoid in $\mathsf{Vect}$ in the sense of Definition
\ref{def:reg_wmba} is a regular weak multiplier bialgebra in the sense of 
\cite{BohmGomezTorrecillasLopezCentella:wmba,Bohm:wmba_comod}, as axiom
(vi) in \cite[Definition 1.1]{Bohm:wmba_comod} is not required in Definition
\ref{def:reg_wmba}. (We will consider an appropriate analogue of this
`missing' axiom in Section \ref{sec:modules}.)  
\end{example}


\section{The base objects} \label{sec:base}

The key feature of the generalization of bialgebras that are known as 
{\em weak bialgebras} is the structure of the categories of modules 
\cite{WHAII,BoCaJa:wba,PastroStreet}; this remains true in any 
braided monoidal category with split idempotents, not just the 
classical case of vector spaces. As in the case of ordinary 
bialgebras, these categories of modules are monoidal. 
However, in contrast to ordinary bialgebras, their monoidal structure 
is {\em not} lifted from the base category. The {\em base 
object} of a weak bialgebra is both a subalgebra and a quotient 
coalgebra; we shall usually call it $L$. These algebra and coalgebra 
structures obey the compatibility axioms of a {\em separable 
Frobenius algebra}. As a consequence, the category of $L$-bimodules 
(equivalently, $L$-bicomodules) is a monoidal category in which
the monoidal product is given by splitting a canonical idempotent
morphism on the monoidal product of the underlying objects. The 
category of modules over a weak bialgebra is monoidal via a  
lifting of this monoidal structure of the category of bi(co)modules over its
base object $L$. 

The above properties generalize nicely to {\em regular weak multiplier
bialgebras} over fields \cite{BohmGomezTorrecillasLopezCentella:wmba} 
with full comultiplication. For such a weak multiplier bialgebra $A$, 
the base object $L$ is no longer a subalgebra of $A$ (but it is a non-unital
subalgebra of its multiplier algebra). As it has no unit, it can no longer be
a separable Frobenius algebra. But it turns out to possess the more general
structure of coseparable coalgebra (hence it is a so-called {\em firm}
algebra, see \cite{BrKaWi:cosep_coalg}). This structure is enough for the
category of $L$-bicomodules (isomorphically, of firm $L$-bimodules) to have a
monoidal structure where, again, the monoidal product is given by splitting a
canonical idempotent morphism. A suitably defined category of $A$-modules  is 
monoidal via a lifting of this monoidal structure.   

The aim of this and of the next sections is to extend the above results to
{\em regular weak multiplier bimonoids} in nice enough braided 
monoidal categories. We assume that coequalizers exist in our base 
category $\mathsf C$ and that they are preserved by taking the monoidal product
with any object; this preservation assumption is automatic if the 
monoidal category is closed (see Section~\ref{sec:closed}). We will 
also need the technical assumption that the composite 
of regular epimorphisms in $\mathsf C$ is a regular epimorphism again. We
consider regular weak multiplier bimonoids $A$ in $\mathsf C$ whose
multiplication is non-degenerate with respect to  some class 
$\mathcal Y$ containing $I$, $A$ and $A^2$. In this section we look 
for further conditions under which $A$ has a canonical quotient $L$ 
which carries the structure of a coseparable comonoid. Based on this 
result, in Section \ref{sec:modules} we present conditions for 
suitable $A$-modules to constitute a monoidal category, whose
monoidal structure is lifted from the category of $L$-bicomodules.   

\begin{example}
In abelian categories coequalizers exist, and since every epimorphism 
is regular, the composite of regular epimorphism is again a regular 
epimorphism. Thus any braided monoidal closed abelian category 
satisfies our assumptions. Several examples of this type will be 
discussed in Section~\ref{sec:closed}:  the category of modules over 
a commutative ring (and in particular the category of vector spaces 
over a given field), and the category of vector spaces graded by a 
given group.  
\end{example}

\begin{example}\label{ex:regepi-born}
The symmetric monoidal closed category of complete bornological 
spaces is not abelian, but it does have coequalizers 
\cite[Section~1.3]{Meyer}. The regular epimorphisms are those 
$f\colon X\to Y$ for which both $f$ itself and the induced function 
$f_*$ between the bornologies are surjective; it follows easily that 
the regular epimorphisms are closed under composition. 
\end{example}

\begin{example}\label{ex:Hilb}
Let \Hilb be the category whose objects are complex Hilbert spaces 
and whose morphisms are the continuous linear maps (not required to 
preserve the inner product). This has finite limits and colimits, 
and is enriched over abelian groups, but is not an abelian category. 
It is easy to see that the monomorphisms are the injective maps, 
and that a monomorphism in \Hilb is regular if and only if its 
image is closed (equivalently, it factorizes as an isomorphism in 
\Hilb followed by an inner-product preserving injection). 
The epimorphisms are those maps whose codomain is the closure of
the image, so that the closure of the image allows any morphism to be
factorized as an epimorphism followed by a regular monomorphism. 
The cokernel $q\colon K\to Q$ of a morphism $f\colon H\to K$ can be
characterized  by the following properties:         
\begin{itemize}
\item $q$ is surjective;
\item $q.f$ is zero;
\item 
$f$ corestricts to an epimorphism $H\to \mathsf{Ker}(q)$. 
\end{itemize}
In particular, the regular epimorphisms are precisely the surjections, and
these are clearly closed under composition.   
Furthermore, both regular monomorphisms and regular epimorphisms are
always split: we can use orthogonal projection 
to construct  their left and right inverses, respectively. 

If $H$ and $K$ are Hilbert spaces, their tensor product $H\otimes K$ 
as vector spaces has an inner product, but is not in general 
complete. If we define $H\hat\otimes K$ to be its completion, we 
obtain a symmetric monoidal structure on \Hilb \cite[Propositions 
2.6.5 and 2.6.12]{Kadison/Ringrose:1983}. A cokernel diagram
$
\xymatrix@C=18pt{
H \ar[r]|(.45){\,f\,} &
K \ar[r]|(.45){\,q\,} &
Q}
$
will be preserved by taking the monoidal product with any Hilbert space $L$
provided that $f \hat \otimes 1$ and $q \hat \otimes 1$ obey the properties in
the characterization of cokernels given above. The second property evidently
holds and the first one does because the monoidal product preserves regular
epimorphisms (since they are split). In order to verify the third property,
note that the monoidal product preserves epimorphisms as well. This in turn
follows from the fact that if $I$ is a dense linear subspace of a Hilbert
space $K$, then $I\otimes L$ is a dense linear subspace of $K\otimes L$, for
any Hilbert space $L$. So if $p\colon H \to Z$ is an epimorphism, then the
image of the equal paths around the diagram   
$$
\xymatrix{
H \otimes L \ \ar@{>->}[r] \ar[d]_-{p\otimes 1} &
H \hat \otimes L \ar[d]^-{p\hat \otimes 1} \\
Z\otimes L \ \ar@{>->}[r] &
Z \hat \otimes L }
$$
is dense in $Z\hat \otimes L$ so that also the image of the right column is
dense for any Hilbert space $L$. With this preservation of epimorphisms at
hand, we see that if $f$ corestricts to an epimorphism $p\colon H\to
\mathsf{Ker}(q)$ then $f\hat \otimes 1$ corestricts to an epimorphism $p\hat
\otimes 1\colon H\hat \otimes L \to \mathsf{Ker}(q)\hat \otimes
L=\mathsf{Ker}(q\hat \otimes 1)$, where the equality follows since the
inclusion $\mathsf{Ker}(q) \to K$ is a split monomorphism thus it is preserved
by taking the monoidal product with $L$.  
This proves that the monoidal product preserves cokernels (and hence
coequalizers). 
\end{example}

The non-degenerate morphisms in the categories of the above 
Examples will be investigated in Section \ref{sec:closed}. In 
particular, we shall see there that in the category of vector spaces
(even if graded by a given group), as well as in \Hilb, any morphism 
which is non-degenerate on either side with respect to the base 
field --- playing the role of the monoidal unit --- is non-degenerate 
on that side with respect to any object of the category in 
question.

If $(A, t_1,t_2,t_3,t_4,e_1,e_2,j)$ is a regular weak multiplier
bimonoid in a braided monoidal category, then the (idempotent) morphisms $e_1$
and $e_2$ can be regarded as the components of an \mM-morphism $e\colon I \nto
A^2$, cf. \eqref{eq:e_components}. Moreover, by Axiom 
\hyperlink{AxVII}{\textnormal{VII}} for the weakly counital fusion morphism
$(t_1,e_1,j)$ in $\mathsf C$, there is an \mM-morphism $\overline
\sqcap^R\colon A \nto A$ with components   
\begin{equation}\label{eq:pibarr}
\begin{tikzpicture}
\path (1,1.3) node {$\overline\sqcap^R_1:=$};
\path (2,1.5) node[arr,name=pibarr1] {$e_1$}
(2.5,1) node[arr,name=upibarr1] {$\ $};
\path[braid,name path=s1] (1.7,2) to[out=270,in=45] (pibarr1);
\draw[braid,name path=s2] (2.3,2) to[out=270,in=135] (pibarr1);
\fill[white, name intersections={of=s1 and s2}] (intersection-1) circle(0.1);
\draw[braid] (1.7,2) to[out=270,in=45] (pibarr1);
\draw[braid] (pibarr1) to[out=315,in=135] (upibarr1);
\draw[braid] (pibarr1) to[out=225,in=90] (1.7,.8);
\path (4,1.3) node {$\overline\sqcap^R_2:=$};
\path (5,1.5) node[arr,name=pibarr2] {$t_1$}
(5.5,1) node[arr,name=upibarr2] {$\ $};
\draw[braid] (4.7,2) to[out=270,in=135] (pibarr2);
\draw[braid] (5.3,2) to[out=270,in=45] (pibarr2);
\draw[braid] (pibarr2) to[out=315,in=135] (upibarr2);
\draw[braid] (pibarr2) to[out=225,in=90] (4.7,.8);
\end{tikzpicture}\ .
\end{equation}
Symmetrically, by Axiom \hyperlink{AxVII}{\textnormal{VII}} 
for the weakly counital fusion morphism $(t_2,e_2,j)$ in $\sfc\rev$, there is
an \mM-morphism $\overline \sqcap^L\colon A \nto A$ with components 
\begin{equation}\label{eq:pibarl}
\begin{tikzpicture}
\path (1,1.3) node {$\overline\sqcap^L_1:=$};
\path (2,1.5) node[arr,name=pibarl1] {$t_2$}
(1.5,1) node[arr,name=upibarl1] {$\ $};
\draw[braid] (1.7,2) to[out=270,in=135] (pibarl1);
\draw[braid] (2.3,2) to[out=270,in=45] (pibarl1);
\draw[braid] (pibarl1) to[out=225,in=45] (upibarl1);
\draw[braid] (pibarl1) to[out=315,in=90] (2.3,.8);
\path (4,1.3) node {$\overline\sqcap^L_2:=$};
\path (5,1.5) node[arr,name=pibarl2] {$e_2$}
(4.5,1) node[arr,name=upibarl2] {$\ $};
\path[braid,name path=s1] (4.7,2) to[out=270,in=45] (pibarl2);
\draw[braid,name path=s2] (5.3,2) to[out=270,in=135] (pibarl2);
\fill[white, name intersections={of=s1 and s2}] (intersection-1) circle(0.1);
\draw[braid] (4.7,2) to[out=270,in=45] (pibarl2);
\draw[braid] (pibarl2) to[out=225,in=45] (upibarl2);
\draw[braid] (pibarl2) to[out=315,in=90] (5.3,.8);
\end{tikzpicture}\ .
\end{equation}
By Axiom \hyperlink{AxVII}{\textnormal{VII}} for the weakly
counital fusion morphism $(t_4,e_1,j)$ in $\overline\sfc\rev$, 
there is an \mM-morphism $\sqcap^L\colon A \nto A$ with components
\begin{equation}\label{eq:pil}
\begin{tikzpicture}
\path (1,1.3) node {$\sqcap^L_1:=$};
\path (2,1.5) node[arr,name=pil1] {$e_1$}
(1.5,1) node[arr,name=upil1] {$\ $};
\draw[braid] (1.7,2) to[out=270,in=135] (pil1);
\draw[braid] (2.3,2) to[out=270,in=45] (pil1);
\draw[braid] (pil1) to[out=225,in=45] (upil1);
\draw[braid] (pil1) to[out=315,in=90] (2.3,.8);
\path (4,1.3) node {$\sqcap^L_2:=$};
\path (5,1.5) node[arr,name=pil2] {$t_4$}
(4.5,1) node[arr,name=upil2] {$\ $};
\path[braid,name path=s3] (4.7,2) to[out=270,in=45] (pil2);
\draw[braid,name path=s4] (5.3,2) to[out=270,in=135] (pil2);
\fill[white, name intersections={of=s3 and s4}] (intersection-1) circle(0.1);
\draw[braid] (4.7,2) to[out=270,in=45] (pil2);
\draw[braid] (pil2) to[out=225,in=45] (upil2);
\draw[braid] (pil2) to[out=315,in=90] (5.3,.8);
\end{tikzpicture}\ .
\end{equation}
Finally, by Axiom \hyperlink{AxVII}{\textnormal{VII}} for the
weakly counital fusion morphism $(t_3,e_2,j)$ in $\overline\sfc$, there is an
\mM-morphism $\sqcap^R:A \nto A$ with components
\begin{equation}\label{eq:pir} 
\begin{tikzpicture}
\path (1,1.3) node {$\sqcap^R_1:=$};
\path (2,1.5) node[arr,name=pir1] {$t_3$}
(2.5,1) node[arr,name=upir1] {$\ $};
\path[braid,name path=s3] (1.7,2) to[out=270,in=45] (pir1);
\draw[braid,name path=s4] (2.3,2) to[out=270,in=135] (pir1);
\fill[white, name intersections={of=s3 and s4}] (intersection-1) circle(0.1);
\draw[braid] (1.7,2) to[out=270,in=45] (pir1);
\draw[braid] (pir1) to[out=315,in=135] (upir1);
\draw[braid] (pir1) to[out=225,in=90] (1.7,.8);
\path (4,1.3) node {$\sqcap^R_2:=$};
\path (5,1.5) node[arr,name=pir2] {$e_2$}
(5.5,1) node[arr,name=upir2] {$\ $};
\draw[braid] (4.7,2) to[out=270,in=135] (pir2);
\draw[braid] (5.3,2) to[out=270,in=45] (pir2);
\draw[braid] (pir2) to[out=315,in=135] (upir2);
\draw[braid] (pir2) to[out=225,in=90] (4.7,.8);
\end{tikzpicture}\ .
\end{equation}
They will play crucial roles in the considerations of the paper. They
generalize simultaneously some important maps. In a regular 
multiplier bimonoid as in \cite{BohmLack:braided_mba} they all reduce 
to the composite of the counit $A \nto I$ and the \mM-morphism $I\nto 
A$ whose components are equal to the identity map $A\to A$ (cf. \cite[Theorem
2.11]{BohmGomezTorrecillasLopezCentella:wmba}). In particular, in an
ordinary bialgebra they reduce to the composite of the counit $A\to I$ with
the unit $I \to A$. In the regular weak multiplier bialgebra in \cite[Example
2.12]{BohmGomezTorrecillasLopezCentella:wmba}, spanned by the morphisms of an
arbitrary category, one half of them reduces to the (linear extension) of the
source map, the other half to the target map. If we think of the 
fusion morphisms encoding some generalized comultiplication as in 
\ref{par:d1_1}, then the above morphisms behave like generalized 
counits in the sense of \ref{par:A8} and its dual counterparts. 

The symmetries of Corollary \ref{cor:Z2xZ2} permute these morphisms: taking
the opposite corresponds to interchanging simultaneously the morphisms with
and without bar and the labels $1$ and $2$; taking the coopposite corresponds
to interchanging simultaneously the morphisms with and without bar and the
labels $L$ and $R$; finally, taking the opposite-coopposite corresponds to
interchanging simultaneously the labels $R$ with $L$ and $1$ with $2$. For
example, the opposite of $\sqcap^L_1$ is $\overline \sqcap^L_2$, the
coopposite is $\overline \sqcap^R_1$, while the opposite-coopposite is
$\sqcap^R_2$.   

Under the standing assumptions of the section, for a regular weak multiplier 
bimonoid $(t_1,t_2,t_3,t_4,e_1,e_2,j)$ with underlying object $A$ in $\mathsf
C$, consider the coequalizer
\begin{equation} \label{eq:L}
\xymatrix@C=45pt{
A^2 \ar@<2pt>[r]^-{\sqcap^L_1} 
\ar@<-2pt>[r]_-{\overline \sqcap^R_1.c^{-1}} &
A \ar[r]^-p &
L}
\end{equation}
in $\mathsf C$. We will refer to the object $L$ as the {\em base object} of
$A$; note that it is unique up to isomorphism. By the first equality of
\ref{par:A10} and \ref{par:A11}, and by \eqref{eq:components_compatibility} and
non-degeneracy of the multiplication, \eqref{eq:L} determines a unique
\mM-morphism $n\colon L\nto A$ with components occurring in the diagrams 
\begin{equation} \label{eq:n_1}
\xymatrix@C=35pt{
A^3 \ar@<2pt>[r]^-{\sqcap^L_1 1} 
\ar@<-2pt>[r]_-{\overline \sqcap^R_1 1.c^{-1}1} &
A^{2} \ar[r]^-{p1} \ar[rd]_-{\sqcap^{L}_{1}} &
LA \ar@{-->}[d]^-{n_{1}} \\
&& A }
\qquad
\xymatrix@C=35pt{
A^3 \ar@<2pt>[r]^-{1\sqcap^L_1} 
\ar@<-2pt>[r]_-{1\overline \sqcap^R_1.1c^{-1}} &
A^{2} \ar[r]^-{1p} \ar[rd]_-{\sqcap^{L}_{2}} &
AL \ar@{-->}[d]^-{n_{2}} \\
&& A .}
\end{equation}
Recall that if the multiplication $m$ of $A$ is non-degenerate with respect to
some class $\mathcal Y$ then $n_1$ is non-degenerate on the right
with respect to  $\mathcal Y$ if and only if $n_2$ is non-degenerate on the
left with respect to  $\mathcal Y$. 

\begin{theorem}\label{thm:L}
Let $\mathsf C$ be a braided monoidal category in which
coequalizers exist and are preserved by the monoidal product, and the
composite of regular epimorphisms is a regular epimorphism. Let
$(t_1,t_2,t_3,t_4,e_1,e_2,j)$ be a regular weak multiplier bimonoid in
$\mathsf C$ with underlying object $A$ such that its multiplication in Lemma
\ref{lem:induced_multiplication} is a regular epimorphism and non-degenerate
with respect to  some class $\mathcal Y$ containing $I$, $A$, $A^{2}$ and the
object $L$ from \eqref{eq:L}. Assume further that the morphism $n_{1}$ of
\eqref{eq:n_1} is non-degenerate on the right with respect to  $\mathcal
Y$. Then the following hold.    
\begin{itemize}
\item[{(1)}] There is an associative multiplication $\mu\colon L^2\to 
L$ with respect to which $n_{1}$ is an associative action. 
\item[{(2)}] There is a coassociative comultiplication $\delta\colon L \to L^2$
  rendering commutative
$$
\xymatrix{
A^2 \ar[r]^-m \ar[d]_-{t_1} &
A \ar[r]^-p &
L \ar@{-->}[d]^-\delta \\
A^2 \ar[rr]_-{pp} &&
L^2\ .}
$$
\item[{(3)}] The equality $\xymatrix@C=15pt{L\ar[r]^-\delta & L^2 \ar[r]^-\mu
  &L}=1$ holds.
\item[{(4)}] The comultiplication $\delta$ is a morphism of
  $L$-bimodules. That is, the following diagram commutes.  
$$
\xymatrix@R=10pt{
L^2 \ar[rr]^-{1\delta} \ar[dd]_-{\delta 1} \ar[rd]^-\mu &&
L^3 \ar[dd]^-{\mu 1} \\
& L \ar[rd]^-\delta \\
L^3 \ar[rr]_-{1\mu} &&
L^2}
$$
\item[{(5)}] The comultiplication $\delta$ admits a counit $\varepsilon$
satisfying $\varepsilon.p=j$. 
\end{itemize}
In particular, $L$ carries the structure of a coseparable comonoid.
\end{theorem}

\begin{proof} 
(1) The top row of
$$
\xymatrix@C=45pt{
LA^{2} \ar@<2pt>[r]^-{1\sqcap^{L}_{1}} 
\ar@<-2pt>[r]_-{1\overline\sqcap^{R}_{1}.1c^{-1}} &
LA \ar[r]^-{1p} \ar[d]^-{n_{1}} &
L^{2} \ar@{-->}[d]^-{\mu} \\
&
A \ar[r]_-{p} &
L}
$$
is a coequalizer and the left-bottom path coequalizes the parallel 
arrows of the top row by
$$
\begin{tikzpicture}[scale=1]
\path (1.25,1.9) node[arr,name=pibarr1] {${}_{\overline \sqcap^R_1}$}
(1,1.2) node[arr,name=pil1] {${}_{\sqcap^L_1}$}
(1,.5) node[arr,name=p] {$\ p\ $};
\draw[braid] (.5,2.5) to[out=270,in=135] (pil1);
\path[braid,name path=i3>pibarr1] (1.5,2.5) to[out=270,in=135] (pibarr1);
\draw[braid,name path=i2>pibarr1] (1,2.5) to[out=270,in=45] (pibarr1);
\fill[white, name intersections={of=i2>pibarr1 and i3>pibarr1}] (intersection-1) circle(0.1);
\draw[braid] (1.5,2.5) to[out=270,in=135] (pibarr1);
\draw[braid] (pibarr1) to[out=270,in=45] (pil1);
\draw[braid] (pil1) to[out=270,in=90] (p);
\draw[braid] (p) to[out=270,in=90] (1,0);
\draw (2,1.2) node[empty] {$\stackrel{~\ref{par:A10}}=$};
\end{tikzpicture}
\begin{tikzpicture}[scale=1]
\path (1.25,1.2) node[arr,name=pibarr1] {${}_{\overline \sqcap^R_1}$}
(.75,2.2) node[arr,name=pil1] {${}_{\sqcap^L_1}$}
(1.25,.5) node[arr,name=p] {$\ p\ $};
\draw[braid] (.5,2.5) to[out=270,in=135] (pil1);
\draw[braid] (1,2.5) to[out=270,in=45] (pil1);
\path[braid,name path=i3>pibarr1] (1.5,2.5) to[out=270,in=135] (pibarr1);
\draw[braid,name path=pil1>pibarr1] (pil1) to[out=270,in=45] (pibarr1);
\fill[white, name intersections={of=pil1>pibarr1 and i3>pibarr1}]
(intersection-1) circle(0.1); 
\draw[braid] (1.5,2.5) to[out=270,in=135] (pibarr1);
\draw[braid] (pibarr1) to[out=270,in=90] (p);
\draw[braid] (p) to[out=270,in=90] (1.25,0);
\draw (2.2,1.2) node[empty] {$\stackrel{~\eqref{eq:L}}=$};
\end{tikzpicture}
\begin{tikzpicture}[scale=1]
\path (1.25,1.25) node[arr,name=pil1d] {${}_{\sqcap^L_1}$}
(.75,2) node[arr,name=pil1u] {${}_{\sqcap^L_1}$}
(1.25,.5) node[arr,name=p] {$\ p\ $};
\draw[braid] (.5,2.5) to[out=270,in=135] (pil1u);
\draw[braid] (1,2.5) to[out=270,in=45] (pil1u);
\draw[braid] (1.5,2.5) to[out=270,in=45] (pil1d);
\draw[braid] (pil1u) to[out=270,in=135] (pil1d);
\draw[braid] (pil1d) to[out=270,in=90] (p);
\draw[braid] (p) to[out=270,in=90] (1.25,0);
\draw (2,1.2) node[empty] {$\stackrel{~\ref{par:A10}}=$};
\end{tikzpicture}
\begin{tikzpicture}[scale=1]
\path (.75,1.25) node[arr,name=pil1d] {${}_{\sqcap^L_1}$}
(1.25,2) node[arr,name=pil1u] {${}_{\sqcap^L_1}$}
(.75,.5) node[arr,name=p] {$\ p\ $};
\draw[braid] (.5,2.5) to[out=270,in=135] (pil1d);
\draw[braid] (1,2.5) to[out=270,in=135] (pil1u);
\draw[braid] (1.5,2.5) to[out=270,in=45] (pil1u);
\draw[braid] (pil1u) to[out=270,in=45] (pil1d);
\draw[braid] (pil1d) to[out=270,in=90] (p);
\draw[braid] (p) to[out=270,in=90] (.75,0);
\end{tikzpicture}
$$
and the fact that $\sqcap^L_1=n_{1}.p1$ and $p11\colon A^{3} \to LA^{2}$ is an
epimorphism. This proves the existence of a unique morphism $\mu$ as in the
diagram. It also renders commutative the diagrams 
$$
\xymatrix{
A^2 \ar[r]^-{pp} \ar[d]_-{\sqcap^L_{1}} &
L^{2}\ar[d]^-{\mu} \\
A \ar[r]_-{p} &
L}\quad
\xymatrix{
A^2 \ar[r]^-{pp} \ar[d]_-{\overline \sqcap^R_{1}.c^{-1}} &
L^{2}\ar[d]^-{\mu} \\
A \ar[r]_-{p} &
L .}\
$$
The stated associativity properties follow by \ref{par:A10} together with the
fact that $ppp\colon A^{3} \to L^{3}$ and $pp1\colon A^{3} \to L^{2}A$ are
epimorphisms. 

(2) The top row of the diagram of part (2) is a regular epimorphism by
assumption. It follows by \ref{par:A21} --- applied together with the
non-degeneracy conditions on the multiplication and $n_1$  --- that the
left-bottom path coequalizes those morphisms whose coequalizer is in the top
row. Thus the desired (unique) morphism $\delta$ exists by universality. 

In the diagrams
$$
\xymatrix@R=10pt{
A^3 \ar[r]^-{1m} \ar[dddd]_-{t_11} &
A^2 \ar[r]^-m \ar[dddd]^-{t_1} &
A \ar[r]^-p &
L \ar[dddd]^-\delta
&
A^3 \ar[r]^-{m1} \ar[d]_-{1t_1} &
A^2 \ar[r]^-m \ar[dddd]^-{t_1} &
A \ar[r]^-p &
L \ar[dddd]^-\delta \\
&&&
&
A^3 \ar[d]_-{c1} \\
&&&
&
A^3 \ar[d]_-{1t_1} \\
&&&
&
A^3 \ar[d]_-{c^{-1}1} \\
A^3 \ar[r]^-{1m} \ar[d]_-{1t_1} &
A^2 \ar[rr]^-{pp} &&
L^2 \ar[d]^-{1\delta} 
&
A^3 \ar[r]^-{m1} \ar[d]_-{t_11} &
A^2 \ar[rr]^-{pp} &&
L^2 \ar[d]^-{\delta 1} \\
A^3 \ar[rrr]_-{ppp} &&&
L^3
&
A^3 \ar[rrr]_-{ppp} &&&
L^3}
$$
the top-left regions commute by \ref{par:A1}, and by 
\eqref{eq:short} (short fusion equation) for the weakly counital 
fusion morphism $(t_1,e_1,j)$. All other regions commute by the
construction of $\delta$. The top rows are equal epimorphisms by the
associativity of $m$. Since the left verticals are equal by Axiom
\hyperlink{AxI}{\textnormal{I}} for $t_1$, this proves the
coassociativity of $\delta$.

(3) In the diagram 
$$
\xymatrix@R=10pt{
A^2 \ar[r]^-m \ar[d]_-{t_1} &
A \ar[r]^-p &
L \ar[d]^-\delta \\
A^2 \ar[rr]^-{pp} \ar[d]_-{\sqcap^L_1}&&
L^2 \ar[d]^-\mu\\
A \ar[rr]_-p &&
L}
$$
both regions commute by the constructions of $\delta$ and $\mu$,
respectively. By \ref{par:A8} the left-bottom path and the top row are equal
epimorphisms, which proves that the right vertical is the identity morphism.  

(4) In the second of the diagrams
$$
\xymatrix@R=10pt{
A^3 \ar[r]^-{1m} \ar[d]_-{1t_1} &
A^2 \ar[r]^-{pp} &
L^2 \ar[d]^-{1\delta} 
&
A^3 \ar[r]^-{1m} \ar[d]_-{\sqcap^L_1 1} &
A^2 \ar[r]^-{pp} \ar[d]^-{\sqcap^L_1} &
L^2 \ar[d]^-{\mu} \\
A^3 \ar[rr]^-{ppp} \ar[d]_-{\sqcap^L_1 1}&&
L^3 \ar[d]^-{\mu 1} 
&
A^2 \ar[r]^-m \ar[d]_-{t_1} &
A \ar[r]^- p &
L \ar[d]^-\delta \\
A^2 \ar[rr]_-{pp} &&
L^2
&
A^2 \ar[rr]_-{pp} &&
L^2}
$$
the top-left region commutes by \ref{par:A1} and all other regions commute by
the construction of $\mu$ or $\delta$. The top rows are epimorphisms and the
left verticals are equal by the first identity in \ref{par:A7}. Hence the
right verticals are equal, which  proves the left $L$-linearity of $\delta$.
Similarly, in the first of the diagrams
$$
\xymatrix@R=10pt{
A^4 \ar[r]^-{mm} \ar[d]_-{m11} &
A^2 \ar[dd]^-{\sqcap^L_1} \ar[r]^-{pp} &
L^2 \ar[dd]^-\mu 
&
A^4 \ar[r]^-{mm} \ar[d]_-{11m} &
A^2 \ar[r]^-{pp} &
L^2 \ar[dd]^-{\delta 1} \\
A^3 \ar[d]_-{\sqcap^L_1 1} &&
&
A^3 \ar[d]_-{t_1 1}\\
A^2 \ar[r]^-m \ar[d]_-{t_1} &
A \ar[r]^-p &
A \ar[d]^-\delta 
&
A^3  \ar[rr]^-{ppp} \ar[d]_-{1 \sqcap^L_1} &&
L^3 \ar[d]^-{1\mu}\\
A^2 \ar[rr]_-{pp} &&
L^2
&
A^2 \ar[rr]_-{pp} &&
L^2}
$$
the top-left region commutes by \ref{par:A1} and all other regions commute by
the construction of $\mu$ or $\delta$. The top rows are epimorphisms and the
left-bottom paths are equal by \ref{par:A22.5}, applied together with the
non-degeneracy conditions on $n_1$ and $n_2$, which proves the right 
$L$-linearity of $\delta$.  
 
(5) The morphism $j$ evidently coequalizes the parallel arrows of
\eqref{eq:L}, which proves the existence of $\varepsilon$ as in the claim. 
The diagrams
$$
\xymatrix{
A^2 \ar[r]^-m \ar[d]_-{t_1} &
A \ar[r]^-p &
L \ar[d]^-\delta
&
A^2 \ar[r]^-m \ar[d]_-{t_1} &
A \ar[r]^-p &
L \ar[d]^-\delta \\
A^2 \ar[rr]^-{pp} \ar[d]_-{j1}&&
L^2 \ar[d]^-{\varepsilon 1}
&
A^2 \ar[rr]^-{pp} \ar[d]_-{1j}&&
L^2 \ar[d]^-{1\varepsilon} \\
A  \ar[rr]_-p &&
L
&
A \ar[rr]_-p &&
L .}
$$
commute by definition of $\delta$ and $\varepsilon$. The left-bottom composite 
in the first diagram is equal to $p.m$ by \eqref{eq:multiplication_1}, while in 
the second diagram this is true by  the first equality in
\ref{par:A12} and the non-degeneracy of $n_2$. This proves that their 
right verticals are identity morphisms; that is, $\varepsilon$ is the 
counit of $\delta$.   
\end{proof}  

\begin{remark}\label{rem:alternative_delta}
Symmetrically to the construction of the comultiplication $\delta$ in the
proof of Theorem \ref{thm:L}, we can define it as the unique morphism
rendering commutative 
\begin{equation}\label{eq:alternative_delta} 
\xymatrix@R=15pt{
A^2 \ar[r]^-m \ar[d]_-c &
A \ar[r]^-p &
L \ar@{-->}[dd]^-\delta \\
A^2 \ar[d]_-{t_4} \\
A^2 \ar[rr]_-{pp} &&
L^2\ .}
\end{equation}
It follows by \ref{par:A20}, applied together with the non-degeneracy
conditions on $n_1$ and $n_2$, that this yields the same morphism
$\delta$.    
\end{remark}

\begin{remark} \label{rem:firm}
Let us recall from \cite{BrKaWi:cosep_coalg} that in a coseparable comonoid
$(L,\delta,\varepsilon)$, the bicomodule section $\mu$ of $\delta$ is a {\em
firm} multiplication in the sense of \cite{Quillen:firm}. That is, it is an
associative multiplication such that
$$
\xymatrix@C=35pt{
L^3 \ar@<2pt>[r]^-{\mu 1} \ar@<-2pt>[r]_-{1\mu} &
L^2 \ar[r]^-\mu &
L}
$$
is a coequalizer. Consequently \cite{BohmVerc:Morita,BohmGom-Tor:FirmFrob},
left $L$-comodules can be identified with firm left $L$-modules; that is,
with associative $L$-actions $\xi\colon LX \to X$ such that 
$$
\xymatrix@C=35pt{
L^2X \ar@<2pt>[r]^-{\mu 1} \ar@<-2pt>[r]_-{1\xi} &
LX \ar[r]^-\xi &
X}
$$
is a coequalizer. Explicitly, an $L$-coaction $\tau\colon X \to LX$ determines a
firm $L$-action
$$
\xi:=\xymatrix@C=25pt{
LX \ar[r]^-{1\tau} &
L^2X \ar[r]^-{\mu 1} &
LX \ar[r]^-{\varepsilon 1} &
X}
$$
and, conversely, a firm $L$-action $\xi\colon LX \to X$ determines a unique
$L$-coaction $\tau$ such that 
$$
\xymatrix{
LX \ar[r]^-\xi  \ar[d]_-{\delta 1} &
X \ar@{-->}[d]^-\tau \\
L^2X \ar[r]_-{1\xi} &
LX}
$$
commutes. There is a symmetric correspondence between right $L$-comodules and
firm right $L$-modules. 

Summarizing, an $L$-bicomodule is the same as a left and right firm
$L$-bimodule. To avoid the use of unnecessary multiple terminology, in 
this paper we will only speak about $L$-comodules (but keeping the above
correspondence in mind).
\end{remark}

\begin{remark} \label{rem:base_vec}
Consider a regular weak multiplier bialgebra $A$ over a field; this includes
the assumptions that the multiplication is surjective and non-degenerate. In
\cite{BohmGomezTorrecillasLopezCentella:wmba} the base object of $A$ was
defined as the image of the map $\sqcap^L$ from $A$ to its multiplier
algebra. We claim that --- whenever the comultiplication of $A$ is {\em left
full} in the sense of \cite[Theorem
3.13]{BohmGomezTorrecillasLopezCentella:wmba} --- this gives the same vector
space as the coequalizer $L$ in \eqref{eq:L}.  

In this case we can write $\sqcap^L_1(a\otimes b)=\sqcap^L(a)b$ and
$\overline \sqcap^R_1(a\otimes b)= \overline \sqcap^R(a)b$ for all $a,b\in
A$, in terms of maps $\sqcap^L$ and $\overline \sqcap^R$ from $A$ to the
multiplier algebra of $A$ (see \eqref{eq:components}). By \cite[Lemma
3.8]{BohmGomezTorrecillasLopezCentella:wmba} the map $\sqcap^L$ coequalizes
the parallel morphisms of \eqref{eq:L}. Hence there is a unique epimorphism
$f\colon L \to \mathsf{Im}(\sqcap^L)$ such that $\sqcap^L(a)=f(p(a))$ for all
$a\in A$. In order to see that $f$ is injective as well, we need to show that
$f(p(a))=0$ for some $a$ if and only if $p(a)=0$. Equivalently,
$\sqcap^L(a)=0$ if and only if $a$ belongs to the image of $\overline
\sqcap^R_1.c - \sqcap^L_1$. 

By \cite[Theorem 4.7 and Lemma 4.8]{BohmGomezTorrecillasLopezCentella:wmba},
the non-unital algebra $\mathsf{Im}(\overline \sqcap^R)$ possesses local
units; and by (a symmetric variant of) \cite[Proposition
5.2]{BohmGomezTorrecillasLopezCentella:wmba} $A$ is a firm module over it (in
the sense of \cite{Quillen:firm}). Hence for any element $a$ of $A$ there is
an associated element $b$ of $A$ such that $\overline \sqcap^R(b)a=a$. Choose
$a$ such that $\sqcap^L(a)=0$. Then
$$
a=\overline \sqcap^R(b)a=\overline \sqcap^R(b)a - \sqcap^L(a)b.
$$
So we conclude that in this situation the canonical surjection 
$A \to \mathsf{Im}(\sqcap^L)$ is the coequalizer in \eqref{eq:L}. 

In fact, as we shall see in Section \ref{sec:closed_L}, in the 
category of vector spaces the morphism $n_1$ of \eqref{eq:n_1} is 
non-degenerate on the right (with respect to the base field, 
equivalently, with respect to any vector space) if and only if $L$ in 
\eqref{eq:L} is isomorphic to the image of the map $\sqcap^L$. As 
the above considerations show, these properties hold whenever the 
comultiplication of $A$ is left full. 

In this sense Theorem \ref{thm:L} gives a new insight also to regular 
weak multiplier bialgebras over fields: It says that the assumption 
about the {\em fullness of the comultiplication} in \cite[Theorem 
4.7]{BohmGomezTorrecillasLopezCentella:wmba} can be replaced by the 
{\em non-degeneracy of $n_{1}$}. 
\end{remark}


\section{The monoidal category of modules} \label{sec:modules}

If $(t_1,t_2,t_3,t_4,e_1,e_2,j)$ is a regular weak multiplier
bimonoid in a braided monoidal category $\mathsf C$ such that all
assumptions of Theorem \ref{thm:L} hold, then Theorem \ref{thm:L} says
that the base object $L$ carries the structure of a coseparable
comonoid. Consequently, whenever idempotent morphisms in $\mathsf C$ split,
the category of $L$-bicomodules is monoidal. The monoidal unit is $L$, with
both coactions given by the comultiplication $\delta$. The monoidal product of
$L$-bicomodules $V$ and $W$ --- with left coactions denoted by $\tau$ and
right coactions denoted by $\overline \tau$ --- is the object $V \circ W$
occurring in the splitting  
$$
\xymatrix{
VW \ar@{->>}[r]^-{\hat s} &
V \circ W \ \ar@{>->}[r]^-{\check s} &
VW}
$$
of the idempotent morphism 
\begin{equation} \label{eq:s}
s:=\xymatrix{
VW\ar[r]^-{\overline \tau\tau} &
VL^2W \ar[r]^-{1\mu 1} &
VLW \ar[r]^-{1\varepsilon 1} &
VW\  .}
\end{equation}
(It can be regarded as the usual $L$-comodule tensor product, equivalently, as
the module tensor product of the firm $L$-modules $V$ and $W$, see Remark
\ref{rem:firm}.) 

The aim of this section is to see --- under the hypotheses of Theorem
\ref{thm:L}, and assuming that idempotent morphisms in $\mathsf C$ split ---
what else is needed for the category of modules (in an
appropriate sense, see below) over a regular weak multiplier bimonoid to be
monoidal via the lifting of this monoidal structure on the category of
$L$-bicomodules.   
 
Throughout this section, we assume that $\mathsf C$ is a braided
monoidal category in which coequalizers exist and are preserved by the
monoidal product, and that the composite of regular epimorphisms is a regular
epimorphism. We further assume that $(A,t_1,t_2,t_3,t_4,e_1,e_2,$ $j)$ is a
regular weak multiplier bimonoid in $\mathsf C$ whose multiplication $m=j1.t_1$ is 
a regular epimorphism and non-degenerate with respect to  some class $\mathcal Y$
containing $A$, the unit object $I$, and the object $L$ from 
\eqref{eq:L}, and closed under the monoidal product.
Finally, the morphism $n_1$ of \eqref{eq:n_1} is assumed to be 
non-degenerate on the right with respect to  $\mathcal Y$; 
equivalently, $n_2$ of \eqref{eq:n_1} is assumed to be 
non-degenerate on the left with respect to  $\mathcal Y$. 

Note that, without any loss of generality, we may assume that 
$\mathcal Y$ is closed under retracts. Indeed, if $i\colon Z\to X$ is 
a monomorphism preserved by the functor $(-)Q$ for any object $Q$ 
(as happens for example if $i$ has a left inverse), and some 
morphism $\nu\colon QV \to W$ is non-degenerate on the right with 
respect to $X$, then it is non-degenerate on the right with 
respect to $Z$ as well: from the composite
$$
\xymatrix{
PV \ar[r]^-{f1} &
ZQV \ar[r]^-{1\nu} &
ZW}
$$
we can uniquely recover any $f\colon P \to ZQ$ by post-composing  with $i1$; 
applying the non-degeneracy with respect to $X$ to recover $i1.f$; 
and then using the fact that $i1:ZQ\to XQ$ is a monomorphism.

For some results in the section, we will also need to assume that idempotent
morphisms in $\mathsf C$ split. This happens, for example, if any  morphism of
$\mathsf C$ admits an epi-mono factorization, or if \sfc has 
coequalizers or equalizers. 

\begin{example}\label{ex:regepi-mono-ab}
In an abelian category any morphism $f$ has an epi-mono factorization  
through the image of $f$. This includes in particular some examples from
Section \ref{sec:closed}: the category of modules over a commutative ring (so
in particular the category of vector spaces over a given field) and the
category of group-graded vector spaces. 
\end{example}

\begin{example}\label{ex:regepi-mono-bor}
In the (non-abelian) category of complete bornological vector spaces we 
can factorize any morphism $f\colon X\to Y$ through the image
$\mathsf{Im}(f)$, computed in the category of vector spaces. We may equip
$\mathsf{Im}(f)\cong X/\mathsf{Ker(f)}$ with the quotient bornology (which can
be different, however, from the subspace bornology of $\mathsf{Im}(f)\subseteq
Y$). Since it is complete, this gives a factorization of $f$ as a composite of
a regular epimorphism $X \to \mathsf{Im}(f)$ (see Example
\ref{ex:regepi-born}) and a monomorphism $\mathsf{Im}(f) \to Y$.  
\end{example}

\begin{example}
The category \Hilb is finitely complete and cocomplete, so 
idempotents split. 
\end{example}

\begin{proposition} \label{prop:L-coactions_on_A}
Under the standing assumptions of the section, 
the object $A$ carries the structure of a bicomodule over the comonoid $L$ in
Theorem \ref{thm:L}.
\end{proposition}

\begin{proof}
By \ref{par:A14} and the non-degeneracy conditions on the multiplication
and $n_2$, the left-bottom path of 
$$
\xymatrix{
A^2 \ar[r]^-m \ar[d]_-{t_1} &
A \ar@{-->}[d]^-\tau \\
A^2 \ar[r]_-{p1} & 
LA}
$$
coequalizes any pair of morphisms that the top row does. So we can use the
universality of the coequalizer in the top row to construct a left
$L$-coaction $\tau$ on $A$. By the constructions of $\tau$ and $\delta$, 
by the short fusion equation \eqref{eq:short} on $t_1$, by 
\ref{par:A1}, and by functoriality of the monoidal product, both diagrams 
$$
\xymatrix@R=10pt{
A^3 \ar[r]^-{m1} \ar[d]_-{1t_1} &
A^2 \ar[r]^-m \ar[dddd]^-{t_1} &
A \ar[dddd]^-\tau
&
A^3 \ar[r]^-{1m} \ar[dddd]_-{t_1 1} &
A^2 \ar[r]^-m \ar[dddd]^-{t_1} &
A \ar[dddd]^-\tau \\
A^3 \ar[d]_-{c1} \\
A^2 \ar[d]_-{1 t_1} \\
A^3 \ar[d]_-{c^{-1}1} \\
A^3 \ar[r]^-{m1} \ar[d]_-{t_11} &
A^2 \ar[r]^-{p1} &
LA \ar[d]^-{\delta 1} 
&
A^3 \ar[r]^-{1m} \ar[d]_-{1t_1} &
A^2 \ar[r]^-{p1} \ar[d]^-{1\tau} &
LA \ar[d]^-{1 \tau} \\
A^3 \ar[rr]_-{pp1} &&
L^2A 
&
A^3 \ar[r]_-{1p1} &
ALA \ar[r]_-{p11} &
L^2A}
$$
commute. Since their top rows are equal epimorphisms by the associativity of
$m$, and their left verticals are equal by Axiom~\hyperlink{AxI}{\textnormal{I}}, we conclude that $\tau$ is coassociative. Also 
$$
\xymatrix@R=10pt{
A^2 \ar[r]^-m \ar[d]_-{t_1} & 
A \ar[d]^-\tau \\
A^2 \ar[r]^-{p1} \ar[d]_-{j1} &
LA \ar[d]^-{\varepsilon 1} \\
A \ar@{=}[r] &
A}
$$
commutes thanks to the constructions of $\tau$ and $\varepsilon$. By 
\eqref{eq:multiplication_1} the left vertical is equal to the epimorphism in
the top row, proving that the right vertical is the identity morphism. That is
to say, $\tau$ is also counital. 

Similarly, by the coopposite of the first equality of \ref{par:A4} and by the
non-degeneracy of $n_1$ on the right with respect to  $A$, the left-bottom
path of  
$$
\xymatrix@R=10pt{
A^2 \ar[r]^-m \ar[d]_-c &
A \ar@{-->}[dd]^-{\overline \tau} \\
A^2 \ar[d]_-{t_4} \\
A^2 \ar[r]_-{1p} & 
AL}
$$
coequalizes any pair of morphisms that the top row does. So we can use the
universality of the coequalizer in the top row to construct a right
$L$-coaction $\overline \tau$ on $A$. By the construction of $\overline \tau$,
by \eqref{eq:alternative_delta}, the short fusion equation 
\eqref{eq:short} for $t_4$, naturality of the braiding, \ref{par:A1}, 
and by the functoriality of the monoidal product, both diagrams  
$$
\xymatrix@R=10pt{
A^3 \ar[r]^-{m1} \ar[d]_-{c_{A^2,A}} &
A^2 \ar[r]^-m \ar[d]^-c &
A \ar[ddddd]^-{\overline \tau} 
&
A^3 \ar[r]^-{1m} \ar[d]_-{c_{A,A^2}} &
A^2 \ar[r]^-m \ar[d]^-c &
A \ar[ddddd]^-{\overline \tau} \\
A^3 \ar[r]^-{1m} \ar[d]_-{1c} &
A^2 \ar[dddd]^-{t_4} &
&
A^3 \ar[r]^-{m1} \ar[d]_-{1c^{-1}} &
A^2 \ar[dddd]^-{t_4} \\
A^3 \ar[d]_-{t_4 1} &&
&
A^3 \ar[dd]_-{t_4 1} \\
A^3 \ar[d]_-{1c^{-1}}\\
A^3 \ar[d]_-{t_4 1} &&
&
A^3 \ar[d]_-{1c}\\
A^3 \ar[r]^-{1m} \ar[d]_-{1c} &
A^2 \ar[r]^-{1p} &
AL \ar[dd]^-{1\delta} 
&
A^3 \ar[r]^-{m1} \ar[d]_-{c1} &
A^2 \ar[r]^-{1p} \ar[dd]^-{\overline \tau 1} &
AL \ar[dd]^-{\overline \tau 1} \\
A^3 \ar[d]_-{1 t_4} &&
&
A^3 \ar[d]_-{t_4 1} \\
A^3 \ar[rr]_-{1pp} &&
AL^2
&
A^3 \ar[r]_-{1p1} &
ALA \ar[r]_-{11p} &
AL^2}
$$
commute. The top rows are equal epimorphisms by the associativity of $m$ and
the left verticals are equal by the coopposite of Axiom
\hyperlink{AxI}{\textnormal{I}}  (the fusion equation for $t_4$), so we
conclude that $\overline \tau$ is coassociative. The diagram  
$$
\xymatrix@R=10pt{
A^2 \ar[r]^-m \ar[d]_-c & 
A \ar[dd]^-{\overline \tau} \\
A^2 \ar[d]_-{t_4} \\
A^2 \ar[r]^-{1p} \ar[d]_-{1j} &
AL \ar[d]^-{1\varepsilon} \\
A \ar@{=}[r] &
A}
$$
is commutative by the constructions of $\overline \tau$ and $\varepsilon$. By
\eqref{eq:multiplication_4}, the left vertical is equal to the epimorphism in
the top row, proving that the right vertical is the identity morphism. Thus
$\overline \tau$ is counital. Finally, by the constructions of $\tau$ and
$\overline \tau$, by \ref{par:A1} and by the functoriality of the monoidal
product, the diagrams 
$$
\xymatrix@R=10pt{
A^3 \ar[r]^-{1m} \ar[d]_-{c1} &
A^2 \ar[r]^-m \ar[d]^-c &
A \ar[ddd]^-{\overline \tau} 
&
A^3 \ar[r]^-{1m} \ar[dd]_-{t_1 1} &
A^2 \ar[r]^-m \ar[dd]^-{t_1} &
A \ar[dd]^-\tau\\
A^3 \ar[d]_-{t_41} &
A^2 \ar[dd]^-{t_4} \\
A^3 \ar[d]_-{1c} &&
&
A^3 \ar[r]^-{1m} \ar[d]_-{1c} &
A^2 \ar[r]^-{p1} \ar[dd]^-{1 \overline \tau} &
LA \ar[dd]^-{1 \overline \tau} \\
A^3 \ar[r]^-{m1} \ar[d]_-{t_1 1} &
A^2 \ar[r]^-{1p} \ar[d]^-{\tau 1} &
AL \ar[d]^-{\tau 1}
&
A^3 \ar[d]_-{1 t_4} \\
A^3 \ar[r]_-{p11} &
LA^2 \ar[r]_-{11p} &
LAL
&
A^3 \ar[r]_-{11p} &
A^2L \ar[r]_-{p11} &
LAL}
$$
commute. Their top rows are equal epimorphisms by the associativity of $m$ and
their left-bottom paths are equal by \ref{par:A22} applied together with the
non-degeneracy conditions on $n_1$ and $n_2$. This proves that the left and
right coactions $\tau$ and $\overline \tau$ on $A$ commute.  
\end{proof}

\begin{definition} 
By a {\em module} over a semigroup $(M,m)$ we mean an object $V$ in $\mathcal Y$
together with a morphism $v\colon MV \to V$ --- called the {\em action} ---
subject to the following conditions.  
\begin{itemize}
\item $v$ is associative; that is, the following diagram commutes.
$$
\xymatrix{
M^2V \ar[r]^-{1v} \ar[d]_-{m1} &
MV \ar[d]^-v \\
MV \ar[r]_-v &
V}
$$
\item $v$ is a regular epimorphism.
\item $v$ is non-degenerate on the left with respect to the class $\mathcal Y$.
\end{itemize}
A {\em morphism of modules} is a morphism $f\colon V\to V'$ which is
compatible with the actions in the sense of the commutative diagram
$$
\xymatrix{
MV \ar[r]^-{1f}\ar[d]_-v &
MV'\ar[d]^-{v'} \\
V \ar[r]_-f &
V'\ .}
$$
\end{definition} 

\begin{theorem} \label{thm:L-coactions_on_modules}
Under the standing assumptions of the section, for any 
module $v\colon AV \to V$ of the semigroup $A$, the object $V$ admits the
structure of a bicomodule over the comonoid $L$ in Theorem \ref{thm:L}. Any
morphism of $A$-modules is a bicomodule morphism with respect to these
$L$-coactions. Hence there is a functor $U$ from the category of $A$-modules
to the category of $L$-bicomodules which acts on the morphisms as the identity
map.
\end{theorem}

\begin{proof}
In the diagram
$$
\xymatrix@R=10pt{
A^3 \ar[dd]_-{t_3 1} &&&
A^4 \ar[r]^-{11m} \ar[d]^-{c^{-1}11} \ar[lll]_-{11m} &
A^3 \ar[d]^-{c^{-1}1} \\
&&& A^4 \ar[d]^-{m11} &
A^3 \ar[dd]^-{m1} \\
A^3 \ar[d]_-{c^{-1}1} &&
A^3 \ar[d]^-{1p1} \ar[ld]_-{\sqcap^L_2 1} &
A^3 \ar[d]^-{1m} \ar[l]_-{1t_1} \\
A^3 \ar[r]_-{1\overline \sqcap^L_1} &
A^2 &
ALA \ar[l]^-{n_2 1} &
A^2 \ar[l]^-{1\tau} \ar@{=}[r] &A^2}
$$
the large region on the left commutes by \ref{par:A14}. All other 
regions commute by the construction of $\tau$ and functoriality of 
the monoidal product. Since those of the top row are equal 
epimorphisms, we deduce the equality of the left-bottom 
and right-bottom paths. Using it in the third equality, together with 
the associativity of $v$ in the first and the penultimate equalities, 
with \eqref{eq:components_compatibility} in the second equality, and 
with \eqref{eq:components_are_A-linear} in the last equality, we obtain  
$$
\begin{tikzpicture}[scale=.86]
\path (1.8,.4) node[arr,name=vd] {$\ v\ $}
(1.5,1.2) node[arr,name=pil2] {${}_{\overline \sqcap^L_2}$}
(.8,2) node[arr,name=t3] {$t_3$}
(2.2,2) node[arr,name=vu] {$\ v\ $};
\path[braid,name path=s1] (.6,2.5) to[out=270,in=45] (t3);
\draw[braid,name path=s2] (1,2.5) to[out=270,in=135] (t3);
\fill[white, name intersections={of=s2 and s1}] (intersection-1) circle(0.1);
\draw[braid] (.6,2.5) to[out=270,in=45] (t3);
\path[braid,name path=s3] (t3) to[out=315,in=90] (.8,0);
\path[braid,name path=s4] (1.6,2.5) to[out=270,in=135] (pil2);
\draw[braid,name path=s5] (t3) to[out=225,in=45] (pil2);
\fill[white, name intersections={of=s3 and s5}] (intersection-1) circle(0.1);
\fill[white, name intersections={of=s4 and s5}] (intersection-1) circle(0.1);
\draw[braid] (t3) to[out=315,in=90] (.6,0);
\draw[braid] (1.6,2.5) to[out=270,in=135] (pil2);
\draw[braid] (2,2.5) to[out=270,in=135] (vu);
\draw[braid] (2.4,2.5) to[out=270,in=45] (vu);
\draw[braid] (pil2) to[out=270,in=135] (vd);
\draw[braid] (vu) to[out=270,in=45] (vd);
\draw[braid] (vd) to[out=270,in=90] (1.8,0);
\path (3,1.4) node {$=$};

\path (4.8,.4) node[arr,name=vd] {$\ v\ $}
(4.5,1.2) node[arr,name=pil2] {${}_{\overline \sqcap^L_2}$}
(3.8,2) node[arr,name=t3] {$t_3$};
\path[braid,name path=s1] (3.6,2.5) to[out=270,in=45] (t3);
\draw[braid,name path=s2] (4,2.5) to[out=270,in=135] (t3);
\fill[white, name intersections={of=s2 and s1}] (intersection-1) circle(0.1);
\draw[braid] (3.6,2.5) to[out=270,in=45] (t3);
\path[braid,name path=s3] (t3) to[out=315,in=90] (3.8,0);
\path[braid,name path=s4] (4.6,2.5) to[out=270,in=135] (pil2);
\draw[braid,name path=s5] (t3) to[out=225,in=45] (pil2);
\fill[white, name intersections={of=s3 and s5}] (intersection-1) circle(0.1);
\fill[white, name intersections={of=s4 and s5}] (intersection-1) circle(0.1);
\draw[braid] (t3) to[out=315,in=90] (3.6,0);
\draw[braid] (4.6,2.5) to[out=270,in=135] (pil2);
\draw[braid] (pil2) to[out=270,in=180] (4.6,.7) to[out=0,in=270] (5,2.5);
\draw[braid] (4.6,.7) to[out=270,in=135] (vd);
\draw[braid] (5.4,2.5) to[out=270,in=45] (vd);
\draw[braid] (vd) to[out=270,in=90] (4.8,0);
\path (6,1.4) node {$=$};

\path (7.8,.4) node[arr,name=vd] {$\ v\ $}
(7.8,1.2) node[arr,name=pil1] {${}_{\overline \sqcap^L_1}$}
(6.8,2) node[arr,name=t3] {$t_3$};
\path[braid,name path=s1] (6.6,2.5) to[out=270,in=45] (t3);
\draw[braid,name path=s2] (7,2.5) to[out=270,in=135] (t3);
\fill[white, name intersections={of=s2 and s1}] (intersection-1) circle(0.1);
\draw[braid] (6.6,2.5) to[out=270,in=45] (t3);
\path[braid,name path=s3] (t3) to[out=315,in=90] (6.8,0);
\path[braid,name path=s4] (7.6,2.5) to[out=270,in=180] (7.6,.8)
to[out=0,in=270] (pil1); 
\draw[braid,name path=s5] (t3) to[out=225,in=135] (pil1);
\fill[white, name intersections={of=s3 and s5}] (intersection-1) circle(0.1);
\fill[white, name intersections={of=s4 and s5}] (intersection-1) circle(0.1);
\draw[braid] (t3) to[out=315,in=90] (6.8,0);
\draw[braid] (7.6,2.5) to[out=270,in=180] (7.6,.8) to[out=0,in=270] (pil1);
\draw[braid] (7.6,.8) to[out=270,in=135] (vd);
\draw[braid] (8,2.5) to[out=270,in=45] (pil1);
\draw[braid] (8.4,2.5) to[out=270,in=45] (vd);
\draw[braid] (vd) to[out=270,in=90] (7.8,0);
\path (9,1.4) node {$=$};

\path (11,.4) node[arr,name=v] {$\ v\ $}
(10,1.2) node[arr,name=n2] {$n_2$}
(10.7,1.8) node[arr,name=tau] {$\ \tau\ $};
\draw[braid] (9.6,2.5) to[out=270,in=180] (9.8,1.7) to[out=0,in=270] (10,2.5);
\draw[braid] (9.8,1.7) to[out=270,in=135] (n2);
\draw[braid] (10.7,2.5) to[out=270,in=90] (tau);
\draw[braid] (10.7,1) to[out=270,in=135] (v);
\draw[braid] (11.2,2.5) to[out=270,in=45] (v);
\path[braid,name path=s2] (10.4,2.5) to[out=270,in=180] (10.7,1)
to[out=0,in=315] (tau);
\draw[braid,name path=s6] (tau) to[out=225,in=45] (n2);
\fill[white, name intersections={of=s6 and s2}] (intersection-1) circle(0.1);
\draw[braid] (10.4,2.5) to[out=270,in=180] (10.7,1) to[out=0,in=315] (tau);
\draw[braid] (v) to[out=270,in=90] (11,0);
\draw[braid] (n2) to[out=270,in=90] (10,0);
\path (12,1.4) node {$=$};

\path (14,.4) node[arr,name=vd] {$\ v\ $}
(13,1.2) node[arr,name=n2] {$n_2$}
(14.2,1.2) node[arr,name=vu] {$\ v\ $}
(13.7,1.8) node[arr,name=tau] {$\ \tau\ $};
\draw[braid] (12.6,2.5) to[out=270,in=180] (12.8,1.7) to[out=0,in=270] (13,2.5);
\draw[braid] (12.8,1.7) to[out=270,in=135] (n2);
\draw[braid] (13.7,2.5) to[out=270,in=90] (tau);
\path[braid,name path=s2] (13.3,2.5) to[out=270,in=135] (vd);
\draw[braid] (14.4,2.5) to[out=270,in=45] (vu);
\draw[braid,name path=s6] (tau) to[out=225,in=45] (n2);
\fill[white, name intersections={of=s6 and s2}] (intersection-1) circle(0.1);
\draw[braid] (13.3,2.5) to[out=270,in=135] (vd);
\draw[braid] (vu) to[out=270,in=45] (vd);
\draw[braid] (tau) to[out=315,in=135] (vu);
\draw[braid] (vd) to[out=270,in=90] (14,0);
\draw[braid] (n2) to[out=270,in=90] (13,0);
\path (15,1.4) node {$=$};

\path (17,.4) node[arr,name=vd] {$\ v\ $}
(16,1.2) node[arr,name=n2] {$n_2$}
(17.2,1.2) node[arr,name=vu] {$\ v\ $}
(16.7,1.8) node[arr,name=tau] {$\ \tau\ $};
\draw[braid] (15.9,2.5) to[out=270,in=135] (n2);
\draw[braid] (16.7,2.5) to[out=270,in=90] (tau);
\path[braid,name path=s2] (16.3,2.5) to[out=270,in=135] (vd);
\draw[braid] (17.4,2.5) to[out=270,in=45] (vu);
\draw[braid,name path=s6] (tau) to[out=225,in=45] (n2);
\fill[white, name intersections={of=s6 and s2}] (intersection-1) circle(0.1);
\draw[braid] (16.3,2.5) to[out=270,in=135] (vd);
\draw[braid] (vu) to[out=270,in=45] (vd);
\draw[braid] (tau) to[out=315,in=135] (vu);
\draw[braid] (vd) to[out=270,in=90] (17,0);
\draw[braid] (15.6,2.5) to[out=270,in=180] (15.8,.7) to[out=0,in=270] (n2);
\draw[braid] (15.8,.7) to[out=270,in=90] (15.8,0);
\path (17.5,0) node {$.$};
\end{tikzpicture}
$$
By the non-degeneracy of $v$, $m$ and $n_2$ on the left with respect to
$\mathcal Y$, this proves that the left-bottom path of   
\begin{equation}\label{eq:tau_v}
\xymatrix{
AV \ar[r]^-v \ar[d]_-{\tau 1} &
V \ar@{-->}[d]^-\tau \\
LAV \ar[r]_-{1v} &
LV}
\end{equation}
coequalizes any pair of morphisms that the top row does. Thus we can use the
universality of the coequalizer in the top row of \eqref{eq:tau_v} to define a
left $L$-coaction $\tau$ on $V$. 

Similarly, using the fact that $m11\colon A^4\to A^3$ is epi, it 
follows from the construction of the right $L$-coaction $\overline 
\tau$ on $A$, equation \eqref{eq:components_are_A-linear} and the 
opposite of the second equality of \ref{par:A15} that  
$$
\xymatrix@R=10pt{
A^3 \ar[r]^-{1m} \ar[d]_-{1c} &
A^2 \ar[r]^-{\overline \tau 1} &
ALA \ar[d]^-{1c^{-1}} \\
A^3 \ar[d]_-{1t_3} &&
A^2L \ar[d]^-{1n_2} \\
A^3 \ar[r]_-{c^{-1}1} &
A^3 \ar[r]_-{\sqcap^R_1 1} &
A^2}
$$
commutes. With its help one derives
$$
\begin{tikzpicture}
\path (1.2,.8) node[arr,name=vd] {$\ v\ $}
(1,1.4) node[arr,name=pir2] {${}_{\sqcap^R_2}$}
(1.4,2) node[arr,name=t3] {$t_3$}
(2,2) node[arr,name=vu] {$\ v \ $};
\draw[braid] (.8,2.5) to[out=270,in=135] (pir2);
\draw[braid,name path=s4] (1.6,2.5) to[out=270,in=135] (t3);
\path[braid,name path=s3] (1.2,2.5) to[out=270,in=45] (t3);
\fill[white, name intersections={of=s3 and s4}] (intersection-1) circle(0.1);
\draw[braid] (1.2,2.5) to[out=270,in=45] (t3);
\draw[braid] (1.8,2.5) to[out=270,in=135] (vu);
\draw[braid] (2.2,2.5) to[out=270,in=45] (vu);
\draw[braid] (t3) to[out=225,in=45] (pir2);
\draw[braid] (pir2) to[out=270,in=135] (vd);
\path[braid,name path=s2] (t3) to[out=315,in=90] (2,0);
\draw[braid,name path=s1] (vu) to[out=270,in=45] (vd);
\fill[white, name intersections={of=s1 and s2}] (intersection-1) circle(0.1);
\draw[braid] (t3) to[out=315,in=90] (2,0);
\draw[braid] (vd) to[out=270,in=90] (1.2,0);
\path (2.5,1.4) node {$=$};
\end{tikzpicture}
\begin{tikzpicture} 
\path (.5,.4) node [arr,name=vd] {$\ v\ $}
(1.6,.9) node [arr,name=n2] {$n_2$}
(1.6,1.55) node [arr,name=vu] {$\ v\ $}
(1.6,2.1) node [arr,name=taub] {$\, \overline \tau \, $};
\draw[braid] (.4,2.5) to [out=270,in=135] (vd);
\path[braid,name path=i>n2d] (.8,2.5) to[out=270,in=180] (1.2,.5)
to[out=0,in=270] (n2);
\path[braid,name path=i>n2u] (1.2,2.5) to[out=270,in=135] (n2);
\draw[braid] (1.6,2.5) to[out=270,in=90] (taub);
\draw[braid,name path=i>vu] (2,2.5) to[out=270,in=45] (vu);
\draw[braid] (taub) to[out=225,in=135] (vu);
\draw[braid,name path=vu>vd] (vu) to[out=225,in=45] (vd);
\path[braid,name path=taub>n2] (taub) to[out=315,in=45] (n2);
\draw[braid] (vd) to[out=270,in=90] (.5,0);
\draw[braid] (1.2,.5) to[out=270,in=90] (1.2,0);
\fill[white, name intersections={of=vu>vd and i>n2d}] (intersection-1) circle(0.1);
\fill[white, name intersections={of=vu>vd and i>n2u}] (intersection-1) circle(0.1);
\draw[braid] (.8,2.5) to[out=270,in=180] (1.2,.5) to[out=0,in=270] (n2);
\draw[braid] (1.2,2.5) to[out=270,in=135] (n2);
\fill[white, name intersections={of=taub>n2 and i>vu}] (intersection-1) circle(0.1);
\draw[braid] (taub) to[out=315,in=45] (n2);
\end{tikzpicture} .
$$
By the non-degeneracy of $m$, $v$ and $n_2$ on the left with respect to  $\mathcal
Y$, this implies that the left-bottom path of  
\begin{equation}\label{eq:taubar_v}
\xymatrix@R=10pt{
AV \ar[r]^-v \ar[d]_-{\overline\tau 1} &
V \ar@{-->}[dd]^-{\overline \tau} \\
ALV \ar[d]_-{1c} \\
AVL \ar[r]_-{v1} &
VL}
\end{equation}
coequalizes any pair of morphisms that the top row does. Thus we can use the
universality of the coequalizer in the top row of \eqref{eq:taubar_v} to
define a right $L$-coaction $\overline\tau$ on $V$. 

The left $L$-coaction $\tau$ on $V$ is coassociative and counital by the
coassociativity and the counitality of the left $L$-coaction $\tau$ on $A$;
the right $L$-coaction $\overline \tau$ on $V$ is coassociative and counital
by the coassociativity and the counitality of the right $L$-coaction
$\overline \tau$ on $A$; and the coactions $\tau$ and $\overline \tau$ on $V$
commute since the coactions $\tau$ and $\overline \tau$ on $A$ do.

For any morphism $f\colon V\to V'$ of $A$-modules, in the diagrams
$$
\xymatrix@R=10pt{
V \ar[rrr]^-f \ar[dddd]_-\tau &&&
V'\ \ar[dddd]^-{\tau'}
&
V \ar[rrr]^-f \ar[dddd]_-{\overline \tau} &&&
V' \ar[dddd]^-{\overline \tau'} \\
& AV \ar[lu]_(.4)v \ar[r]^-{1f} \ar[dd]_-{\tau 1} &
AV' \ar[ru]^(.4){v'} \ar[dd]^-{\tau 1} &
&
& AV \ar[lu]_(.4)v \ar[r]^-{1f} \ar[d]^-{\overline \tau 1} &
AV' \ar[ru]^(.4){v'} \ar[d]_-{\overline \tau 1} \\
&&&
&
& ALV \ar[d]^-{1c} &
ALV' \ar[d]_-{1c} \\
& LAV \ar[ld]^(.4){1v} \ar[r]^-{11f} &
LAV' \ar[rd]_(.4){1v'} &
&
& AVL \ar[ld]^(.4){v1} \ar[r]^-{1f1} &
AV'L \ar[rd]_(.4){v'1} \\
LV \ar[rrr]_-{1f} &&& 
LV' &
VL \ar[rrr]_-{f1} &&&
V'L}
$$
the regions at the middle commute by the functoriality of the monoidal product
and the naturality of the braiding. The regions on the left and on the right
commute by the constructions of the coactions $\tau$ and $\overline \tau$. The
upper and lower regions commute because $f$ is a morphism of
$A$-modules. Since $v\colon AV \to V$ is epi, this proves that $f$ is a
morphism of left and right $L$-comodules.  
\end{proof}

Our final aim is to lift the monoidal structure on the category of 
$L$-bicomodules through the functor $U$ in 
Theorem~\ref{thm:L-coactions_on_modules} to give a monoidal structure 
on the category of $A$-modules.

\begin{proposition}\label{prop:act_on_L}
Under the standing assumptions of the section, the base 
object $L$ is an $A$-module via the action  
$$ 
\xymatrix{
AL \ar[r]^-{n_2} &
A \ar[r]^-p &
L\ ,}
$$
where $p$ and $n_2$ are defined as in \eqref{eq:L} and 
\eqref{eq:n_1}, respectively.  
\end{proposition} 

\begin{proof}
The object $L$ belongs to the class $\mathcal Y$ by assumption. The region at
the right of
\begin{equation} \label{eq:act_on_L}
\xymatrix{
A^2 \ar[d]_-{1p} \ar[rr]^-m \ar[rd]^-{\sqcap^L_2} &&
A \ar[d]^-p \\
AL \ar[r]_-{n_2} &
A \ar[r]_-p &
L}
\end{equation}
commutes by the second identity of \ref{par:A12} applied together with the
non-degeneracy of $n_2$ on the right.
The top-right path is a composite of regular epimorphisms, thus 
it is a regular epimorphism, hence so is the left-bottom path. 
Since the left column is an epimorphism, this proves that the
bottom row is a regular epimorphism. It follows immediately by 
commutativity of the diagram of  \eqref{eq:act_on_L} and the 
associativity of $m$ in the top row that the bottom row of 
\eqref{eq:act_on_L} is an associative action.  

It remains to see the non-degeneracy of the stated action on the left
with respect to  $\mathcal Y$. If $p1.n_21$ coequalizes $1f^1$ and $1f^2$ for
some morphisms $f^1$ and $f^2$ to $LY$ where $Y\in \mathcal Y$, then the
morphism  
$$
\begin{tikzpicture}
\path (1,3.5) node[arr,name=t1] {$t_1$}
(1.2,2.7) node[arr,name=fi] {$f^i$}
(.8,2.1) node[arr,name=n2] {$n_2$}
(.8,1.4) node[arr,name=p] {$\ p\ $}
(1.5,.7) node[arr,name=n1] {$n_1$};
\draw[braid] (.8,4) to[out=270,in=135] (t1);
\draw[braid] (1.2,4) to[out=270,in=45] (t1);
\draw[braid,name path=s1] (1.6,4) to[out=270,in=90] (fi);
\draw[braid,name path=fi>o] (fi) to[out=315,in=90] (2,.3);
\path[braid,name path=s2] (t1) to[out=315,in=45] (n1);
\fill[white, name intersections={of=fi>o and s2}] (intersection-1) circle(0.1);
\fill[white, name intersections={of=s1 and s2}] (intersection-1) circle(0.1);
\draw[braid] (t1) to[out=315,in=45] (n1);
\draw[braid] (t1) to[out=225,in=135] (n2);
\draw[braid] (fi) to[out=225,in=45] (n2);
\draw[braid] (n2) to[out=270,in=90] (p);
\draw[braid] (p) to[out=270,in=135] (n1);
\draw[braid] (n1) to[out=270,in=90] (1.5,.3);
\path (2.5,2.5) node {$=$};

\path (3.5,3.5) node[arr,name=t1] {$t_1$}
(4.3,3.5) node[arr,name=fi] {$f^i$}
(3.5,2.1) node[arr,name=n2] {$n_2$}
(4,1) node[arr,name=pil1] {${}_{\sqcap^L_1}$};
\draw[braid] (3.3,4) to[out=270,in=135] (t1);
\draw[braid] (3.7,4) to[out=270,in=45] (t1);
\draw[braid] (4.3,4) to[out=270,in=90] (fi);
\draw[braid,name path=s1] (fi) to[out=225,in=45] (n2);
\path[braid,name path=s2] (t1) to[out=315,in=45] (pil1);
\fill[white, name intersections={of=s1 and s2}] (intersection-1) circle(0.1);
\draw[braid] (fi) to[out=315,in=90] (4.5,.3);
\draw[braid] (t1) to[out=315,in=45] (pil1);
\draw[braid] (t1) to[out=225,in=135] (n2);
\draw[braid] (n2) to[out=270,in=135] (pil1);
\draw[braid] (pil1) to[out=270,in=90] (4,.3);
\path (5.4,2.5) node {$=$};

\path (7.6,3.5) node[arr,name=fi] {$f^i$}
(6.5,2.1) node[arr,name=n2] {$n_2$};
\draw[braid] (6.3,4) to[out=270,in=135] (n2);
\draw[braid] (7.6,4) to[out=270,in=90] (fi);
\path[braid,name path=s2] (6.7,4) to[out=315,in=0] (6.7,1) to[out=180,in=270]
(n2); 
\draw[braid] (6.7,1) to[out=270,in=90] (6.7,.3);
\draw[braid,name path=s1] (fi) to[out=225,in=45] (n2);
\fill[white, name intersections={of=s1 and s2}] (intersection-1) circle(0.1);
\draw[braid] (6.7,4) to[out=315,in=0] (6.7,1) to[out=180,in=270]
(n2); 
\draw[braid] (fi) to [out=315,in=90] (7.8,.3);
\end{tikzpicture}
$$
does not depend on $i\in \{1,2\}$. The second equality follows by
\ref{par:A20.5} using that $11p\colon A^3 \to A^2L$ is epi. By the
non-degeneracy of $m$ on the right and $n_2$ on the left with respect to
$\mathcal Y$, this implies $f^1=f^2$ hence non-degeneracy of the $A$-action on
$L$ on the left with respect to  $\mathcal Y$. 
\end{proof}

\begin{lemma}\label{lem:E}
Under the standing assumptions of the section, for any 
modules $v\colon AV \to V$ and $w\colon AW\to W$ over the semigroup 
$A$, the following diagrams commute.   
$$
\xymatrix{
A^2VW \ar[r]^-{1c1} \ar[d]_-{e_111} &
AVAW \ar[r]^-{vw} &
VW \ar[d]^-s
&
A^2VW \ar[d]_-{e_211} \ar[r]^-{11s} &
A^2VW \ar[r]^-{1c1} &
AVAW \ar[d]^-{vw}\\
A^2VW\ar[r]_-{1c1} &
AVAW \ar[r]_-{vw} &
VW
&
A^2VW \ar[r]_-{1c1} &
AVAW \ar[r]_-{vw} &
VW}
$$
where $s$ is the morphism \eqref{eq:s}.
\end{lemma}

\begin{proof}
The diagram
\begin{equation}\label{eq:E_VW}
\xymatrix{
A^2VW \ar[r]^-{1c1} \ar[d]_-{\overline \tau \tau 11} &
AVAW \ar[d]^-{\overline \tau 1 \tau 1} \ar[r]^-{vw} &
VW \ar[dd]^-{\overline \tau \tau} \\
AL^2AVW\ar[r]^-{11c_{LA,V}1} \ar[rd]_-{1c_{L^2A,V}1} \ar[dd]_-{1\mu111} &
ALVLAW \ar[d]^-{1c111} \\
& AVL^2AW \ar[r]^-{v11w} \ar[d]^-{11\mu 11} &
VL^2W \ar[d]^-{1\mu 1} \\
ALAVW \ar[r]^-{1c_{LA,V}1} \ar[d]_-{1\varepsilon 111} &
AVLAW \ar[d]^-{11\varepsilon 11} &
VLW \ar[d]^-{1\varepsilon 1} \\
A^2VW \ar[r]_-{1c1} &
AVAW \ar[r]_-{vw} &
VW}
\end{equation}
is commutative by the constructions of the $L$-coactions on $V$ and $W$, and
naturality and coherence of the braiding. Since $mm\colon A^4 \to A^2$ is epi,
commutativity of
\begin{equation} \label{eq:e_1_of_A^2}
\xymatrix{
A^2 \ar[ddd]_-{\overline \tau \tau} &
A^4 \ar[l]_-{mm} \ar[rr]^-{mm} \ar[d]^-{c11} &&
A^2 \ar[dd]^-{e_1} \\
& A^4 \ar[d]^-{t_4t_1} \\
& A^4 \ar[r]^-{1\sqcap^L_11} \ar[d]^-{1pp1} &
A^3 \ar[r]^-{1j1} \ar[d]^-{1p1} &
A^2 \ar@{=}[d] \\
AL^2A \ar@{=}[r] &
AL^2A \ar[r]_-{1\mu 1} &
ALA \ar[r]_-{1\varepsilon 1} &
A^2}
\end{equation}
implies that the left-bottom path of \eqref{eq:E_VW} is equal to $vw.1c1.e_1
11$. Hence so is the top-right path, proving commutativity of the first
diagram of the claim. In \eqref{eq:e_1_of_A^2} the 
region on the left commutes by the constructions of the coactions 
$\bar \tau$ and $\tau$, the middle region at the bottom commutes by 
the construction of $\mu$, the bottom-right region commutes by the 
construction of $\varepsilon$, and the top-right region commutes by 
the second equality of \ref{par:A4}. In order to prove commutativity 
of the second diagram of the claim, note that 
$b^{(2)}:=
\xymatrix@C=25pt{
A^2VW \ar[r]|(.46){\, 1c1\, } &
AVAW \ar[r]|(.52){\, vw\, } &
VW}$
is an associative action for the multiplication 
$m^{(2)}:=$\break
$\xymatrix@C=25pt{
A^4 \ar[r]|(.46){\, 1c1\,} &
A^4 \ar[r]|(.47){\, mm\,} &
A^2}$.
Hence the bottom regions of the diagram
$$
\xymatrix{
A^2VW \ar[dd]_-{e_211} &&
A^4VW \ar[ll]_-{11b^{(2)}} \ar[ld]_-{e_21111} \ar[rr]^-{11b^{(2)}} 
\ar[rd]^-{11e_111} &&
A^2VW \ar[dd]^-{11s} \\
& A^4VW \ar[ld]_-{11b^{(2)}} \ar[rd]^-{m^{(2)}11} &&
A^4VW \ar[rd]^-{11b^{(2)}} \ar[ld]_-{m^{(2)}11} \\
A^2VW \ar@{=}[d] &&
A^2 VW \ar[d]^-{b^{(2)}} &&
A^2VW \ar@{=}[d] \\
A^2VW \ar[rr]_-{b^{(2)}} &&
VW &&
A^2VW \ar[ll]^-{b^{(2)}}}
$$
commute. The region at the middle commutes by \eqref{eq:e_components}, the
triangle-shaped region at the top-left commutes by functoriality of the
monoidal product and the triangle-shaped region at the top-right commutes by
commutativity of the first diagram of the claim. Since the morphisms of the
top row are equal epimorphisms, this completes the proof. 
\end{proof}

\begin{lemma} \label{lem:b_0}
Under the standing assumptions of the section, for any 
modules $v\colon AV \to V$ and $w\colon AW\to W$ over the semigroup $A$, the
object $VW$ admits an associative action $b^0\colon AVW\to VW$ (which may not
obey the non-degeneracy conditions on a module), which renders commutative the
diagrams
$$
\xymatrix{
AVW \ar[r]^-{1s} \ar[rd]^-{b^0} \ar[d]_-{b^0}&
AVW \ar[d]^-{b^0}
&
A^3VW \ar[r]^-{d_211} \ar[d]_-{11b^0} &
A^2VW \ar[r]^-{1c1} &
AVAW \ar[d]^-{vw} \\
VW \ar[r]_-s & 
VW
&
A^2VW \ar[r]_-{1c1} &
AVAW \ar[r]_-{vw} &
VW}
$$
in which $s$ is the morphism \eqref{eq:s} and 
$
d_2:=\xymatrix@C=15pt{
A^3 \ar[r]^-{c^{-1}1} &
A^3 \ar[r]^-{1t_2} &
A^3 \ar[r]^-{c1} &
A^3 \ar[r]^-{1m} &
A^2}
$, of \ref{par:d1_1}.
\end{lemma}

\begin{proof}
For the morphism $d_1:=
\xymatrix@C=15pt{
A^3 \ar[r]^-{1c^{-1}} &
A^3 \ar[r]^-{t_11} &
A^3 \ar[r]^-{1c} &
A^3 \ar[r]^-{m1} &
A^2}$
of \ref{par:d1_1},
$$
\begin{tikzpicture}
\path (1,.8) node[arr,name=vd] {$\ v\ $}
(1.4,1.3) node[arr,name=vl] {$\ v\ $}
(2.4,1.3) node[arr,name=vr] {$\, w\, $}
(1.5,2) node[arr,name=d1] {$d_1$};
\draw[braid] (.8,2.5) to[out=270,in=135] (vd);
\draw[braid] (1.2,2.5) to[out=270,in=135] (d1);
\draw[braid] (1.5,2.5) to[out=270,in=90] (d1);
\draw[braid] (1.8,2.5) to[out=270,in=45] (d1);
\draw[braid] (2.6,2.5) to[out=270,in=45] (vr);
\draw[braid,name path=s1] (2.2,2.5) to[out=270,in=45] (vl);
\path[braid,name path=s2] (d1) to[out=315,in=135] (vr);
\fill[white, name intersections={of=s1 and s2}] (intersection-1) circle(0.1);
\draw[braid] (d1) to[out=315,in=135] (vr);
\draw[braid] (d1) to[out=225,in=135] (vl);
\draw[braid] (vl) to[out=270,in=45] (vd);
\draw[braid] (vd) to[out=270,in=90] (1,.5);
\draw[braid] (vr) to[out=270,in=90] (2.4,.5);
\draw (3.2,1.7) node[empty] {$\stackrel{~(a)}=$};

\path (4.4,1) node[arr,name=vl] {$\ v\ $}
(5.4,1) node[arr,name=vr] {$\, w\, $}
(4.5,2) node[arr,name=d1] {$d_1$};
\draw[braid] (3.8,2.5) to[out=270,in=180] (4,1.5) to[out=0,in=225] (d1); 
\draw[braid] (4,1.5) to[out=270,in=135] (vl);
\draw[braid] (4.2,2.5) to[out=270,in=135] (d1);
\draw[braid] (4.5,2.5) to[out=270,in=90] (d1);
\draw[braid] (4.8,2.5) to[out=270,in=45] (d1);
\draw[braid] (5.6,2.5) to[out=270,in=45] (vr);
\draw[braid,name path=s1] (5.2,2.5) to[out=270,in=45] (vl);
\path[braid,name path=s2] (d1) to[out=315,in=135] (vr);
\fill[white, name intersections={of=s1 and s2}] (intersection-1) circle(0.1);
\draw[braid] (d1) to[out=315,in=135] (vr);
\draw[braid] (vl) to[out=270,in=90] (4.4,.5);
\draw[braid] (vr) to[out=270,in=90] (5.4,.5);
\draw (6.2,1.7) node[empty] {$\stackrel{~\ref{par:d1_1}}=$};

\path (7.4,1) node[arr,name=vl] {$\ v\ $}
(8.3,1) node[arr,name=vr] {$\, w\, $}
(7,2) node[arr,name=t2] {$t_2$};
\draw[braid] (6.8,2.5) to[out=270,in=135] (t2);
\draw[braid] (7.2,2.5) to[out=270,in=45] (t2);
\draw[braid,name path=s3] (7.5,2.5) to[out=270,in=0] (7.1,1.5)
to[out=180,in=225](t2); 
\path[braid,name path=s4] (7.9,2.5) to[out=270,in=0] (7.5,1.5)
to[out=180,in=315](t2); 
\fill[white, name intersections={of=s3 and s4}] (intersection-1) circle(0.1);
\draw[braid] (7.9,2.5) to[out=270,in=0] (7.5,1.5) to[out=180,in=315](t2); 
\draw[braid] (8.7,2.5) to[out=270,in=45] (vr);
\draw[braid] (7.1,1.5) to[out=270,in=135] (vl);
\draw[braid,name path=s1] (8.3,2.5) to[out=270,in=45] (vl);
\path[braid,name path=s2] (7.5,1.5) to[out=270,in=135] (vr);
\fill[white, name intersections={of=s1 and s2}] (intersection-1) circle(0.1);
\draw[braid] (7.5,1.5) to[out=270,in=135] (vr);
\draw[braid] (vl) to[out=270,in=90] (7.4,.5);
\draw[braid] (vr) to[out=270,in=90] (8.3,.5);
\draw (9.2,1.7) node[empty] {$\stackrel{~(a)}=$};

\path (10,1) node[arr,name=vd] {$\ v\ $}
(10.8,2) node[arr,name=vu] {$\ v\ $}
(11.2,1) node[arr,name=wd] {$\, w\, $}
(11.6,2) node[arr,name=wu] {$\, w\, $}
(10,2) node[arr,name=t2] {$t_2$};
\draw[braid] (9.8,2.5) to[out=270,in=135] (t2);
\draw[braid] (10.2,2.5) to[out=270,in=45] (t2);
\draw[braid] (10.6,2.5) to[out=270,in=135] (vu);
\draw[braid] (11.8,2.5) to[out=270,in=45] (wu);
\draw[braid,name path=s3] (11.5,2.5) to[out=270,in=45] (vu); 
\path[braid,name path=s4] (10.9,2.5) to[out=270,in=135] (wu); 
\fill[white, name intersections={of=s3 and s4}] (intersection-1) circle(0.1);
\draw[braid] (10.9,2.5) to[out=270,in=135] (wu); 
\draw[braid,name path=s1] (vu) to[out=270,in=45] (vd);
\path[braid,name path=s2] (t2) to[out=315,in=135] (wd);
\fill[white, name intersections={of=s1 and s2}] (intersection-1) circle(0.1);
\draw[braid] (t2) to[out=315,in=135] (wd);
\draw[braid] (t2) to[out=225,in=135] (vd);
\draw[braid] (wu) to[out=270,in=45] (wd);
\draw[braid] (vd) to[out=270,in=90] (10,.5);
\draw[braid] (wd) to[out=270,in=90] (11.2,.5);
\end{tikzpicture}
$$
where we are now using $(a)$ for associativity of  the 
actions. By the non-degeneracy of $v$ on the left with respect to  
$\mathcal Y$, this implies that the morphism in the left-bottom path of 
$$
\xymatrix{
A^3VW \ar[r]^-{11c1} \ar[d]_-{d_111} & 
A^2VAW \ar[r]^-{1vw} & 
AVW \ar@{-->}[d]^-{b^0} \\
A^2VW \ar[r]_-{1c1} & 
AVAW \ar[r]_-{vw} & 
VW}
$$
coequalizes morphisms $1f$ and $1g$ whenever $vw.1c1$ coequalizes
$f$ and $g$. The top row is the image of the coequalizer $vw.1c1$ under the
functor $A(-)$ hence it is a coequalizer; we can use its universality 
to define $b^0$. It is associative by the associativity of $d_1$: see 
the third equality of \ref{par:d1_1}.  

Concerning the commutativity of the diagrams of the claim, we use again the
associative action 
$b^{(2)}:=
\xymatrix@C=25pt{
A^2VW \ar[r]|(.46){\, 1c1\, } &
AVAW \ar[r]|(.52){\, vw\, } &
VW}$
for the multiplication 
$m^{(2)}:=$\break
$\xymatrix@C=25pt{
A^4 \ar[r]|(.46){\, 1c1\,} &
A^4 \ar[r]|(.47){\, mm\,} &
A^2}$.
Observe that the diagrams 
$$
\xymatrix{
A^3VW \ar[r]^-{1b^{(2)}} \ar[d]_-{1e_111} &
AVW \ar[d]^-{1s}
&
A^3VW \ar[r]^-{1b^{(2)}} \ar[dd]_-{d_111} &
AVW \ar[dd]^-{b^0} 
&
A^3VW \ar[r]^-{1b^{(2)}} \ar[d]_-{d_111} &
AVW \ar[d]^-{b^0} \\
A^3VW \ar[r]^-{1b^{(2)}} \ar[d]_-{d_111} &
AVW \ar[d]^-{b^0} 
&&&
A^2VW \ar[r]^-{b^{(2)}} \ar[d]_-{e_111} &
AVW \ar[d]^-{s} \\
A^2VW \ar[r]_-{b^{(2)}} &
VW
&
A^2VW \ar[r]_-{b^{(2)}} &
VW
&
A^2VW \ar[r]_-{b^{(2)}} &
VW}
$$
commute by the construction of $b^0$ and commutativity of the first diagram of
Lemma \ref{lem:E}. Since the left verticals are equal by the fourth and fifth
equalities of \ref{par:d1_1}, and the top rows are equal epimorphisms, we
deduce the equality of the right verticals; that is, commutativity of the
first diagram of the claim. In the diagram 
$$
\xymatrix{
A^3VW \ar[dd]_-{d_211} &&
A^5VW \ar[ll]_-{111b^{(2)}} \ar[ld]_-{d_21111} \ar[rr]^-{111b^{(2)}} 
\ar[rd]^-{11d_111} &&
A^3VW \ar[dd]^-{11b^0} \\
& A^4VW \ar[ld]_-{11b^{(2)}} \ar[rd]^-{m^{(2)}11} &&
A^4VW \ar[rd]^-{11b^{(2)}} \ar[ld]_-{m^{(2)}11} \\
A^2VW \ar@{=}[d] &&
A^2 VW \ar[d]^-{b^{(2)}} &&
A^2VW \ar@{=}[d] \\
A^2VW \ar[rr]_-{b^{(2)}} &&
VW &&
A^2VW \ar[ll]^-{b^{(2)}}}
$$
the bottom regions commute by the associativity of $b^{(2)}$. The region at
the middle commutes by the second equality of \ref{par:d1_1}. The
triangle-shaped region at the top-left commutes by the functoriality of the
monoidal product and the triangle-shaped region at the top-right commutes by
the construction of $b^0$. Since the morphisms of the top row are equal
epimorphisms, we deduce the equality of the left-bottom and right-bottom
paths. This proves commutativity of the second diagram of the claim. 
\end{proof}

Assume now that idempotent morphisms in $\mathsf C$ split. Then in
particular $e_1$ --- which is equal to $s\colon A^2\to A^2$ of \eqref{eq:s} by
\eqref{eq:e_1_of_A^2} --- splits by some epimorphism $\hat e_1\colon A^2 \to
A\circ A$, via some monomorphism $\check e_1\colon A \circ A \to A^2$; 
and $e_2$ splits by some epimorphism $\hat e_2\colon A^2 \to A\bullet A$, via
some monomorphism $\check e_2\colon A \bullet A \to A^2$. By the last equality
of \ref{par:d1_1} and  its opposite-coopposite, and by 
universality of the equalizers in the bottom rows of  
\begin{equation}\label{eq:dhat} 
\xymatrix{
& A^3 \ar[d]^-{d_1} \ar@{-->}[ld]_-{\hat d_1} 
&&
& A^3 \ar[d]^-{d_2} \ar@{-->}[ld]_-{\hat d_2} \\
A \circ A \ar[r]_-{\check e_1} &
A^2 \ar@<2pt>[r]^-1 \ar@<-2pt>[r]_-{e_1} &
A^2
&
A \bullet A \ar[r]_-{\check e_2} &
A^2 \ar@<2pt>[r]^-1 \ar@<-2pt>[r]_-{e_2} &
A^2\ ,}
\end{equation}
there exist unique morphisms $\hat d_1$ (equal to $\hat e_1.d_1$, in fact) and
$\hat d_2$ (equal to $\hat e_2.d_2$, in fact) rendering commutative the
diagrams.    

\begin{proposition} \label{prop:b}
If we add to the standing assumptions of the section the requirements 
that idempotent morphisms in $\mathsf C$ split and that the morphisms 
$\hat d_1$ and $\hat d_2$ in \eqref{eq:dhat} are regular 
epimorphisms, then for any modules $v\colon AV \to V$ and $w\colon 
AW\to W$, the object $V\circ W$ admits an $A$-module structure too. 
Moreover, for any morphisms of $A$-modules $f\colon V \to V'$ and 
$W\to W'$, $f\circ g\colon V\circ W \to V'\circ W'$ is a morphism of 
$A$-modules too.
\end{proposition} 

\begin{proof}
By our assumption on $\mathcal Y$ being closed under the monoidal 
product, the object $VW$ belongs to the class $\mathcal Y$. Since
$\check s\colon  V\circ W \to VW$ is a split monomorphism, it follows by
our assumption on $\mathcal Y$ being closed under retracts that 
$V\circ W$ also belongs to $\mathcal Y$.  

Thanks to the commutativity of the first diagram of Lemma \ref{lem:b_0}, the
left-bottom path coequalizes the parallel morphisms of
$$
\xymatrix{
AVW \ar@<2pt>[r]^-{111} \ar@<-2pt>[r]_-{1s} &
AVW \ar[r]^-{1\hat s} \ar[d]_-{b^0} &
A(V\circ W) \ar@{-->}[d]^-b \\
& VW \ar[r]_-{\hat s} &
V\circ W\ .}
$$
Hence we can use the universality of the coequalizer in the top row to define
$b$. It is an associative $A$-action since $b^0$ is: see Lemma \ref{lem:b_0}. 

Let us see that $b$ is a regular epimorphism. In doing so, we use again the
shorthand notation  
$b^{(2)}:=
\xymatrix@C=25pt{
A^2VW \ar[r]|(.46){\, 1c1\, } &
AVAW \ar[r]|(.52){\, vw\, } &
VW}$.
In the first commutative diagram of
$$
\xymatrix@C=14pt{
A^2VW \ar[rr]^-{b^{(2)}} \ar[dd]_-{\hat e_1 11} \ar[rdd]^-{e_1 11} &&
VW \ar[dd]^-s \ar[r]^-{\hat s} &
V \circ W \ar@{=}[dd] 
&
A^3VW \ar[rr]^-{1b^{(2)}} \ar[d]^-{d_1 11} \ar@/_2pc/[dd]_-{\hat d_1 11} &&
AVW \ar[r]^-{1\hat s} \ar[d]^-{b^0} &
A(V\circ W) \ar[d]^-b \\
&&&&
A^2VW \ar[rr]^-{b^{(2)}} \ar[d]^-{\hat e_1 11} \ar[rd]^-{e_111} &&
VW \ar[r]^-{\hat s} \ar[d]^-s &
V\circ W \ar@{=}[d]\\
(A\circ A)VW \ar[r]_-{\check e_1 11} &
A^2 VW \ar[r]_-{b^{(2)}} &
VW \ar[r]_-{\hat s} &
V \circ W
&
(A\circ A) VW \ar[r]_-{\check e_1 11} &
A^2VW \ar[r]_-{b^{(2)}} &
VW \ar[r]_-{\hat s} &
V\circ W}
$$
(where the middle region commutes by Lemma \ref{lem:E}), the top row is a
composite of regular epimorphisms; hence a regular epimorphism. Since the left
column is epi, this shows that the bottom row is a regular epimorphism. Thus
also the left-bottom path of the second commutative diagram is a regular
epimorphism. Since the top row of the second diagram is an epimorphism, this
proves that $b$ is a regular epimorphism.  

It remains to prove the non-degeneracy of $b$ on the left with respect to
$\mathcal Y$. If $b1$ coequalizes the morphisms $1f$ and $1g$ to $A(V\circ
W)Y$ for some morphisms $f$ and $g$ to $(V\circ W)Y$ and $Y\in \mathcal Y$,
then the equal paths of the diagram  
$$
\xymatrix{
A^3(V\circ W)Y \ar@{=}[rr]\ar[rd]_(.3){111\check s1} \ar[ddd]_-{d_211} &&
A^3(V\circ W)Y \ar[r]^-{11b1} &
A^2(V\circ W)Y \ar[dd]^-{11\check s} \\
& A^3VWY \ar[r]^-{11b^01} \ar[ru]_(.7){111\hat s1} 
\ar@{=}[d] &
A^2VWY \ar[rd]_-{11s1} \ar[ru]^-{11\hat s1} \\
& A^3VWY \ar[rr]^-{11b^01} \ar[d]^-{d_2111} &&
A^2VWY \ar[d]^-{b^{(2)}1} \\
A^2(V\circ W)Y \ar[r]_-{11\check s1} &
A^2VWY \ar[rr]_-{b^{(2)}1} &&
VWY}
$$
(where both regions at the bottom-right commute by Lemma \ref{lem:b_0})
coequalize the morphisms $111f$ and $111g$ to $A^3(V\circ W)Y$. Then 
using the factorization \eqref{eq:dhat} of $d_2$, we conclude
that the equal paths around  
$$
\xymatrix{
A^2(V\circ W) Y \ar[r]^-{e_211} \ar[d]^-{11\check s1} 
\ar@/_2pc/[dd]_-{11\check s 1}&
A^2(V\circ W) Y \ar[d]^-{11\check s 1} \\
A^2VWY \ar[r]^-{e_2111} \ar[d]^-{11s1} &
A^2VWY \ar[d]^-{b^{(2)}1} \\
A^2VWY \ar[r]_-{b^{(2)}1} &
VWY}
$$
(where the bottom-right region commutes by Lemma \ref{lem:E})
coequalize the morphisms $11f$ and $11g$ to $A^2(V\circ W)Y$. Since 
$v$ and $w$ are non-degenerate on the left with respect to  $\mathcal Y$ and 
$\check s\colon V\circ W \to VW$ is a (split) monomorphism, this proves 
$f=g$ hence the non-degeneracy of $b$ on the left with respect to  $\mathcal Y$.

As for functoriality of the monoidal product $\circ$, since 
\raisebox{1.5pt}{$
\xymatrix@C=15pt{
A^3VW \ar[r]^-{1b^{(2)}} &
AVW \ar[r]^-{1\hat s} &
A(V\circ W)}$}
is an epimorphism, it follows by the commutativity of
$$
\xymatrix@R=10pt{
A(V\circ W) \ar[rrrrr]^-{1(f\circ g)} \ar[ddddd]_-b &&&&&
A(V'\circ W') \ar[ddddd]^-{b'} \\
& AVW \ar[lu]_(.4){1\hat s} \ar[rrr]^-{1fg} \ar[ddd]_-{b^0} &&&
AV'W' \ar[ru]^(.4){1\hat s'} \ar[ddd]^-{b^{\prime 0}} \\
&& A^3VW \ar[lu]_(.4){1b^{(2)}} \ar[r]^-{111fg} \ar[d]_-{d_1 11} &
A^3V'W' \ar[ru]^(.4) {1b^{\prime (2)}} \ar[d]^-{d_1 11} \\
&& A^2VW \ar[ld]^(.4){b^{(2)}} \ar[r]_-{11fg} &
A^2V'W' \ar[rd]_(.4){b^{\prime (2)}} \\
& VW \ar[ld]^(.4){\hat s} \ar[rrr]_-{fg} &&&
V'W'\ar[rd]_(.4){\hat s'} \\
V\circ W \ar[rrrrr]_-{f\circ g} &&&&&
V'\circ W'\ ,}
$$
for any morphisms $f$ and $g$ of $A$-modules, that $f\circ g$ is a morphism of
$A$-modules. 
\end{proof}

Observe that, in the category of vector spaces, the requirement that $\hat d_1$
and $\hat d_2$ be regular epimorphisms becomes axiom (iv) in
\cite[Definition 2.1]{BohmGomezTorrecillasLopezCentella:wmba}; equivalently,
axiom (vi) of \cite[Definition 1.1]{Bohm:wmba_comod} that we did not
require in Definition \ref{def:reg_wmba}.

\begin{lemma} \label{lem:alternative_b}
Under the same hypotheses as in Proposition \ref{prop:b}, the action
$b\colon A(V\circ W)\to V\circ W$, constructed in the proof of Proposition
\ref{prop:b}, admits the following equivalent characterizations.
\begin{itemize}
\item[{(1)}] $b$ is the unique morphism rendering commutative
$$
\xymatrix{
A^2VW \ar[r]^-{1c1} \ar[d]_-{t_111} &
AVAW \ar[r]^-{11w} &
AVW \ar[r]^-{1\hat s} &
A(V\circ W) \ar[d]^-b \\
A^2VW \ar[r]_-{1c1} &
AVAW \ar[r]_-{vw} &
VW \ar[r]_-{\hat s} &
V\circ W\ .}
$$
\item[{(2)}] $b$ is the unique morphism rendering commutative
$$
\xymatrix@R=15pt{
A^2VW \ar[rr]^-{1v1} \ar[d]_-{c11} &&
AVW \ar[r]^-{1\hat s} &
A(V\circ W) \ar[dd]^-b \\
A^2VW \ar[d]_-{t_411} \\
A^2VW \ar[r]_-{1c1} &
AVAW \ar[r]_-{vw} &
VW \ar[r]_-{\hat s} &
V\circ W\ .}
$$
\end{itemize}
\end{lemma}

\begin{proof}
The exterior of 
$$
\xymatrix@R=15pt{
A^3VW \ar[r]^-{11c1} \ar[d]^-{1c^{-1}11} \ar@/_2.5pc/[dddd]_-{d_111} &
A^2VAW \ar[r]^-{1v11} &
AVAW \ar[d]^-{1c^{-1}1} \ar[r]^-{11w} &
AVW \ar[r]^-{1\hat s} &
A(V\circ W) \ar[dddd]^-b \\
A^3VW \ar[rr]^-{11v1} \ar[d]^-{t_1111} &&
A^2VW \ar[d]^-{t_111} \\
A^3VW \ar[d]^-{1c11} \ar[rr]^-{11v1} &&
A^2VW \ar[d]^-{1c1} \\
A^3VW \ar[r]^-{11c1} \ar[d]^-{m111} &
A^2VAW \ar[r]^-{1v11} \ar[d]^-{m111} &
AVAW \ar[d]^-{v11} \\
A^2VW \ar[r]_-{1c1} &
AVAW \ar[r]_-{v11} &
VAW \ar[r]_-{1w} &
VW \ar[r]_-{\hat s} &
V\circ W}
$$
commutes by the constructions of $b^0$ and $b$. All of the small regions in
the left half commute by functoriality of the monoidal product, naturality and
coherence of the braiding, and associativity of $v$. Since 
\raisebox{1.5pt}{$ \xy
\POS (0,0)*{A^3VW}="1", 
     (23,0)*{A^2VAW}="2", 
     (46,-.5)*{AVAW}="3"
\POS  "1" \ar "2" |(.47){\, 11c1\, } 
\POS  "2" \ar "3" |(.48){\, 1v11\, } 
\endxy
$}
is epi, this proves commutativity of the large region on the right; that is,
assertion (1).

Similarly, applying \eqref{eq:c14}, we obtain the alternative expression 
$$
\xymatrix{
A^3 \ar[r]^-{c1} &
A^3 \ar[r]^-{t_4 1} &
A^3 \ar[r]^-{1m} &
A^2}
$$
of $d_1$. Using it, we deduce commutativity of the leftmost region of
$$
\xymatrix@R=15pt{
A^3VW \ar[r]^-{11c1} \ar[d]^-{c111} \ar@/_2.5pc/[dddd]_-{d_111} &
A^2VAW \ar[r]^-{111w} \ar[d]^-{c111} &
A^2VW \ar[r]^-{1v1} \ar[d]^-{c11} &
AVW \ar[r]^-{1\hat s} &
A(V\circ W) \ar[dddd]^-b \\
A^3VW \ar[r]^-{11c1} \ar[d]^-{t_4111} &
A^2VAW \ar[d]^-{t_4111} &
A^2VW \ar[d]^-{t_411} \\
A^3VW \ar[r]^-{11c1} \ar[dd]^-{1m11} &
A^2VAW \ar[d]^-{1c11} &
A^2VW \ar[d]^-{1c1} \\
& AVA^2W \ar[r]^-{111w} \ar[d]^-{11m1} &
AVAW \ar[d]^-{11w} \\
A^2VW \ar[r]_-{1c1} &
AVAW \ar[r]_-{11w} &
AVW \ar[r]_-{v1} &
VW \ar[r]_-{\hat s} &
V \circ W .}
$$
Since the regions in the left and central columns commute by functoriality of 
the monoidal product, naturality of the braiding and associativity of 
$w$, and since in the top row the composite 
\raisebox{2pt}{$\xy
\POS (0,0)*{A^3VW}="1", 
     (23,0)*{A^2VAW}="2", 
     (46,0)*{A^2VW}="3"
\POS  "1" \ar "2" |(.47){\, 11c1\, } 
\POS  "2" \ar "3" |(.48){\, 111w\, } 
\endxy
$}
is epi, this proves mutativity of the large region on the 
right; that is, part (2).
\end{proof}

\begin{theorem} \label{thm:module_cat}
If we add to the standing assumptions of the section that 
idempotent morphisms in $\mathsf C$ split, and that the morphisms 
$\hat d_1$ and $\hat d_2$ in \eqref{eq:dhat} are regular 
epimorphisms, then there is a monoidal structure on the category of 
$A$-modules for which the functor $U$ in Theorem 
\ref{thm:L-coactions_on_modules} is strict monoidal.  
\end{theorem}

\begin{proof}
In view of Propositions \ref{prop:act_on_L} and \ref{prop:b}, we only need to
show that the unit and associativity isomorphisms of the monoidal category of
$L$-bicomodules, if evaluated at $A$-modules, are $A$-module morphisms.

\begin{figure} 
\centering
\begin{sideways}
\scalebox{.8}{
\xymatrix@C=25pt@R=15pt{
A^4VWZ \ar[rr]^-{111v11} \ar[d]_-{t_111111} &&
A^3VWZ \ar[r]^-{11c11} \ar[d]_-{t_11111} &
A^2VAWZ \ar[rr]^-{111w1} &&
A^2VWZ \ar[r]^-{11 \hat s 1} &
A^2(V\circ W)Z \ar[r]^-{1c1} \ar[d]^-{t_111} &
A(V\circ W)AZ \ar[r]^-{11z} &
A(V\circ W)Z \ar[r]^-{1\hat s} &
A((V\circ W)\circ Z) \ar@{-->}[ddd] \\
A^4VWZ \ar[d]_-{1c_{A,A^2VW}1} &&
A^3VWZ \ar[r]^-{11c11} \ar[d]_-{1c_{A,AVW}1} &
A^2VAWZ \ar[rr]^-{111w1} 
&&
A^2VWZ \ar[r]^-{11\hat s 1} 
&
A^2(V\circ W)Z \ar[d]^-{1c1} \\
A^3VWAZ \ar[d]_-{t_111111} &&
A^2VWAZ \ar[r]^-{1c111} \ar[d]^-{t_11111} &
AVAWAZ \ar[rr]^-{11w11} &&
AVWAZ \ar[r]^-{1\hat s 11} &
A(V\circ W)AZ \ar[d]^-{bz} \\
A^3VWAZ \ar[rr]_-{11v111} &&
A^2VWAZ \ar[r]_-{1c111} &
AVAWAZ \ar[rr]_-{vwz} &&
VWZ \ar[r]_-{\hat s 1} &
(V\circ W) Z \ar[rrr]_-{\hat s} &&&
(V\circ W) \circ Z
\\
\\
A^4VWZ \ar[r]^-{1c_{A,A^2VW}1} \ar[d]_-{c_{A^3,A}111} &
A^3VWAZ \ar[rr]^-{1c_{A,AV}111} &&
A^2VAWAZ \ar[r]^-{11111z} \ar[d]^-{c11111} &
A^2VAWZ \ar[rr]^-{111w1} \ar[d]^-{c1111} &&
A^2VWZ \ar[r]^-{111\hat s} &
A^2V(W\circ Z) \ar[r]^-{1v1} \ar[d]^-{c11} &
AV(W\circ Z) \ar[r]^-{1\hat s} &
A(V\circ (W\circ Z)) \ar@{-->}[dddddd] \\
A^4VWZ \ar[r]^-{11c_{A^2,V}11} \ar[d]_-{t_411111} &
A^2VA^2WZ \ar[rr]^-{111c_{A,AW}1} &&
A^2VAWAZ \ar[d]^-{t_411111} &
A^2VAWZ \ar[d]^-{t_41111} &&&
A^2V(W\circ Z) \ar[d]^-{t_411} \\
A^4VWZ \ar@{=}[dd] \ar[r]^-{11c_{A^2,V}11} &
A^2VA^2WZ \ar[rr]^-{111c_{A,AW}1} &&
A^2VAWAZ \ar[d]^-{1c1111} &
A^2VAWZ \ar[rr]^-{111w1} \ar[d]^-{1c111} &&
A^2VWZ \ar[r]^-{111\hat s} 
&
A^2V(W\circ Z) \ar[d]^-{1c1} \\
&&& AVA^2WAZ \ar[d]^-{11c111} &
AVA^2WZ \ar[rr]^-{111w1} \ar[d]^-{11c11} &&
AVAWZ \ar[r]^-{111\hat s} &
AVA(W\circ Z) \ar[ddd]^-{vb} \\
A^4VWZ \ar[r]^-{111c11} \ar[d]_-{1c_{A^2,A}111} &
A^3VAWZ \ar[r]^-{1c_{A^2,VA}11} &
AVA^3WZ \ar[r]^-{1111c1} &
AVA^2WAZ \ar[dd]^-{11t_4111} &
AVA^2WZ \ar[dd]^-{11t_411}\\
A^4VWZ \ar[d]_-{1t_41111} \\
A^4VWZ \ar[r]_-{111c_{A,VW}1} &
A^3VWAZ \ar[rr]_-{1c_{A^2,V}111} &&
AVA^2WAZ \ar[r]_-{11111z} &
AVA^2WZ \ar[r]_-{111c1} &
AVAWAZ \ar[r]_-{vwz} &
VWZ \ar[r]_-{1\hat s} &
V(W\circ Z) \ar[rr]_-{\hat s} &&
V\circ (W\circ Z)}}
\thispagestyle{empty}
\end{sideways}
\caption{The $A$-actions on $(V\circ W)\circ Z$ and $V\circ (W\circ Z)$}
\label{fig:(VW)Z&V(WZ)}
\end{figure}

The object $L\circ V$ --- defined up-to isomorphism --- can be chosen to be
$V$. With this choice,
$$
\check s = \xymatrix@C=15pt{
V \ar[r]^\tau &
LV}
\quad \textrm{and} \quad
\hat s = \xymatrix@C=15pt{
LV \ar[r]^-{1\tau} &
L^2 V \ar[r]^-{\mu 1} &
LV \ar[r]^-{\varepsilon 1} &
V\ .}
$$
The left unit constraint is a morphism of $A$-modules if and only if, for any
$A$-module $v\colon AV \to V$, the resulting $A$ action $b\colon 
A(L\circ V)=AV \to V$ is equal to $v$. By Lemma 
\ref{lem:alternative_b}~(2) this is equivalent to the commutativity 
of the large central region of
\begin{equation}\label{eq:lambda_A_lin}
\xymatrix@R=10pt{
A^2LAV \ar[rr]^-{1n_211} \ar[ddd]_-{c111} \ar[rd]^-{111v} &&
A^3V \ar[rrr]^-{1\sqcap^L_1 1} \ar[d]^-{11v} &&&
A^2V \ar[ddddddd]^-{m1} \ar[ld]_-{1v} \\
& A^2LV \ar[r]^-{1n_21} \ar[dd]^-{c11} &
A^2V \ar[r]^-{1p1} &
ALV \ar[r]^-{1\hat s} &
AV \ar[dddd]_-v \\
\\
A^2LAV \ar[dddd]_-{t_4111} &
A^2LV \ar[dd]^-{t_411} \\
\\
& A^2LV \ar@{=}[d] \ar[rr]^-{b^{(2)}} &&
LV \ar[r]^-{\hat s} &
V \\
& A^2LV \ar[r]^-{1c1} &
ALAV \ar[r]^-{n_211} &
A^2V \ar[r]^-{\sqcap^L_1 1} &
AV \ar[u]^-v \\
A^2LAV \ar[rr]_-{1c11} \ar[ru]^-{111v}&&
ALA^2V \ar[r]_-{n_2 111} &
A^3V \ar[r]_-{\sqcap^L_1 11} &
A^2V \ar[r]_-{m1} \ar[u]^-{1v} &
AV . \ar[luu]_-v}
\end{equation}
Above the central region, the region on the right commutes by the
commutativity of
$$
\xymatrix@R=15pt{
A^2V \ar@{=}[r] \ar[dd]_-{1v} &
A^2V \ar[d]_-{1\tau 1} &
A^3V \ar[l]_-{1m1} \ar[d]^-{1t_11} \ar[rr]^-{1m1} &&
A^2V \ar[d]^-{\sqcap^L_11} \\
& ALAV \ar[d]_-{11v} &
A^3V \ar[l]_-{1p11} \ar[r]^-{\sqcap^L_1 11} \ar[d]^-{11v} &
A^2V \ar[r]^-{j11} 
&
AV \ar[d]^-v \\
AV \ar[r]^-{1\tau} \ar[d]_-{p1} &
ALV \ar[d]_-{p11} &
A^2V \ar[l]_-{1p1} \ar[r]^-{\sqcap^L_1 1} &
AV \ar[r]^-{j1} \ar[d]^-{p1} &
V \ar@{=}[d] \\
LV \ar[r]^-{1\tau} \ar@/_1pc/[rrrr]_-{\hat s} &
L^2V \ar[rr]^-{\mu 1} &&
LV \ar[r]^-{\varepsilon 1} &
V}
$$
(where the top-right region commutes by the second equality of \ref{par:A9}),
because the morphisms in the top row are equal epimorphisms. From this we
also obtain
$$
\xymatrix@R=20pt{
A^2LV \ar[r]_-{1c1} \ar@/^1pc/[rrrr]^-{b^{(2)}} &
ALAV \ar[r]_-{n_2 11} &
A^2V \ar[r]_-{1v} \ar[d]_-{\sqcap^L_1 1} &
AV \ar[r]_-{p1} &
LV \ar[d]^-{\hat s} \\
&& AV \ar[rr]_-v &&
V}
$$
proving commutativity of the region below the central region of
\eqref{eq:lambda_A_lin}. All other regions around the central region 
commute by functoriality of the monoidal product and the associativity of
$v$. Thus since $111v\colon  A^2LAV \to A^2LV$ is epi, commutativity of the
central region is equivalent to the commutativity of the exterior of
\eqref{eq:lambda_A_lin}. This holds by the coopposite of the second equality of
\ref{par:A20.6} (using that $11p11\colon A^4V \to A^2LAV$ is epi).  

An analogous reasoning traces back the $A$-module morphism property of the
right unit constraint to the first equality of \ref{par:A20.6}.

The associativity isomorphism is a morphism of $A$-modules if and only if, for
any $A$-modules $v\colon AV \to V$, $w\colon AW\to W$ and $z\colon AZ\to Z$, the
actions $A((V\circ W) \circ Z) \to (V\circ W) \circ Z$ and $A(V\circ (W \circ
Z)) \to V\circ (W \circ Z)$ coincide (omitting the associativity constraint in
the category of $L$-bicomodules). These actions fit the respective diagrams of
Figure \ref{fig:(VW)Z&V(WZ)}. The rightmost regions, as well as
the bottom regions on their left, commute by parts (1) and (2) of Lemma
\ref{lem:alternative_b}. All other regions commute by functoriality of the
monoidal product, and naturality and coherence of the braiding. The 
top rows are equal epimorphisms (up-to the omitted associativity 
isomorphism of the category of $L$-bicomodules) and the left-bottom 
paths are equal by the associativity of the actions $v$ and $z$, and 
\ref{par:A3}. Hence the right verticals are equal proving the claim.  
\end{proof}

Note the difference between  Theorem~\ref{thm:module_cat} and 
\cite[Theorem 5.6]{BohmGomezTorrecillasLopezCentella:wmba} in the case when
$\mathsf C$ is the category of vector spaces over a given field: In
\cite[Theorem 5.6]{BohmGomezTorrecillasLopezCentella:wmba} the weak multiplier
bialgebra in question is assumed to be {\em left full} while in Theorem
\ref{thm:module_cat} this assumption is replaced by the 
non-degeneracy of $n_{1}$ in \eqref{eq:n_1} on the right with respect 
to the chosen class $\mathcal Y$.

\begin{remark} \label{rem:multiplier_bimonad}
In \cite{BohmLack:braided_mba}, monoidality of the category of modules over a
(nice enough) multiplier bimonoid $A$ in a braided monoidal category $\mathsf
C$ was explained by the structure of the induced endofunctor $A(-)$ on $\mathsf
C$. Namely, it was shown to carry the structure of a {\em multiplier bimonad};
a generalization of bimonad (which is another name for opmonoidal
monad). Recall that a multiplier bimonad on a monoidal category is an 
endofunctor $T$ equipped with a morphism $T_0\colon T(I)\to I$ and natural
transformations  
$$
\xymatrix{
T(XT(Y)) \ar[r]^-{\overleftarrow T_2} & 
T(X)T(Y)  &
T(T(X)Y)\ar[l]_-{\overrightarrow T_2}}
$$
subject to compatibility conditions in \cite{BohmLack:braided_mba}.

A similar explanation of the monoidality of the category of modules over a
(nice enough) {\em weak} multiplier bimonoid $A$ is possible, in fact, 
but the treatment is technically more involved. For this reason, we sketch
here the construction without a detailed proof; leaving the technicalities to
the interested reader. 

In the setting of Theorem \ref{thm:L}, the base object $L$ of a 
regular weak multiplier bimonoid $A$ carries the structure $ 
(\delta,\varepsilon,\mu)$ of a coseparable comonoid; so in particular 
that of a semigroup $\mu\colon L^2\to L$. Using the braiding in the 
base category, any $L$-bicomodule can equivalently be regarded as a 
left comodule over the monoidal product comonoid $LL^{\mathsf{op}}$, 
where $L^{\mathsf{op}}$ is the comonoid with the same underlying 
object $L$, the same counit $\varepsilon$ but the opposite 
comultiplication $c^{-1}.\delta$. The comonoid 
$LL^{\mathsf{op}}$ inherits a coseparable structure of $L$. Hence as 
explained in Remark \ref{rem:firm}, any left 
$LL^{\mathsf{op}}$-comodule can equivalently be regarded as a firm 
left module over the monoidal product semigroup 
$LL^{\mathsf{op}}$, where the semigroup $L^{\mathsf{op}}$ has the 
same underlying object $L$ and opposite multiplication $\mu.c$. This 
yields an isomorphism between the category of $L$-bicomodules and the 
category of firm left $LL^{\mathsf{op}}$-modules. Since the category 
of $L$-bicomodules is monoidal --- via the monoidal product $\circ$ 
and the monoidal unit $L$ --- this isomorphism induces a monoidal 
structure --- also to be denoted by $(\circ,L)$ ---  on the category 
of firm left $LL^{\mathsf{op}}$-modules.     

For a regular weak multiplier bimonoid $A$, in addition to the $L$-actions
$n_1$ and $n_2$ in \eqref{eq:n_1}, we can introduce two more actions 
$$
\xymatrix@C=35pt{
A^3 \ar@<2pt>[r]^-{\sqcap^L_1 1} 
\ar@<-2pt>[r]_-{\overline \sqcap^R_1 1.c^{-1}1} &
A^{2} \ar[r]^-{p1} \ar[rd]_-{\overline \sqcap^{R}_{1}} &
LA \ar@{-->}[d]^-{{\overline n}_{1}} \\
&& A}
\qquad
\xymatrix@C=35pt{
A^3 \ar@<2pt>[r]^-{1\sqcap^L_1} 
\ar@<-2pt>[r]_-{1\overline \sqcap^R_1 .1c^{-1}} &
A^{2} \ar[r]^-{1p} \ar[rd]_-{\overline \sqcap^{R}_{2}} &
AL \ar@{-->}[d]^-{{\overline n}_{2}} \\
&& A}
$$
of the opposite semigroup $L^{\mathsf{op}}$. All four of these actions are
associative, and all commute with each other. Hence they make $A$ an
$LL^{\mathsf{op}}$-bimodule, with the left and right actions  
$$
\xymatrix{
L^2A \ar[r]^-{1{\overline n}_1} &
LA \ar[r]^-{n_1} &
A}\qquad
\xymatrix{
AL^2 \ar[r]^-{n_2 1} &
AL \ar[r]^-{{\overline n}_2} &
A .}
$$
There is an important difference between the left and right actions. 
In the setting of Proposition \ref{prop:L-coactions_on_A}, $n_1$ and 
${\overline n}_1$ are firm actions: they correspond (in the way described in 
Remark 3.7) to the coactions $\tau$ and $c^{-1}.\bar \tau$, 
respectively. On the contrary, $n_2$ and ${\overline n}_2$ need not 
be so without further assumptions. In other words, the 
$LL^{\mathsf{op}}$-bimodule $A$ is firm on the left but not 
necessarily on the right.   

Take now any left $LL^{\mathsf{op}}$-module $X$ (with action $x\colon L^2X \to
X$). We can define an object $A\boxtimes X$ as the usual
$LL^{\mathsf{op}}$-module tensor product of the right
$LL^{\mathsf{op}}$-module $A$ and the left $LL^{\mathsf{op}}$-module $X$; that
is, as the coequalizer 
$$
\xymatrix@C=25pt{
AL^2X \ar@<2pt>[rr]^-{{\overline n}_21.n_211}
\ar@<-2pt>[rr]_-{1x} &&
AX \ar@{->>}[r] &
A\boxtimes X .}
$$
The firm left $LL^{\mathsf{op}}$-action $n_1.1{\overline n}_1$ on $A$ induces
a firm left $LL^{\mathsf{op}}$-action on $A\boxtimes X$ (regardless the
properties of the left $LL^{\mathsf{op}}$-module $X$). In particular,
there is an endofunctor $A\boxtimes (-)$ on the category of firm left
$LL^{\mathsf{op}}$-modules.  

Now the endofunctor $A\boxtimes (-)$ on the monoidal category of firm left
$LL^{\mathsf{op}}$-modules can be equipped with the structure of a
multiplier bimonad. For any left $LL^{\mathsf{op}}$-modules $X$ and $Y$, the
structure morphisms of this multiplier bimonad are constructed using the
universality of the coequalizers in the top rows of 
$$
\xymatrix@R=12pt@C=22pt{
AXAY \ar@{->>}[r] \ar[d]_-{1c^{-1}1} &
A\boxtimes (X\circ (A\boxtimes Y)) \ar@{-->}[ddd] 
&
A^2XY \ar@{->>}[r] \ar[d]_-{c11} &
A\boxtimes ((A\boxtimes X) \circ Y)  \ar@{-->}[ddd] 
&
AL \ar@{->>}[r] \ar[ddd]_-{n_2} &
A\boxtimes L \ar@{-->}[ddd] \\
A^2XY \ar[d]_-{t_111} &
&
A^2XY \ar[d]_-{t_411} \\
A^2XY \ar[d]_-{1c1} &
&
A^2XY \ar[d]_-{1c1} \\
AXAY \ar@{->>}[r] &
(A\boxtimes X) \circ (A\boxtimes Y) 
&
AXAY \ar@{->>}[r] &
(A\boxtimes X) \circ (A\boxtimes Y) 
&
A \ar@{->>}[r]_-p &
L.}
$$
\end{remark}


\section{Uniqueness of the idempotent morphisms and the counit of a regular
  weak multiplier bimonoid}

Consider regular weak multiplier bimonoids 
$(t_1,t_2,t_3,t_4,e_1,e_2,j)$ and 
$(t'_1,t'_2,t'_3,t'_4,e'_1,e'_2,$ $j')$ on the same underlying object 
$A$. By Corollary~\ref{cor:unique_t}, we know that if $t_1=t'_1$ and 
the multiplications $j1.t_1$ and $j'1.t'_1$ are equal and 
non-degenerate with respect to some class containing  $I$, $A$, and 
$A^2$,  then $t_2=t'_2$, $t_3=t'_3$, and $t_4=t'_4$. 
The aim of this section is to find criteria for the
uniqueness of the remaining structure $e_1$, $e_2$, $j$. The findings below
generalize \cite[Lemma 3.3]{VDaWa:Banach} and \cite[Theorem
2.8]{BohmGomezTorrecillasLopezCentella:wmba} . 

\begin{lemma} \label{lem:g}
Let $\mathsf C$ be a braided monoidal category in which
coequalizers are preserved by the monoidal product. Consider a regular weak
multiplier bimonoid $(t_1,t_2,t_3,t_4,e_1,e_{2},j)$ in $\mathsf C$ such that
the induced multiplication $m:=j1.t_{1}$ is a regular epimorphism and
non-degenerate with respect to  some class of objects containing $I$, $A$ and
$A^2$. Then the following hold. 
\begin{itemize}
\item[{(1)}] There is a unique morphism $g\colon A^{2} \to A^{2}$
rendering commutative the equivalent diagrams
$$
\xymatrix{
A^3 \ar[r]^-{1m} \ar[d]_-{1t_{1}} &
A^2 \ar[d]^-g \\
A^3 \ar[r]_-{\overline \sqcap^R_2 1} &
A^2}\quad
\xymatrix{
A^3 \ar[r]^-{m1} \ar[d]_-{t_2 1} &
A^2 \ar[d]^-g \\
A^3 \ar[r]_-{1 \overline \sqcap^L_1} &
A^2 .}
$$
\item[{(2)}] The following diagrams commute.
$$
\xymatrix{
A^{3} \ar[r]^-{g1} \ar[d]_-{1c} &
A^{3} \ar[r]^-{1c} &
A^{3} \ar[d]^-{m1} \\
A^{3} \ar[r]_-{1e_{1}} &
A^{3} \ar[r]_-{m1} &
A^{2}} \quad
\xymatrix{
A^{3} \ar[r]^-{1g} \ar[d]_-{c1} &
A^{3} \ar[r]^-{c1} &
A^{3} \ar[d]^-{1m} \\
A^{3} \ar[r]_-{e_2 1} &
A^{3} \ar[r]_-{1m} &
A^{2}} 
$$
\item[{(3)}] The following diagrams commute.
$$
\xymatrix{
A^{3} \ar[r]^-{1g} \ar[d]_-{m1} &
A^{3} \ar[d]^-{m1} \\
A^{2} \ar[r]_-g &
A^{2}} \quad
\xymatrix{
A^{3} \ar[r]^-{g1} \ar[d]_-{1m} &
A^{3} \ar[d]^-{1m} \\
A^{2} \ar[r]_-g &
A^{2}}
$$
\item[{(4)}] The following diagrams commute.
$$
\xymatrix{
A^{2} \ar[r]^-{t_{1}} \ar[d]_-g \ar[rd]|-{\overline \sqcap^{R}_{2}} &
A^{2} \ar[d]^-{1j} \\
A^{2} \ar[r]_-{1j} &
A} \quad
\xymatrix{
A^{2} \ar[r]^-{t_2} \ar[d]_-g \ar[rd]|-{\overline \sqcap^L_1} &
A^{2} \ar[d]^-{j1} \\
A^{2} \ar[r]_-{j1} &
A} 
$$
\end{itemize}
\end{lemma}

\begin{proof}
{(1)} It follows by \ref{par:A5} and the non-degeneracy of $m$ 
with respect to $A$ that the left-bottom path of the first diagram 
coequalizes any morphisms that the top row does. Since the top row is 
a regular epimorphism, there is a unique morphism $g$ as in the first 
diagram. It obeys 

\begin{gather*}
\begin{tikzpicture}[scale=1] 
\path (.7,1) node[arr,name=g] {$\ g\ $};
\draw[braid] (0,2.5) to[out=270,in=180] (.3,1.5) to [out=0,in=270] (.6,2.5);
\draw[braid] (.8,2.5) to[out=270,in=180] (1.1,1.5) to [out=0,in=270] (1.4,2.5);
\draw[braid] (.3,1.5) to[out=270,in=135] (g);
\draw[braid] (1.1,1.5) to[out=270,in=45] (g);
\draw[braid] (g) to[out=225,in=90] (.3,0);
\draw[braid] (g) to[out=315,in=90] (1.1,0);
\draw (1.9,1.35) node {$=$};
\end{tikzpicture}
\begin{tikzpicture}[scale=1] 
\path (1.1,1.5) node[arr,name=t1] {$t_{1}$}
(.55,.7) node[arr,name=pibarr2] {${}_{\overline\sqcap^R_2}$};
\draw[braid] (0,2.5) to[out=270,in=180] (.3,2) to [out=0,in=270] (.6,2.5);
\draw[braid] (.3,2) to[out=270,in=135] (pibarr2);
\draw[braid] (.8,2.5) to[out=270,in=135] (t1);
\draw[braid] (1.4,2.5) to[out=270,in=45] (t1);
\draw[braid] (t1) to[out=225,in=45] (pibarr2);
\draw[braid] (t1) to[out=315,in=90] (1.4,0);
\draw[braid] (pibarr2) to[out=270,in=90] (.55,0);
\draw (1.8,1.5) node[empty] {$\stackrel{
\eqref{eq:components_are_A-linear}}=$};
\end{tikzpicture}\,
\begin{tikzpicture}[scale=1] 
\path (1.1,1.7) node[arr,name=t1] {$t_{1}$}
(.7,1.1) node[arr,name=pibarr2] {${}_{\overline\sqcap^R_2}$};
\draw[braid] (0,2.5) to (0,1.2) to[out=270,in=180] (.35,.5) to [out=0,in=270] (pibarr2);
\draw[braid] (.4,2.5) to[out=270,in=135] (pibarr2);
\draw[braid] (.35,.5) to[out=270,in=90] (.35,0);
\draw[braid] (.8,2.5) to[out=270,in=135] (t1);
\draw[braid] (1.4,2.5) to[out=270,in=45] (t1);
\draw[braid] (t1) to[out=225,in=45] (pibarr2);
\draw[braid] (t1) to[out=315,in=90] (1.4,0);
\draw (1.9,1.5) node[empty] {$\stackrel{~\eqref{eq:pibarr}}=$};
\end{tikzpicture}
\begin{tikzpicture}[scale=1] 
\path (1.1,2) node[arr,name=t1u] {$t_{1}$}
(.6,1.5) node[arr,name=t1d] {$t_{1}$}
(.9,1.1) node[arr,name=ju] {$\ $};
\draw[braid] (0,2.5) to (0,1.2) to[out=270,in=180] (0.2,0.5) to[out=0,in=225] (t1d);
\draw[braid] (.4,2.5) to[out=270,in=135] (t1d);
\draw[braid] (.8,2.5) to[out=270,in=135] (t1u);
\draw[braid] (1.4,2.5) to[out=270,in=45] (t1u);
\draw[braid] (t1u) to[out=225,in=45] (t1d);
\draw[braid] (0.2,0.5) to (0.2,0);
\draw[braid] (t1u) to[out=315,in=90] (1.4,0);
\draw[braid] (t1d) to[out=315,in=135] (ju);
\draw (2,1.5) node[empty] {$\stackrel{~\eqref{eq:multiplication_2}}=$};
\end{tikzpicture}
\begin{tikzpicture}[scale=1] 
\path (1.1,2) node[arr,name=t1u] {$t_{1}$}
(.6,1.5) node[arr,name=t1d] {$t_{1}$}
(.2,.8) node[arr,name=t2] {$t_{2}$}
(.9,1.1) node[arr,name=ju] {$\ $}
(.5,.4) node[arr,name=jd] {$\ $};
\draw[braid] (0,2.5) to[out=270,in=135] (t2);
\draw[braid] (.4,2.5) to[out=270,in=135] (t1d);
\draw[braid] (.8,2.5) to[out=270,in=135] (t1u);
\draw[braid] (1.4,2.5) to[out=270,in=45] (t1u);
\draw[braid] (t1u) to[out=225,in=45] (t1d);
\draw[braid] (t1d) to[out=225,in=45] (t2);
\draw[braid] (t2) to[out=225,in=90] (0,0);
\draw[braid] (t1u) to[out=315,in=90] (1.4,0);
\draw[braid] (t1d) to[out=315,in=135] (ju);
\draw[braid] (t2) to[out=315,in=135] (jd);
\draw (2,1.5) node[empty] {$\stackrel{~\eqref{eq:interchange_12}}=$};
\end{tikzpicture}
\begin{tikzpicture}[scale=1] 
\path (1.1,1.8) node[arr,name=t1u] {$t_{1}$}
(.7,1) node[arr,name=t1d] {$t_{1}$}
(.3,1.8) node[arr,name=t2] {$t_{2}$}
(1,.5) node[arr,name=jr] {$\ $}
(.4,.5) node[arr,name=jl] {$\ $};
\draw[braid] (0,2.5) to[out=270,in=135] (t2);
\draw[braid] (.6,2.5) to[out=270,in=45] (t2);
\draw[braid] (.8,2.5) to[out=270,in=135] (t1u);
\draw[braid] (1.4,2.5) to[out=270,in=45] (t1u);
\draw[braid] (t1u) to[out=225,in=45] (t1d);
\draw[braid] (t2) to[out=315,in=135] (t1d);
\draw[braid] (t2) to[out=225,in=90] (0,0);
\draw[braid] (t1u) to[out=315,in=90] (1.4,0);
\draw[braid] (t1d) to[out=315,in=135] (jr);
\draw[braid] (t1d) to[out=225,in=45] (jl);
\draw (1.9,1.5) node[empty] {$\stackrel{~\eqref{eq:multiplication_1}}=$};
\end{tikzpicture}\\
\begin{tikzpicture}[scale=1] 
\path (1.1,1.8) node[arr,name=t1u] {$t_{1}$}
(.3,1.8) node[arr,name=t2] {$t_{2}$}
(0.7,.5) node[arr,name=j] {$\ $};
\draw[braid] (0,2.5) to[out=270,in=135] (t2);
\draw[braid] (.6,2.5) to[out=270,in=45] (t2);
\draw[braid] (.8,2.5) to[out=270,in=135] (t1u);
\draw[braid] (1.4,2.5) to[out=270,in=45] (t1u);
\draw[braid] (t1u) to[out=225,in=0] (0.7,1) to[out=180,in=315] (t2);
\draw[braid] (t2) to[out=225,in=90] (0,0);
\draw[braid] (t1u) to[out=315,in=90] (1.4,0);
\draw[braid] (0.7,1) to[out=270,in=90] (j);
\draw (2,1.5) node[empty] {$\stackrel{~\eqref{eq:multiplication_2}}=$};
\end{tikzpicture}
\begin{tikzpicture}[scale=1] 
\path (1.1,1.8) node[arr,name=t1u] {$t_{1}$}
(.7,1) node[arr,name=t2d] {$t_{2}$}
(.3,1.8) node[arr,name=t2] {$t_{2}$}
(1,.5) node[arr,name=jr] {$\ $}
(.4,.5) node[arr,name=jl] {$\ $};
\draw[braid] (0,2.5) to[out=270,in=135] (t2);
\draw[braid] (.6,2.5) to[out=270,in=45] (t2);
\draw[braid] (.8,2.5) to[out=270,in=135] (t1u);
\draw[braid] (1.4,2.5) to[out=270,in=45] (t1u);
\draw[braid] (t1u) to[out=225,in=45] (t1d);
\draw[braid] (t2) to[out=315,in=135] (t1d);
\draw[braid] (t2) to[out=225,in=90] (0,0);
\draw[braid] (t1u) to[out=315,in=90] (1.4,0);
\draw[braid] (t1d) to[out=315,in=135] (jr);
\draw[braid] (t1d) to[out=225,in=45] (jl);
\draw (1.9,1.5) node[empty] {$\stackrel{~\eqref{eq:interchange_12}}=$};
\end{tikzpicture}
\begin{tikzpicture}[scale=1] 
\path (.2,2) node[arr,name=t2u] {$t_{2}$}
(.6,1.5) node[arr,name=t2d] {$t_{2}$}
(1.1,.8) node[arr,name=t1] {$t_{1}$}
(.3,1.1) node[arr,name=ju] {$\ $}
(.8,.4) node[arr,name=jd] {$\ $};
\draw[braid] (0,2.5) to[out=270,in=135] (t2u);
\draw[braid] (.4,2.5) to[out=270,in=45] (t2u);
\draw[braid] (.8,2.5) to[out=270,in=45] (t2d);
\draw[braid] (1.4,2.5) to[out=270,in=45] (t1);
\draw[braid] (t2u) to[out=315,in=135] (t2d);
\draw[braid] (t2d) to[out=315,in=135] (t1);
\draw[braid] (t2u) to[out=225,in=90] (0,0);
\draw[braid] (t1) to[out=315,in=90] (1.4,0);
\draw[braid] (t2d) to[out=225,in=45] (ju);
\draw[braid] (t1) to[out=225,in=45] (jd);
\draw (1.9,1.5) node[empty] {$\stackrel{~\eqref{eq:multiplication_1}}=$};
\end{tikzpicture}
\begin{tikzpicture}[scale=1] 
\path (.2,2) node[arr,name=t2u] {$t_{2}$}
(.6,1.5) node[arr,name=t2d] {$t_{2}$}
(.3,1.1) node[arr,name=ju] {$\ $};
\draw[braid] (0,2.5) to[out=270,in=135] (t2u);
\draw[braid] (.4,2.5) to[out=270,in=45] (t2u);
\draw[braid] (.8,2.5) to[out=270,in=45] (t2d);
\draw[braid] (1.4,2.5) to[out=270,in=0] (1.1,0.8) to[out=180,in=315] (t2d);
\draw[braid] (t2u) to[out=315,in=135] (t2d);
\draw[braid] (t2u) to[out=225,in=90] (0,0);
\draw[braid] (1.1,0.8) to[out=270,in=90] (1.1,0);
\draw[braid] (t2d) to[out=225,in=45] (ju);
\draw (1.9,1.5) node[empty] {$\stackrel{~\eqref{eq:pibarl}}=$};
\end{tikzpicture}
\begin{tikzpicture}[scale=1] 
\path (.3,1.7) node[arr,name=t2] {$t_{2}$}
(.9,1.1) node[arr,name=pibarl1] {${}_{\overline\sqcap^L_1}$};
\draw[braid] (0,2.5) to[out=270,in=135] (t2);
\draw[braid] (.6,2.5) to[out=270,in=45] (t2);
\draw[braid] (1.1,2.5) to[out=270,in=45] (pibarl1);
\draw[braid] (pibarl1) to[out=270,in=180] (1.2,.5) to[out=0,in=270] 
(1.4,2.5);
\draw[braid] (1.2,.5) to[out=270,in=90] (1.2,0);
\draw[braid] (t2) to[out=315,in=135] (pibarl1);
\draw[braid] (t2) to[out=225,in=90] (0,0);
\draw (2.02,1.5) node[empty] {$\stackrel{
\eqref{eq:components_are_A-linear}}=$};
\end{tikzpicture}
\begin{tikzpicture}[scale=1] 
\path (.3,1.5) node[arr,name=t2] {$t_{2}$}
(.8,.8) node[arr,name=pibarl1] {${}_{\overline\sqcap^L_1}$};
\draw[braid] (0,2.5) to[out=270,in=135] (t2);
\draw[braid] (.6,2.5) to[out=270,in=45] (t2);
\draw[braid] (.8,2.5) to[out=270,in=180] (1.1,2) to[out=0,in=270] (1.4,2.5);
\draw[braid] (1.1,2) to[out=270,in=45] (pibarl1);
\draw[braid] (t2) to[out=315,in=135] (pibarl1);
\draw[braid] (t2) to[out=225,in=90] (0,0);
\draw[braid] (pibarl1) to[out=270,in=90](.8,0);
\end{tikzpicture} 
\end{gather*}
where the first equality holds by definition of $g$. Since $11m\colon A^{4}
\to A^{3}$ is epi, this proves commutativity of the second diagram (which is
the opposite-coopposite of the first one). 

{(2)} The first diagram commutes by
$$
\begin{tikzpicture}[scale=1]
\path (.6,1.2) node[arr,name=e1] {$e_1$};
\draw[braid] (0,2.5) to[out=270,in=180] (.2,.5) to[out=0,in=225] (e1);
\draw[braid] (.2,.5) to[out=270,in=90] (.2,0);
\draw[braid] (.2,2.5) to[out=270,in=180] (.4,2) to[out=0,in=270] (.6,2.5);
\path[braid,name path=m>e1] (.4,2) to[out=270,in=45] (e1);
\draw[braid,name path=o>e1] (1,2.5) to[out=270,in=135] (e1);
\fill[white, name intersections={of=m>e1 and o>e1}] (intersection-1) circle(0.1);
\draw[braid] (.4,2) to[out=270,in=45] (e1);
\draw[braid] (e1) to[out=315,in=90] (.8,0);
\draw (1.5,1.5) node[empty] {$\stackrel{\ref{par:A4}}=$};
\end{tikzpicture} 
\begin{tikzpicture}[scale=1]
\path (.6,2) node[arr,name=t1] {$t_1$}
(.6,1) node[arr,name=pibarr1] {${}_{\overline \sqcap^R_1}$};
\draw[braid] (0,2.5) to[out=270,in=180] (.2,.5) to[out=0,in=270] (pibarr1);
\draw[braid] (.2,.5) to[out=270,in=90] (.2,0);
\draw[braid] (.4,2.5) to[out=270,in=135] (t1);
\draw[braid] (.8,2.5) to[out=270,in=45] (t1);
\draw[braid] (t1) to[out=225,in=135] (pibarr1);
\path[braid,name path=t1>o] (t1) to[out=315,in=90] (1.2,0);
\draw[braid,name path=o>pibarr1] (1.2,2.5) to[out=270,in=45] (pibarr1);
\fill[white, name intersections={of=t1>o and o>pibarr1}] (intersection-1) circle(0.1);
\draw[braid] (t1) to[out=315,in=90] (1.2,0);
\draw (1.9,1.5) node[empty] {$\stackrel{\eqref{eq:components_compatibility}}=$};
\end{tikzpicture} 
\begin{tikzpicture}[scale=1]
\path (.6,2) node[arr,name=t1] {$t_1$}
(.2,1) node[arr,name=pibarr2] {${}_{\overline \sqcap^R_2}$};
\draw[braid] (0,2.5) to[out=270,in=135] (pibarr2);
\draw[braid] (.4,2.5) to[out=270,in=135] (t1);
\draw[braid] (.8,2.5) to[out=270,in=45] (t1);
\draw[braid] (t1) to[out=225,in=45] (pibarr2);
\draw[braid,name path=o>m] (1.2,2.5) to[out=270,in=0] (.7,.5) 
to[out=180,in=270] (pibarr2);
\path[braid,name path=t1>o] (t1) to[out=315,in=90] (1.2,0);
\fill[white, name intersections={of=t1>o and o>m}] (intersection-1) circle(0.1);
\draw[braid] (t1) to[out=315,in=90] (1.2,0);
\draw[braid] (.7,.5) to[out=270,in=90] (.7,0);
\draw (1.7,1.3) node {$=$};
\end{tikzpicture} 
\begin{tikzpicture}[scale=1]
\path (.4,1.5) node[arr,name=g] {$\ g\ $};
\draw[braid] (0,2.5) to[out=270,in=135] (g);
\draw[braid,name path=o>m] (.8,2.5) to[out=270,in=0] (.6,2) 
to[out=180,in=270] (.4,2.5);
\draw[braid] (.6,2) to[out=270,in=45] (g);
\draw[braid,name path=o>m] (1.2,2.5) to[out=270,in=0] (.3,.5) 
to[out=180,in=225] (g);
\path[braid,name path=g>o] (g) to[out=315,in=90] (1.2,0);
\fill[white, name intersections={of=g>o and o>m}] (intersection-1) circle(0.1);
\draw[braid] (g) to[out=315,in=90] (1.2,0);
\draw[braid] (.3,.5) to[out=270,in=90] (.3,0);
\end{tikzpicture}
$$
and since $1m1\colon A^{4} \to A^{3}$ is epi, where the 
unlabelled equality holds by definition of $g$. The second diagram 
commutes by the opposite-coopposite reasoning.  
 
{(3)} The left-bottom path of the first diagram in part {(1)} is a left $A$-module morphism by 
\eqref{eq:components_are_A-linear}.
Hence so is the top-right path. Since
$11m\colon A^{4} \to A^{3}$ is also left $A$-linear and epi, this
proves commutativity of the first diagram. The second diagram commutes by the
opposite-coopposite reasoning.  

{(4)} The first diagram commutes by
$$
\begin{tikzpicture}[scale=1]
\path (.2,1.5) node[arr,name=t1] {$t_1$}
(.6,1.1) node[arr,name=j] {$\ $};
\draw[braid] (0,2) to[out=270,in=135] (t1);
\draw[braid] (.4,2) to[out=270,in=45] (t1);
\draw[braid] (t1) to[out=225,in=180] (.4,.5) to[out=0,in=270] (.8,2);
\draw[braid] (.4,.5) to[out=270,in=90] (.4,0);
\draw[braid] (t1) to[out=315,in=135] (j);
\draw (1.3,1.3) node {$=$};
\end{tikzpicture}
\begin{tikzpicture}[scale=1]
\path (.6,1.2) node[arr,name=e1] {$e_1$}
(1,.8) node[arr,name=j] {$\ $};
\draw[braid] (0,2) to[out=270,in=180] (.2,.5) to[out=0,in=225] (e1);
\draw[braid] (.2,.5) to[out=270,in=90] (.2,0);
\path[braid,name path=d] (.4,2) to[out=270,in=45] (e1);
\draw[braid,name path=u] (.8,2) to[out=270,in=135] (e1);
\fill[white, name intersections={of=u and d}] (intersection-1) circle(0.1);
\draw[braid] (.4,2) to[out=270,in=45] (e1);
\draw[braid] (e1) to[out=315,in=135] (j);
\draw (1.5,1.3) node {$=$};
\end{tikzpicture}
\begin{tikzpicture}[scale=1]
\path (.2,1.5) node[arr,name=t1] {$\ g\ $}
(.6,1.1) node[arr,name=j] {$\ $};
\draw[braid] (0,2) to[out=270,in=135] (t1);
\draw[braid] (.4,2) to[out=270,in=45] (t1);
\draw[braid] (t1) to[out=225,in=180] (.4,.5) to[out=0,in=270] (.8,2);
\draw[braid] (.4,.5) to[out=270,in=90] (.4,0);
\draw[braid] (t1) to[out=315,in=135] (j);
\end{tikzpicture}
$$
and the non-degeneracy of $m$. The first equality holds by Axiom~\hyperlink{AxVII}{\textnormal{VII}} and the second by commutativity of
the first diagram in part {(2)}. The second diagram commutes by the
opposite-coopposite reasoning. 
\end{proof}    

\begin{theorem}
Let $\mathsf C$ be a braided monoidal category in which
coequalizers are preserved by the monoidal product. Let
$(t_1,t_2,t_3,t_4,e_1,e_{2},j)$ and $(t_1,t_2,t_3,t_4,e'_1,e'_{2},j')$ be
regular weak multiplier bimonoids in $\mathsf C$ such that the induced
multiplications $m:=j1.t_{1}$ and $m':=j'1.t_{1}$ are equal regular
epimorphisms and  non-degenerate with respect to  some class of objects
containing $I$, $A$ and $A^2$. Then the following are equivalent. 
\begin{itemize}
\item[{(i)}] $e_1=e'_1$,
\item[{(ii)}] $e_2=e'_2$,
\item[{(iii)}] $g=g'$,
\item[{(iv)}] $j=j'$. 
\end{itemize}
These equivalent assertions hold true if in addition idempotent morphisms in
$\mathsf C$ split and the morphisms $\hat d_{1}$ (or $\hat d_{2}$) of 
\eqref{eq:dhat} are epimorphisms both for the unprimed and for the primed
data.   
\end{theorem}     

\begin{proof}
 (i)  $\Leftrightarrow$ (ii).
This follows by \eqref{eq:e_components}, applied both to the unprimed 
and the primed data, and the non-degeneracy of $m=m'$ with respect to  $A$. 

(i) $\Leftrightarrow$ (iii).
This follows by the first diagram in part (2)  of 
Lemma~\ref{lem:g}, applied both to the unprimed and the primed data, 
and the non-degeneracy of $m=m'$ with respect to  $A$.   

(iv) $\Rightarrow$ (iii).
If  (iv)  holds then the left-bottom path of the first 
diagram of Lemma~\ref{lem:g} (1)  is the same for the primed and for 
the unprimed data. Then so is the top-right path. Since $1m$ and 
$1m'$ are equal epimorphisms, this proves $g=g'$. 

(iii) $\Rightarrow$ (iv).
This follows by the the fact that $m$ and $m'$ are equal epimorphisms 
and the calculation
$$
j.m'\stackrel{\eqref{eq:multiplication_1}}=
j'j.t_{1} =
j'j.g\stackrel{(iii)}=
j'j.g'=
j'j.t_{2}\stackrel{\eqref{eq:multiplication_2}}=
j'.m ,
$$
in which the unlabelled equalities hold by Lemma~\ref{lem:g}(4).  

Assume now that idempotent morphisms in $\mathsf C$ split.
The morphisms $d_1$ of \ref{par:d1_1}, built up from the unprimed and 
from the primed data, are the same. Applying the last equality of 
\ref{par:d1_1} to the unprimed and the primed data, respectively, we get 
\begin{equation}\label{eq:ab}
e_1.d_1=d_1 \qquad\qquad\qquad
e'_1.d_1=d_1 .
\end{equation}
Using the second equality, we obtain the equivalent expressions of $d_1$ in
$$
d_1=
e'_1.d_1=
\check e'_1.\hat e'_1.d_1=
\check e'_1.\hat d'_1.
$$
Substituting this expression of $d_1$ in the first equality of \eqref{eq:ab},
we get $e_1.\check e'_1.\hat d'_1=\check e'_1.\hat d'_1$. Since $\hat d'_1$ is
epi by assumption, this shows that $\check e'_1$ equalizes $e_1$ and $1$. Thus
universality of the equalizer 
\begin{equation} \label{eq:f}
\xymatrix{
& A\circ' A \ar[d]^-{\check e'_1} \ar@{-->}[ld]_-f \\
A\circ A \ar[r]_-{\check e_1} &
A^2 \ar@<2pt>[r]^-{e_1}\ar@<-2pt>[r]_-1 &
A^2}
\end{equation}
yields a morphism $f$ as in the diagram. It is an isomorphism with the inverse
constructed by a symmetrical reasoning.

Post-composing the morphisms around the triangle of \eqref{eq:f} 
by $\hat e'_1$ gives $\hat e'_1.\check e_1.f=\hat e'_1.\check 
e'_1=1$; so that $\hat e'_1.\check e_1$ is the inverse of $f$ and 
therefore $f.\hat e'_1.\check e_1=1$. Using this in the penultimate 
equality and commutativity of the triangular region of \eqref{eq:f} 
in the second one, we obtain 
$$
e'_1.e_1=
\check e'_1.\hat e'_1.\check e_1.\hat e_1=
\check e_1.f.\hat e'_1.\check e_1.\hat e_1=
\check e_1.\hat e_1=
e_1.
$$
Since $e_1$ on the right hand side is the first component of an \mM-morphism
$A^2 \nto A^2$ by \eqref{eq:e_components}, so must be $e'_1.e_1$ on the left
hand side. The \mM-morphism with first component $e_1=e'_1.e_1$ has
$e_2=e_2.e'_2$ as the second component, again by \eqref{eq:e_components},
applied both to the primed and the unprimed data.  

Symmetrical reasoning leads to the further equalities 
$$
e_1.e'_1=e'_1\ \textrm{and}\ e'_2.e_2=e'_2,\qquad
e'_2.e_2=e_2\ \textrm{and}\ e_1.e'_1=e_1,\qquad
e_2.e'_2=e'_2\ \textrm{and}\ e'_1.e_1=e'_1.
$$
They immediately imply $e_1=e_1.e'_1=e'_1$; that is, the first of the
equivalent assertions of the theorem.
\end{proof}    


\section{Working in a closed braided monoidal category} \label{sec:closed}

In this final section we assume that the braided monoidal category $\mathsf
C$ is also closed and investigate the consequences of this on the assumptions
and the constructions of the previous sections. 

A braided monoidal category ${\mathsf C}$ is said to be {\em closed}
if, for any object $X$, the functor $X(-)\colon {\mathsf C}\to {\mathsf C}$
possesses a right adjoint, to be denoted by $[X,-]$ (this is equivalent to
$(-)X\cong X(-)$ possessing a right adjoint). The unit and the counit of the
adjunction $X(-) \dashv [X,-]$ will be denoted by $\mathsf{coev}$ and
$\mathsf{ev}$, respectively. 

Since any left adjoint functor preserves coequalizers, in a closed braided
monoidal category coequalizers are preserved by taking the monoidal product
with any object.

\subsection{The multiplier monoid}

In this section we recall some background material from
\cite{BohmLack:cat_of_mbm}. Consider the images of the morphisms
$$
\xymatrix{A^2[A,A]\ar[r]^-{1\mathsf{ev}} & A^2\ar[r]^-m & A}
\qquad
\xymatrix{A^2[A,A]\ar[r]^-{1c} &
A[A,A]A \ar[r]^-{\mathsf{ev} 1} & 
A^2 \ar[r]^-m &
A} 
$$
under the adjunction isomorphism ${\mathsf C}(A^2[A,A],A)\cong {\mathsf
C}([A,A],[A^2,A])$. If their pullback exists, as it will in any 
abelian category, then we call it the {\em multiplier monoid} of
$A$ (a term justified by the verification of its monoid structure in
\cite{BohmLack:cat_of_mbm}) and we denote it by $\mM(A)$ as in the first 
diagram of
\begin{equation}\label{eq:M(A)}
\xymatrix{
\mM(A)\ar@{-->}[r] \ar@{-->}[d] &
[A,A]\ar[d]\\
[A,A]\ar[r]&
[A^2,A]}
\qquad\qquad
\xymatrix{
A\mM(A)A\ar@{-->}[r]^-{1h_1} \ar@{-->}[d]_-{h_2 1} &
A^2\ar[d]^-m\\
A^2\ar[r]_-m &
A.}
\end{equation}
Using the adjunction isomorphism ${\mathsf C}(\mM(A),[A^2,A])\cong {\mathsf
C}(A^2\mM(A),A)$, the first diagram in \eqref{eq:M(A)} translates to the
second one, involving the morphisms 
$
(\xymatrix{
\mM(A)A \ar[r]|-{\,h_1\,} &
A &
A\mM(A) \ar[l]|-{\,h_2\,}})
$
with the universal property that for any object $X$, the morphisms $f\colon X\to
\mM(A)$ correspond bijectively to pairs of morphisms 
$
(\xymatrix{
XA \ar[r]|-{\,f_1\,} &
A &
AX \ar[l]|-{\,f_2\,}})
$
such that 
\begin{equation}\label{eq:component-compatibility}
\xymatrix{
AXA\ar[r]^-{1f_1} \ar[d]_-{f_2 1} &
A^2\ar[d]^-m\\
A^2\ar[r]_-m &
A}
\end{equation}
commutes; that is $f_1$ and $f_2$ are components of an {\em \mM-morphism}
$X\nto A$. We call $f_1$ and $f_2$ the {\em components} also of the
corresponding morphism $f\colon X \to \mM(A)$ in $\mathsf C$. Explicitly, the
correspondence between $f\colon X\to \mM(A)$ and its components $(f_1,f_2)$ is
expressed by the commutative diagrams  
\begin{equation}\label{eq:components}
\xymatrix{
XA\ar[r]^-{f_1}\ar[d]_-{f1} &
A \ar@{=}[d]&
AX \ar[l]_-{f_2} \ar[d]^-{1f}\\
\mM(A)A\ar[r]_-{h_1}&
A&
A\mM(A)\ar[l]^-{h_2}.}
\end{equation}

The pair 
$
(\xymatrix{
\mM(A)A \ar[r]|-{\,h_1\,} &
A &
A\mM(A) \ar[l]|-{\,h_2\,}})
$
in \eqref{eq:M(A)} can be regarded as the components of the identity morphism
$\mM(A)\to \mM(A)$. By the associativity of $m$, 
$
(\xymatrix{
A^2 \ar[r]|-{\,m\,} &
A &
A^2 \ar[l]|-{\,m\,}})
$
are components of a morphism $i\colon A\to \mM(A)$. 

\begin{proposition}\label{prop:dense}
If the pullback $\mM(A)$ exists for a semigroup $A$ with non-degenerate
multiplication $m$, then $h_1$ in \eqref{eq:components} is non-degenerate on
the right and $h_2$ in \eqref{eq:components} is non-degenerate on the left. 
\end{proposition}

\begin{proof}
For morphisms $f,g\colon X\to \mM(A)$, it follows by
\eqref{eq:component-compatibility} that $f_1=h_1.f1$ and $g_1=h_1.g1$ are
equal if and only if $f_2=h_2.1f$ and $g_2=h_2.1g$ are equal. If this is the
case, then $f=g$ by the universality of the pullback.
\end{proof}

In the category of vector spaces, the non-degeneracy properties of $h_1$ and
$h_2$ in Proposition \ref{prop:dense} are referred to as the {\em density of
$A$ in $\mM(A)$}, see \cite{Dauns:Multiplier}.

\begin{corollary}\label{cor:mono-nd}
For a semigroup $A$ satisfying the conditions in Proposition \ref{prop:dense},
and some morphism $f\colon X \to \mM(A)$, the following are equivalent.
\begin{itemize}
\item $f$ is a monomorphism,
\item $f_1$ is non-degenerate on the right,
\item $f_2$ is non-degenerate on the left.
\end{itemize}
In particular, for such a semigroup $A$, the morphism $i\colon A\to \mM(A)$ ---
whose components are equal to $m$ --- is a monomorphism.  
\end{corollary} 

\subsection{A distinguished class of objects}

In a closed braided monoidal category $\mathsf C$, we can 
make the following choice of a class $\mathcal Y$ of objects in $\mathsf C$. 
Let $\mathcal Y$ contain those objects $Y$ in $\mathsf C$ which obey the 
following properties.  
\begin{itemize} 
\item[{(a)}] The functor $Y(-)\colon \mathsf C \to \mathsf C$ preserves
  monomorphisms. 
\item[{(b)}] For any objects $X$ and $Z$ of $\mathsf C$, 
$q:=\xymatrix@C=25pt{
[X,Z]Y \ar[r]^-{\mathsf{coev}} &
[X,X[X,Z]Y] \ar[r]^-{[X,\mathsf{ev}1]} &
[X,ZY]}$
is a monomorphism.
\end{itemize}

\begin{example}\label{ex:mod_k}
In the closed symmetric monoidal category of modules over a commutative ring,
property (a) characterizes the {\em flat} modules $Y$. Property (b) holds for
{\em locally projective} \cite{ZimHui} modules $Y$. Since locally projective
modules are also flat, all locally projective (so in particular all
projective) modules belong to the class $\mathcal Y$.  
\end{example}

From this immediately follows the following.

\begin{example} \label{ex:vec}
In the closed symmetric monoidal category of vector spaces, every object
belongs to the class $\mathcal Y$. 
\end{example}

More generally, we shall see that in the closed symmetric monoidal category
of group graded vector spaces every object belongs to the class $\mathcal
Y$.  We do this in the following, still more general, setting. Let 
$\mathsf C$ be a closed braided monoidal category and let $G$ be a 
cocommutative bimonoid in it. Then the category ${\mathsf C}^G$ of 
$G$-comodules (that is, the Eilenberg-Moore category of the comonad 
$G(-)$ on $\mathsf C$) is a braided monoidal category (via the 
braided monoidal structure lifted from $\mathsf C$). Consequently, in 
this case any $G$-comodule 
$\xymatrix@C=18pt{Z\ar[r]|(.4){\, z\,} & GZ}$ 
induces a commutative diagram  
\begin{equation} \label{eq:lift}
\xymatrix{
{\mathsf C}^G \ar[r]^-{Z(-)} \ar[d]_-U &
{\mathsf C}^G \ar[d]^-U \\
\mathsf C \ar[r]_-{Z(-)} &
\mathsf C }
\end{equation}
in which $U$ denotes the forgetful functor and in which $U$ and 
$Z(-)$ in the bottom row are left adjoint functors. 
Then it follows by a dual version of the {\em Adjoint Lifting 
Theorem} \cite[Theorem~2]{Johnstone} that there is an adjunction 
$Z(-) \dashv \llbracket Z,- \rrbracket \colon {\mathsf C}^G \to 
{\mathsf C}^G$ whenever the equalizer  
$$
\xymatrix@C=30pt{
\llbracket Z,Y \rrbracket \ar[r]^-g &
G[Z,Y] \ar@<2pt>[r]^-{1[Z,y]} \ar@<-2pt>[r]_-{1\widetilde z} &
G[Z,GY]}
$$
in $\mathsf C$ exists, for any $G$-comodules $\xymatrix@C=18pt{Z \ar[r]|(.4)
{\, z\,} & GZ}$ and $\xymatrix@C=18pt{Y \ar[r]|(.4){\, y\,} & GY}$, where
$\widetilde z$ is the mate of $z$ under the adjunction $Z(-)\dashv [Z,-]\colon
\mathsf C \to \mathsf C$. In particular, ${\mathsf C}^G$ is closed whenever
equalizers of coreflexive pairs exist in $\mathsf C$. 

\begin{proposition} \label{prop:Y_in_c^G}
Consider a closed braided monoidal category $\mathsf C$ in which the
equalizers of coreflexive pairs exist. Let $G$ be a cocommutative Hopf monoid
in $\mathsf C$ such that the functor $G(-)\colon \mathsf C \to \mathsf C$
preserves monomorphisms. Then an object $\xymatrix@C=18pt{Z \ar[r]|(.4){\,
z\,} & GZ}$ of ${\mathsf C}^G$ belongs to the class $\mathcal Y$ in the closed
braided monoidal category ${\mathsf C}^G$ whenever $Z$ belongs to $\mathcal Y$
in $\mathsf C$. 
\end{proposition}

\begin{proof}
Property (a). By the assumption that $G(-)\colon \mathsf C \to \mathsf
C$ preserves monomorphisms, so does the forgetful functor $U\colon {\mathsf
C}^G \to \mathsf C$ and therefore the equal paths around the diagram of
\eqref{eq:lift}. Since $U$ is faithful it also reflects monomorphisms proving 
that the functor in the top row of \eqref{eq:lift} preserves monomorphisms. 

Property (b). Denote by $\delta\colon G\to G^2$, $\varepsilon\colon G \to I$
and $\mu\colon G^2\to G$ the comultiplication, the counit and the
multiplication of the Hopf monoid $G$, respectively, and for any $G$-comodules
$\xymatrix@C=18pt{X \ar[r]|(.4){\, x\,} & GX}$, 
$\xymatrix@C=18pt{Y \ar[r]|(.4){\, y\,} & GY}$ and 
$\xymatrix@C=18pt{Z \ar[r]|(.4){\, z\,} & GZ}$, denote by 
$a\colon \llbracket X,Y \rrbracket Z\to G\llbracket X,Y \rrbracket Z$ the
(diagonal) $G$-coaction. The left vertical of the commutative diagram 
$$
\xymatrix{
G[X,Y]Z \ar[d]_-{11z} \ar@{=}[r]&
G[X,Y]Z \ar[d]^-{\delta 1 z}&
\llbracket X,Y \rrbracket Z \ar[d]^-a \ar[l]_-{g1}
\ar[r]^-{\underline{\underline{\mathsf{coev}}}} 
&
\llbracket X,X \llbracket X,Y \rrbracket Z \rrbracket \ar[d]^-g
\ar[r]^-{\llbracket X, \underline{\underline{\mathsf{ev}}} 1\rrbracket} &
\llbracket X,YZ \rrbracket \ar[ddd]^-g \\
G[X,Y]GZ \ar[d]_-{1c1} &
G^2[X,Y]GZ \ar[l]^-{1\varepsilon 11} \ar[d]^-{1c_{G[X,Y],G}1} &
G\llbracket X,Y \rrbracket Z \ar[r]^-{1\mathsf{coev}} \ar[d]^-{1g1} &
G[X,X\llbracket X,Y \rrbracket Z] \ar[d]^-{1[X,1g1]} \\
G^2[X,Y]Z \ar[d]_-{\mu 11} &
G^3[Z,Y]Z \ar[r]^-{\mu 111} &
G^2[X,Y]Z \ar[d]^-{1\varepsilon 11} &
G[X,XG[X,Y]Z] \ar[d]^-{1[X,1\varepsilon 11]} \\
G[X,Y]Z \ar@{=}[rr] &&
G[X,Y]Z \ar[r]_-{1 \mathsf{coev}} &
G[X,X[X,Y]Z] \ar[r]_-{1[X,\mathsf{ev}1]} &
G[X,YZ]}
$$
is an isomorphism since $G$ is a Hopf monoid. Since the object $Z$ of $\mathsf
C$ belongs to $\mathcal Y$, $q\colon [X,Y]Z \to [X,YZ]$ is a monomorphism. The
bottom row is its image under the functor $G(-)$ hence it is a
monomorphism. The left pointing arrow of the top row is the image of the
(regular) monomorphism $g$ under the functor $G(-)$ hence it is a
monomorphism. Thus the left path around the diagram is a monomorphism, proving
that also the right pointing path of the top row is a monomorphism. 
\end{proof}

Since the category of vector spaces graded by a group $G$ is isomorphic to the
abelian category of comodules over the cocommutative Hopf algebra spanned by
$G$, from Proposition \ref{prop:Y_in_c^G} and Example \ref{ex:vec} we obtain
the following.   

\begin{example}\label{ex:G-gr}
In the closed symmetric monoidal category of group graded vector spaces every
object belongs to the class $\mathcal Y$.
\end{example}

\begin{example}\label{ex:dual}
In any closed braided monoidal category, an object which possesses a (left,
equivalently, right) dual, belongs to the class $\mathcal Y$. Indeed, if $V^*$
is the dual of some object $V$, then the functor $V^*(-)\cong [V,-]$ is right
adjoint, hence it preserves monomorphisms proving property (a). The canonical
morphism $q$ in part (b) is equal to the composite
$$
\xymatrix@C=8pt{
[X,Y]V^* \ar[r]^-{\raisebox{8pt}{${}_c$}} &
V^*[X,Y]\ar[r]^-{\raisebox{8pt}{${}_{\cong}$}} &
[V,[X,Y]] \ar[r]^-{\raisebox{8pt}{${}_{\cong}$}} &
[XV,Y] \ar[r]^-{\raisebox{8pt}{${}_{[c^{-1},1]}$}} &
[VX,Y] \ar[r]^-{\raisebox{8pt}{${}_{\cong}$}} &
[X,[V,Y]] \ar[r]^-{\raisebox{8pt}{${}_{\cong}$}} &
[X,V^*Y] \ar[r]^-{\raisebox{8pt}{${}_{[X,c^{-1}]}$}} &
[X,YV^*]}
$$
hence it is an isomorphism, for any objects $X,Y$.
\end{example}

\begin{example}\label{ex:born}
In the closed symmetric monoidal category of complete bornological
vector spaces \cite{Hogbe-Nlend,Meyer}, any object $Y$ obeying the
{\em approximation property} belongs to the class $\mathcal Y$. Indeed,
property (a) is asserted in Lemma 2.2 of \cite{Voigt}, whose Lemma 2.3
discusses a particular case of property (b) (when $Z$ is the base field). A
similar reasoning yields property (b) also for any complete bornological
vector spaces $X$ and $Z$; and $Y$ with the approximation property.
\end{example}

\begin{lemma} \label{lem:prop_P_is_monoidal}
The full subcategory of $\mathsf C$ with objects in $\mathcal Y$, is a
monoidal subcategory. That is, the following assertions hold. 

(1) The monoidal unit belongs to $\mathcal Y$.

(2) If both objects $Y$ and $Y'$ belong to $\mathcal Y$ then so does their
monoidal product $YY'$.
\end{lemma}

\begin{proof}
(1) follows since $I(-)$ is naturally isomorphic to the identity functor and
  for $Y=I$ the morphism $q$ is an isomorphism built from the right unit
  constraints.

(2) Property (a) holds since $YY'(-)$ is naturally isomorphic to the
  composite of the functors $Y'(-)$ and $Y(-)$. In order to see that property
  (b) holds, note first commutativity of the diagram
\begin{equation} \label{eq:q-coev}
\xymatrix{
ZY \ar[r]^-{\mathsf{coev}} \ar[d]_-{\mathsf{coev}1} &
[X,XZY] \ar[d]^-{[X,1\mathsf{coev}1]} \ar@{=}@/^1.3pc/[rd] \\
[X,XZ]Y \ar[r]^-{\mathsf{coev}} \ar@/_1.2pc/[rr]_-q &
[X,X[X,XZ]Y] \ar[r]^-{[X,\mathsf{ev} 1]} &
[X,XZY] }
\end{equation}
for any objects $X,Y,Z$. Using this together with the naturality of $q$ we
deduce the commutativity of 
$$
\xymatrix@C=40pt{
[X,Z]YY' \ar[r]_-{\mathsf{coev} 1} \ar@{=}[d] \ar@/^1.2pc/[rr]^-{q1} &
[X,X[X,Z]Y]Y' \ar[r]_-{[X,\mathsf{ev} 1]1} \ar[d]_-q &
[X,ZY]Y' \ar[d]^-{q} \\
[X,Z]YY' \ar[r]^-{\mathsf{coev}} \ar@/_1.2pc/[rr]_-{q} &
[X,X[X,Z]YY'] \ar[r]^-{[X,\mathsf{ev}11]} &
[X,ZYY'].}
$$
The top row is a monomorphism since $q$ is so and $(-)Y'\cong Y'(-)$
preserves monomorphisms. Since the right vertical is also mono, this 
proves that the bottom row is so.
\end{proof}

A morphism $\xymatrix@C20pt{W\ar[r]|(.48){\,j\,}&V}$ in a braided monoidal
category is said to be a {\em pure monomorphism} if 
$\xymatrix@C22pt{XW\ar[r]|(.5){\,1j\,}&XV}$ is a monomorphism for any 
object $X$; of course we see on taking $X=I$ that $j$ is a 
monomorphism. In particular, any split mono\-morphism is pure.

\begin{lemma} \label{lem:pure_mono}
Let $\xymatrix@C20pt{W\ar[r]|(.48){\,j\,}&V}$ be a pure monomorphism in
$\mathsf C$. If $V$ belongs to $\mathcal Y$ then so does $W$. 
\end{lemma}

\begin{proof} In order to check property (a) of $W$, consider a 
monomorphism $\xymatrix@C=25pt{X\ar[r]|(.5){\,f \, } & Y}$. Then in the
commutative diagram 
$$
\xymatrix{
WX \ar[r]^-{1f} \ar[d]_-{j1} &
WY \ar[d]^-{j1} \\
VX \ar[r]_-{1f} &
VY},
$$
the left-bottom path is a monomorphism by our assumptions. Then so is the
top-right path and therefore the top row.

As for property (b) of $W$, for any objects $X$ and $Z$ the left-bottom path
in the commutative diagram 
$$
\xymatrix{
[X,Z]W \ar[r]^-q \ar[d]_-{1j} &
[X,ZW] \ar[d]^-{[X,Zj]} \\
[X,Z]V \ar[r]_-q &
[X,ZV]}
$$
is a monomorphism. Then so is the top-right path and thus the top row.
\end{proof}

Lemma \ref{lem:pure_mono} tells us that, in particular, the class 
$\mathcal Y$ is closed under retracts.

\begin{proposition} \label{prop:non-deg}
For any morphism $v\colon ZV \to W$ in $\mathsf C$, the following assertions are
equivalent. 
\begin{itemize}
\item[{(i)}] For any object $X$, the map
$$
\mathsf C(X,V) \to C(ZX,W),\qquad f\mapsto 
\xymatrix@C=15pt{ZX \ar[r]^-{1f} &
ZV \ar[r]^-v &
W}
$$
is injective; that is, $v$ is non-degenerate on the left.
\item[{(ii)}] 
$\xymatrix@C=17pt{
V \ar[r]^-{\mathsf{coev}} &
[Z,ZV] \ar[r]^-{[Z,v]} &
[Z,W]}$
is a monomorphism.
\item[{(iii)}] For any object $X$, and any object $Y$ in $\mathcal Y$, the map 
$$
\mathsf C(X,VY) \to C(ZX,WY),\qquad g\mapsto 
\xymatrix@C=15pt{ZX \ar[r]^-{1g} &
ZVY \ar[r]^-{v1} &
WY}
$$
is injective; that is, $v$ is non-degenerate on the left with respect to
$\mathcal Y$. 
\item[{(iv)}] For any object $Y$ in $\mathcal Y$,
$\xymatrix@C=17pt{
VY \ar[r]^-{\mathsf{coev}} &
[Z,ZVY] \ar[r]^-{[Z,v1]} &
[Z,WY]}$
is a monomorphism.
\end{itemize}
\end{proposition}

\begin{proof}
Composing the map of part (iii) with the isomorphism $\mathsf C(ZX,WY)\cong
\mathsf C(X,[Z,WY])$, we obtain $\mathsf C(X,[Z,v1].\mathsf{coev})$. This
proves (iii)$\Leftrightarrow$(iv). Applying it to $Y=I$ proves
(i)$\Leftrightarrow$(ii). 

(iv)$\Rightarrow$(ii) follows by Lemma  \ref{lem:prop_P_is_monoidal}~(1)
putting $Y=I$. 

(ii)$\Rightarrow$(iv). The diagram
$$
\xymatrix{
VY \ar[r]^-{\mathsf{coev} 1} \ar@{=}[d]&
[Z,ZV]Y \ar[r]^-{[Z,v]1} \ar[d]^-q &
[Z,W]Y \ar[d]^-q \\
VY \ar[r]_-{\mathsf{coev}}  &
[Z,ZVY] \ar[r]_-{[Z,v1]} &
[Z,WY]}
$$
commutes by the naturality of $q$ and \eqref{eq:q-coev}. The top row is mono
by (ii) and the fact that $(-)Y\cong Y(-)$ preserves monos. Since the right
vertical is also mono, this proves that  the bottom row is so.
\end{proof}

\subsection{The base object of a regular weak multiplier bimonoid}
\label{sec:closed_L}

If the semigroup $A$ underlying a regular weak multiplier bimonoid in a 
closed braided monoidal category $\mathsf C$ admits a multiplier 
monoid $\mM(A)$, then the \mM-morphisms with components in \eqref{eq:pibarr}, 
\eqref{eq:pibarl}, \eqref{eq:pil} and \eqref{eq:pir} determine 
morphisms $\overline \sqcap^{R}$, $\overline \sqcap^{L}$, $\sqcap^{L}$ 
and $\sqcap^{R}$, respectively, all of them from $A$ to $\mM(A)$. 

Furthermore, the components $n_{1}$ and $n_{2}$ in 
\eqref{eq:n_1} of an \mM-morphism $L\nto A$ determine a 
morphism $n\colon L \to \mM(A)$. It obeys
$$
h_1.n1.p1=
n_{1}.p1=
\sqcap^{L}_{1}=
h_{1}.\sqcap^{L}1.
$$
Hence by the non-degeneracy of $h_{1}$ on the right (see
Proposition \ref{prop:dense}), $\sqcap^{L}$ factorizes through the regular
epimorphism $p$ via the morphism $n$. This morphism $n$ is monic if and only
if $n_1$ is non-degenerate on the right (see Corollary~\ref{cor:mono-nd}).  

Consider a regular weak multiplier bialgebra over a field (that is, a regular
weak multiplier bimonoid $A$ in the closed symmetric monoidal category of
vector spaces such that the morphisms $\hat d_1$ and $\hat d_2$ in
\eqref{eq:dhat} are surjective; i.e. regular epimorphisms). Under the
additional assumption that $A$ is left full, we saw in Remark
\ref{rem:base_vec} that $p\colon A\to L$ differs by an isomorphism from the
corestriction of $\sqcap^L$ to $A\to \mathsf{Im}(\sqcap^L)$. Since 
$n\colon L \to \mM(A)$ differs by the same isomorphism from the 
canonical inclusion $\mathsf{Im}(\sqcap^L) \to \mM(A)$, we conclude 
that in this case $n$ is a monomorphism, equivalently, $n_1$ is 
non-degenerate on the right (with respect to any vector space, see 
Proposition \ref{prop:non-deg} and Example \ref{ex:vec}). 

\subsection{Non-degeneracy in the category of Hilbert spaces}
\label{sec:non-deg_Hilb}

Although the category \Hilb of complex Hilbert spaces and continuous  
maps in Example \ref{ex:Hilb} is not closed, the findings of this 
section can be used to describe non-degenerate morphisms therein.

\begin{proposition}
Let $\nu \colon  Z \hat \otimes V \to W$ be a morphism in $\Hilb$, and 
write $i \colon  Z \otimes V \to Z \hat \otimes V$ for the canonical inclusion.   
The following conditions are equivalent:   

\begin{itemize}
\item[{(i)}] 
$\xymatrix@C=16pt{
Z \hat \otimes V \ar[r]|(.55){\,\nu\,} &
W}$
is non-degenerate in $\Hilb$. 
\item[{(ii)}] 
$\xymatrix@C=16pt{
Z \otimes V \ar[r]|(.48){\,i\,} &
Z \hat \otimes V \ar[r]|(.55){\,\nu\,} &
W}$
is non-degenerate in $\Vect$. 
\item[{(iii)}] 
$\xymatrix@C=16pt{
Z \otimes V \ar[r]|(.48){\,i\,} &
Z \hat \otimes V \ar[r]|(.55){\,\nu\,} &
W}$
is non-degenerate in $\Vect$ with respect to the class of all 
complex vector spaces.  
\item[{(iv)}]  
$\xymatrix@C=16pt{
Z \hat \otimes V \ar[r]|(.55){\,\nu\,} &
W}$
is non-degenerate in $\Hilb$ with respect to the class of all complex 
Hilbert spaces.  
\end{itemize}
\end{proposition}

\begin{proof}
(i)$\Rightarrow$(ii). By hypothesis, the map $\Hilb(X,V) \to 
\Hilb(Z\hat \otimes X,W)$,
\begin{equation}\label{ndhil}
f \mapsto 
\xymatrix{
Z  \hat \otimes X \ar^-{1 \hat \otimes f}[r] & 
Z \hat \otimes V \ar^-{\nu}[r] & 
W}  
\end{equation}
is injective. We should prove that the map $\Vect(X,V) \to \Vect(Z   
\otimes X,W)$,
\begin{equation}\label{ndvect}
f \mapsto 
\xymatrix{
Z  \otimes X \ar^-{1 \otimes f}[r] & 
Z  \otimes V \ar[r]^-i & 
Z\hat\otimes V \ar[r]^-\nu & 
W }  
\end{equation}
is injective for every complex vector space $X$. But, due to the fact 
that vector spaces are all direct sums of copies of the base field, 
and that the algebraic tensor product preserves direct sums, we see 
that it is enough to check the injectivity for $X = \mathbb{C}$. Now, 
$\mathbb{C}$ is certainly a Hilbert space, and every linear map $f 
\colon  \mathbb{C} \to V$ is continuous. Since $Z \hat \otimes 
\mathbb{C} = Z \otimes \mathbb{C}$, we see that the injectivity of 
\eqref{ndvect} for $X = \mathbb{C}$ follows from that of 
\eqref{ndhil}. Thus, $\nu.i$ is non-degenerate. 

(ii)$\Rightarrow$(iii). This holds by Proposition 
\ref{prop:non-deg} and Example \ref{ex:vec}. 

(iii)$\Rightarrow$(iv). By hypothesis, the map $\Vect(X,V 
\otimes Y) \to \Vect(Z  \otimes X ,W \otimes Y)$, 
$$
f \mapsto 
\xymatrix{
Z  \otimes X \ar[r]^-{1  \otimes f} & 
Z  \otimes V \otimes Y \ar[r]^-{i \otimes 1} &
(Z\hat\otimes V)\otimes Y \ar[r]^-{\nu\otimes 1} & 
W \otimes Y} 
$$
is injective for all complex vector spaces $X, Y$, and we should 
prove that also the map
$\Hilb(X,V \hat \otimes Y) \to \Hilb(Z \hat \otimes X ,W \hat \otimes Y)$,
\begin{equation}\label{Ndhil}
f \mapsto 
\xymatrix{
Z  \hat \otimes X \ar[r]^-{1  \hat \otimes f} & 
Z  \hat \otimes V \hat  \otimes Y \ar[r]^-{\nu \hat \otimes 1} & 
W \hat \otimes Y}
\end{equation}
is injective for all Hilbert spaces $X,Y$. But the map 
$\Hilb(Z\hat\otimes X,W\hat\otimes Y)\to \Vect(Z\otimes 
X,W\hat\otimes Y)$ is injective, so it will suffice to show that the 
composite of \eqref{Ndhil} with this last map is injective; and for 
that, it clearly suffices to prove the case where 
$X=\mathbb{C}$.      

Thus we need to prove that the induced map 
$
V\hat\otimes Y\to \Vect(Z,W\hat\otimes Y)
$
is injective, or equivalently that for $z\in Z$ the maps
$$\xymatrix
{
V\hat\otimes Y \ar[r]^-{z\hat\otimes-} & 
Z\hat\otimes V\hat\otimes Y \ar[r]^-{\nu\hat\otimes 1} & 
W\hat\otimes Y }
$$
are jointly injective. These are continuous linear maps, so it will 
suffice to show that the maps   
$$
\xymatrix{
V\otimes Y \ar[d]_{i} \ar[r]^-{z\otimes -} & 
Z\otimes V\otimes Y \ar[r]^-{i\otimes 1} \ar[d]^-{1\otimes i} & 
Z \hat \otimes V\otimes Y \ar[r]^-{\nu\otimes 1} \ar[d]^-i & 
W\otimes Y \ar[d]^-{i} \\
V\hat\otimes Y \ar[r]_-{z\otimes-} &
Z\otimes (V\hat\otimes Y) \ar[r]_-i &
Z\hat\otimes V\hat\otimes Y \ar[r]_-{\nu\hat\otimes 1} & 
W\hat\otimes Y }
$$
are jointly injective, or equivalently that the upper horizontals are 
jointly injective.  
But this follows from (iii). 

(iv)$\Rightarrow$(i). Put $Y = \mathbb{C}$. 
\end{proof}


\appendix

\section{Some identities and their string diagrammatic proofs}
\label{app:strings}

Throughout this Appendix, $A$ will be an object in a braided monoidal category
$\mathsf C$, and $t_1,t_2,t_3,t_4\colon A^2\to A^2$, $e_1,e_2\colon 
A^2 \to A^2$ and $j\colon A\to I$ are morphisms making $A$ a regular 
weak multiplier bimonoid in $\mathsf C$. The multiplication
$$

$$
is assumed to be non-degenerate with respect to some class of 
objects containing $I$, $A$ and $A^{2}$. The following notation is used.
\begin{eqnarray*}
&&%
\end{eqnarray*}

\begin{subsection}{}\label{par:A1}
For any morphisms
$\xymatrix@C=25pt{XA \ar[r]|(.5){\,f_1\,} & A & AX
\ar[l]|(.5){\,f_2\,}}$ 
satisfying {\eqref{eq:components_compatibility}}, the following
identities hold. 
\begin{itemize}
\item[{(i)}]
\raisebox{-45pt}{$%
$}
\end{itemize}
These identities hold in particular if both $f_1$ and $f_2$ are equal to the
multiplication of $A$.

\begin{proof}
Note that the identities of part (i) hold for any morphisms satisfying the
conditions of Lemma \ref{lem:compatibility}; and the identities of part (ii)
hold for any morphisms obeying \eqref{eq:e_components}.  
We prove one identity in each group; all the others follow symmetrically. 

(i) The first identity follows by non-degeneracy of the multiplication and
$$
%
\ .
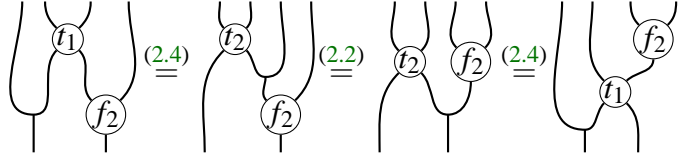
$$

(ii) Similarly, the first identity follows by non-degeneracy of the
multiplication and 
$$
%
$}

\begin{proof}
The stated equality is obtained composing by a suitable braiding isomorphism
the equal expressions in
$$
%
$}

\begin{proof}
Note that the first equality of the claim holds, in fact, for any weakly
counital fusion morphism $(t_1,e_1,j)$, by the following reasoning.
$$
%
 
$$
By the first equality of the claim and by its coopposite
$$
%
$$
proving the second equality of the claim.
\end{proof}
\end{subsection}

\begin{subsection}{}\label{par:A5}
\raisebox{-50pt}{$%
$}
\begin{proof}
Note that this holds, in fact, for any weakly counital fusion morphism
$(t_1,e_1,j)$ and a morphism $e_2$ obeying \eqref{eq:e_components}.
It follows by the non-degeneracy of the multiplication and
$$
%
$}
\begin{proof}
The first identity of the claim follows from Axiom~\hyperlink{AxV}{\textnormal{V}}  on  post-composing 
with $j11$, and the second follows from \eqref{eq:interchange_12} on 
post-composing with $j11$. The third identity follows by the 
non-degeneracy of the multiplication and the calculation
$$
%
 
$$
in which the unlabelled equality follows from the opposite-coopposite 
of the second identity of the claim. 

The last identity of the claim follows by the calculation  
$$
%
$$
and the non-degeneracy of the multiplication. 
\end{proof}
\end{subsection}

\begin{subsection}{}\label{par:A6.5}
\raisebox{-35pt}{$%
$}

\begin{proof}
Note that the claim holds for any weakly counital fusion morphism
$(t_1,e_1,j)$ and a morphism $e_2$ obeying \eqref{eq:e_components}.
It follows by the calculation 

$$
%
$$
and the non-degeneracy of the multiplication. 
\end{proof}
\end{subsection}

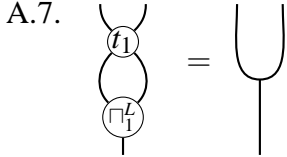
\begin{subsection}{}\label{par:A8}
\raisebox{-50pt}{$%
$}

\begin{proof}
This follows by Axiom \hyperlink{AxIII}{\textnormal{III}} and
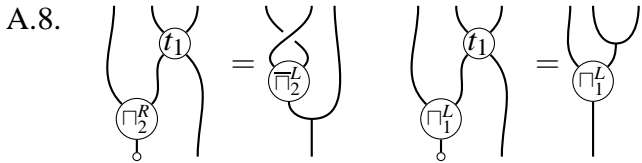
\eqref{eq:multiplication_1}. 
\end{proof}
\end{subsection}

\begin{subsection}{}\label{par:A9}
\raisebox{-50pt}{$%
$}

\begin{proof}
The first identity of the claim follows by the non-degeneracy of the
multiplication and 
$$
%
$$
where the unlabelled equality holds by the opposite of Axiom
\hyperlink{AxV}{\textnormal{V}}. The second identity of the 
claim follows by 
$$ 
%
\ .
$$
\end{proof}
\end{subsection}

\begin{subsection}{}\label{par:d1_1}
For the morphisms 
$d_1:=
\raisebox{-20pt}{$%
$}$,
the following hold.

$%
$

\noindent
The second equality says that $d_1$ and $d_2$ are the components of 
an \mM-morphism $d\colon A\nto A^2$, which can be regarded as a 
generalized (multiplier-valued) comultiplication; the third says that 
this \mM-morphism is multiplicative in the sense of 
\cite{BohmLack:cat_of_mbm}.   

\begin{proof}
The first identity follows by 
$$
%
 \ .
$$
Using this together with the associativity of the 
multiplication we obtain 
$$
%
$$
that is, the second identity of the claim. The third identity of the claim
follows by 
$$
%
\ .
$$
The penultimate identity of the claim is immediate by Axiom
\hyperlink{AxIV}{\textnormal{IV}}  and the last one follows by the first
identity in part (ii) of \ref{par:A1} and  Axiom~\hyperlink{AxIII}{\textnormal{III}}. 
\end{proof}
\end{subsection}

\begin{subsection}{}\label{par:A10}
\raisebox{-50pt}{$
%
$}

\begin{proof}
The first identity of the claim follows by the non-degeneracy of the
multiplication and 
$$
%
$$
in which the unlabelled equality follows from the  
coopposite of the third displayed identity of \ref{par:A7} on 
post-composing  with $1j$.  The  second 
identity is equivalent, via non-degeneracy of the 
multiplication and repeated use of 
{\eqref{eq:components_compatibility}}, to   
$$
%
$$
which in turn follows from the last equality of \ref{par:A7}.

The last identity is equivalent, via repeated use of {
  \eqref{eq:components_compatibility}}, to  
$$%
$$
which in turn is a straightforward consequence of \eqref{eq:interchange_14}.
\end{proof}
\end{subsection}

\begin{subsection}{}\label{par:A11}
\raisebox{-50pt}{$
%
\ . $}

\begin{proof}
This follows by post-composing with $j1$ from the 
calculation  
$$
%
 
$$
and non-degeneracy of the multiplication.
\end{proof}
\end{subsection}

\begin{subsection}{}\label{par:A12}
\raisebox{-50pt}{$%
$}
\begin{proof}
The first identity of the claim follows by 
$$
%
$$
and the second is similar, using \eqref{eq:multiplication_4} in place of
\eqref{eq:multiplication_1}. 
\end{proof}
\end{subsection}

\begin{subsection}{}\label{par:A14}
\raisebox{-50pt}{$
%
$}

\begin{proof}
The first identity follows by the non-degeneracy of the multiplication and the
equality of the second and last expressions in the calculation  
$$
%
$$
while the second identity follows from the associativity and non-degeneracy of the multiplication, and the equality of the first and penultimate expressions. 
\end{proof}
\end{subsection}

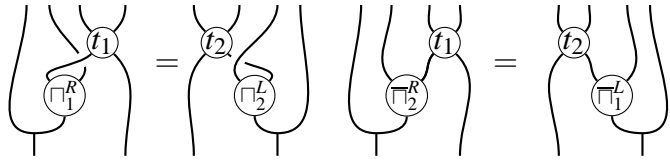
\begin{subsection}{}\label{par:A15}
\raisebox{-50pt}{$%
$}

\begin{proof}
The first assertion holds by 
$$
%
$$
and the second follows similarly on applying 
\eqref{eq:c12} twice. 
\end{proof}
\end{subsection}

\begin{subsection}{}\label{par:A16}
\raisebox{-50pt}{$%
$}

\begin{proof}
This follows from the calculation 
$$
%
$$
and non-degeneracy and associativity of the multiplication.
\end{proof}
\end{subsection}

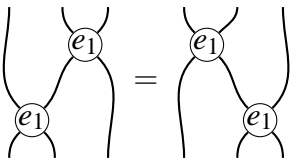
\begin{subsection}{}\label{par:A18}
\raisebox{-50pt}{$%
$}
\bigskip

\noindent
Recall that in the definitions of the papers 
\cite{VDaWa,VDaWa:Banach,BohmGomezTorrecillasLopezCentella:wmba} this 
is imposed as an axiom. 

\begin{proof}
This follows by the non-degeneracy of the multiplication and 
$$
%
$$
where the equalities labelled \ref{par:A4} also use the coopposite of 
\ref{par:A4}. 
\end{proof}
\end{subsection}

\begin{subsection}{}\label{par:old(A5)}
\raisebox{-50pt}{$
%
$}

\begin{proof}
By the non-degeneracy of the multiplication with respect to $A$ and 
\eqref{eq:components_compatibility}, the claim is equivalent to   
$$
%
$}
\begin{proof}
The second equality of the claim is immediate by the coopposite of the first
identity of \ref{par:A4}. The first equality of the claim follows by 
$$
%
$}

\begin{proof}
This follows by pre-composing with a suitable braid isomorphism 
the equal outermost expressions of 
$$
%
$}

\begin{proof}
From the coopposites of the identities of \ref{par:A12} and 
multiple uses of {\eqref{eq:components_compatibility}} we obtain
$$
%
\ .
$$
With their help both equalities of the claim become equivalent to the 
equality  of the outermost expressions in 
$$
%
$}
\begin{proof}
This follows by 
$$
%
$$
where the unlabelled equality is obtained using an equivalent form of 
the first identity in \ref{par:A20}. 
\end{proof} 
\end{subsection}

\begin{subsection}{}\label{par:A22}
\raisebox{-50pt}{$%
$}
\begin{proof}
This follows by 
$$
%
 
$$
where the equalities labelled by \ref{par:A4} use this in its 
coopposite form. 
\end{proof}
\end{subsection}

\begin{subsection}{}\label{par:A22.5}
\raisebox{-63pt}{$%
$}
\begin{proof}
This follows by 
$$
%
 \ .
$$
\end{proof}
\end{subsection}


\bibliographystyle{plain}

\begin{thebibliography}{10}

\bibitem{Bohm:wmba_comod}
G. B\"ohm,
{\em Comodules over weak multiplier bialgebras,}
Int. J. Math. 25 (2014), 1450037. 

\bibitem{BoCaJa:wba}
G. B\"ohm, S. Caenepeel and K. Janssen,
{\em Weak bialgebras and monoidal categories,}
Comm. Algebra 39 (2011), no. 12 (special volume dedicated to Mia 
Cohen), 4584-4607.   

\bibitem{BohmGom-Tor:FirmFrob}
G. B\"ohm and J. G\'omez-Torrecillas,
{\em Firm Frobenius monads and firm Frobenius algebras,}
Bull. Math. Soc. Sci. Math. Roumanie 56(104) (2013), no. 3, 281-298. 

\bibitem{BohmGomezTorrecillasLopezCentella:wmba}
G. B\"ohm, J. G\'omez-Torrecillas and E. L\'opez-Centella, 
{\em Weak multiplier bialgebras,} 
Trans. Amer. Math. Soc. 367 (2015), 8681-8721.

\bibitem{BohmLack:braided_mba}
G. B\"ohm and S. Lack,
{\em Multiplier bialgebras in braided monoidal categories,}
Journal of Algebra 423 (2015), 853-889.

\bibitem{BohmLack:cat_of_mbm}
G. B\"ohm and S. Lack,
{\em A category of multiplier bimonoids,}
Applied Categorical Structures, to appear; available as 
\href{http://arxiv.org/abs/1509.07171}{http://arxiv.org/abs/1509.07171}.

\bibitem{BohmLack:simplicial}
G. B\"ohm and S. Lack,
{\em A simplicial approach to multiplier bimonoids,}
Preprint available at 
\href{http://arxiv.org/abs/1512.01259}{http://arxiv.org/abs/1512.01259}.

\bibitem{BohmLack:multiplier_Hopf_monoid}
G. B\"ohm and S. Lack,
{\em Multiplier Hopf monoids,}
Preprint available at 
\href{http://arxiv.org/abs/1511.03806}{http://arxiv.org/abs/1511.03806}.

\bibitem{WHAI}
G. B\"ohm, F. Nill and K. Szlach\'anyi, 
{\em Weak Hopf algebras. I. Integral theory and C*-structure,}
J. Algebra 221 (1999), no. 2, 385-438. 

\bibitem{WHAII}
G. B\"ohm and K. Szlach\'anyi, 
{\em Weak Hopf algebras. II. Representation theory, dimensions, and the 
Markov trace,}  
J. Algebra 233 (2000), 156-212.  

\bibitem{BohmVerc:Morita}
G. B\"ohm and J. Vercruysse,
{\em Morita theory for comodules over corings,}
Comm. Algebra 37 (2009), no. 9, 3207-3247. 

\bibitem{BrKaWi:cosep_coalg}
T. Brzezi\'nski, L. Kadison and R. Wisbauer,
{\em On coseparable and biseparable corings,}
in: ``Hopf Algebras in Noncommutative Geometry and Physics'', 
S. Caenepeel and F. Van Oystaeyen (eds.),
Monographs on pure and applied mathematics vol. 239 pp 71-88,
Marcel Dekker New York, 2004.

\bibitem{Dauns:Multiplier}
J. Dauns, 
{\em Multiplier rings and primitive ideals,}
Trans. Amer. Math. Soc. 145 (1969), 125-158. 

\bibitem{Etingof/Nikshych/Ostrik:2005}
P. Etingof, D. Nikshych, and V. Ostrik, {\em On fusion categories},
 Annals of Mathematics. Second Series 162 (2005): 581--642.

\bibitem{Hogbe-Nlend}
H. Hogbe-Nlend,
{\em Bornologies and functional analysis.} 
Translated from the French by V. B. Moscatelli. 
North-Holland Mathematics Studies, Vol. 26. Notas de Matem\'atica, No. 62.  
North-Holland Publishing Co., Amsterdam-New York-Oxford, 1977. 

\bibitem{Johnstone}
P.T. Johnstone,
{\em Adjoint Lifting Theorems for Categories of Algebras,}
Bull. London Math. Soc. 7 (1975) no. 3, 294-297.

\bibitem{Kadison/Ringrose:1983}
R.V. Kadison and J.R. Ringrose, 
{\em Fundamentals of the theory of operator algebras, vol. I.}
Academic Press, New York, 1983.

\bibitem{Meyer}
R. Meyer,
{\em Local and Analytic Cyclic Homology.}
EMS Tracts in Mathematics Vol. 3, 2007.

\bibitem{Nill}
F. Nill,
{\em Axioms for weak bialgebras,}
preprint available at
\href{http://arxiv.org/abs/math/9805104}{http://arxiv.org/abs/math/9805104}.

\bibitem{PastroStreet}
C. Pastro and R. Street,
{\em Weak Hopf monoids in braided monoidal categories,}
Algebra and Number Theory 3 (2009) no. 2, 149-207.

\bibitem{Quillen:firm}
D. Quillen, 
{\em Module theory over nonunital rings,}
Unpublished Notes, 1997.

\bibitem{VanDaele:multiplier_Hopf}
A. Van Daele,
{\em Multiplier Hopf algebras,}
Trans. Amer. Math. Soc. 342 (1994), 917-932.

\bibitem{VDae:Sep}
A. Van Daele,
{\em Separability Idempotents and Multiplier Algebras,}
preprint available at
\href{http://arxiv.org/abs/1301.4398}{http://arxiv.org/abs/1301.4398}.

\bibitem{VDaWa}
A. Van Daele and S. Wang,
{\em Weak Multiplier Hopf Algebras. The main theory,}
Journal f\"ur die reine und angewandte Mathematik (Crelles Journal) 
705 (2015), 155-209. 

\bibitem{VDaWa:Banach}
A. Van Daele and S. Wang,
{\em Weak Multiplier Hopf Algebras. Preliminaries, motivation and basic
  examples,}
in: Operator Algebras and Quantum Groups, W. Pusz and P.M. So\l tan (eds.), 
Banach Center Publ. 98 (2012), 367-415.

\bibitem{Voigt}
C. Voigt, 
{\em Bornological quantum groups,}
Pacific J. Math. 235 (2008) no. 1, 93-135. 

\bibitem{ZimHui}
B. Zimmermann-Huisgen,
{\em Pure submodules of direct products of free modules,}
Math. Annalen 224 (1976) no. 2, 233-245.

{}
\end{thebibliography}

\end{document}